\documentclass[11pt]
{article}

\title{Algebraic entropy for amenable semigroup actions}
\date{}

\usepackage{amssymb,amsmath,tikz,amsthm}
\usepackage{enumerate,hyperref}

\usepackage{vmargin}
\setmarginsrb{5mm}{5mm}{5mm}{10mm}%
          {5mm}{5mm}{5mm}{7mm}
\usepackage{xspace}

\newtheorem{lemma}{Lemma}[section]
\newtheorem{proposition}[lemma]{Proposition}
\newtheorem{theorem}[lemma]{Theorem}
\newtheorem{corollary}[lemma]{Corollary}
\newtheorem{conjecture}[lemma]{Conjecture}

\newtheorem{claim}[lemma]{Claim}
\newtheorem{question}[lemma]{Question}

\theoremstyle{definition}
\newtheorem{definition}[lemma]{Definition}
\newtheorem{remark}[lemma]{Remark}
\newtheorem{example}[lemma]{Example}
\newtheorem{examples}[lemma]{Examples}

\newcommand*{\norm}[1]{\rVert #1 \rVert}

\def\N{\mathbb N}
\def\R{\mathbb R}
\def\Z{\mathbb Z}
\def\Q{\mathbb Q}

\def\F{\mathbb F}

\def\P{\mathcal P}
\def\Pf{\P_{fin}}
\def\ent{\mathrm{ent}}
\def\Aut{\mathrm{Aut}}

\def\Se{\mathfrak S}
\newcommand{\halg}{h_{alg}}

\newcommand{\Halg}{H_{alg}}

\DeclareMathOperator{\cardm}{card}

\newcommand*{\card}[1]{\left\vert #1 \right\vert}
\newcommand*{\abs}[1]{\left\lvert #1 \right\rvert}

\newcommand*{\set}[1]{\left\{ #1 \right\}}

\newcommand*{\Pa}[1]{\bigl(#1\bigr)}

\newcommand{\SDiff}{\mathbin{\Delta}}

\newcommand{\rest}{\mathbin\restriction}

\newcommand{\Aver}{\mathcal H}

\newcommand{\FolSeq}{\mathfrak F}
\newcommand{\Folner}{F\o lner\xspace} 
\newcommand{\eqeps}{\mathrel{=_{\eps}}}
\newcommand{\leqeps}{\mathrel{\leq_{\eps}}}

\newcommand{\eps}{\varepsilon}
\newcommand{\beps}{{\bar\varepsilon}}

\numberwithin{equation}{section}

\author{Dikran Dikranjan \and Antongiulio Fornasiero \and Anna Giordano Bruno}
\newlength{\bibitemsep}
\newlength{\bibparskip}
\let\oldthebibliography\thebibliography
\renewcommand\thebibliography[1]{%
  \oldthebibliography{#1}%
  \setlength{\parskip}{\bibitemsep}%
  \setlength{\itemsep}{\bibparskip}%
}

\begin{document}

\maketitle

\abstract{
We introduce two notions of algebraic entropy for actions of cancellative right amenable semigroups $S$ on discrete abelian groups $A$ by endomorphisms; these extend the classical algebraic entropy for endomorphisms of abelian groups, corresponding to the case $S=\mathbb N$. 
We investigate the fundamental properties of the algebraic entropy and compute it in several examples, paying special attention to the case when $S$ is an amenable group.

For actions of cancellative right amenable monoids on torsion abelian groups, we prove the so called Addition Theorem. 
In the same setting, we see that a Bridge Theorem connects the algebraic entropy with the topological entropy of the dual action 
by means of the Pontryagin duality, so that we derive an Addition Theorem for the topological entropy of actions of cancellative left amenable monoids on totally disconnected compact abelian groups.}
\\

2010 Mathematics Subject Classification: {\sl Primary 20K30, 20M20; Secondary 37A35, 37B40, 43A07.}

Keywords: {\sl algebraic entropy, topological entropy, amenable semigroup, amenable monoid, group endomorphism, semigroup action.}


\section{Introduction}

The notion of entropy was largely studied for discrete dynamical systems since the mid fifties, when 
Kolmogorov and Sinai defined the measure entropy in ergodic theory. Inspired by their work, Adler, Konheim, and McAndrew \cite{AKM} introduced the topological entropy for continuous selfmaps of compact topological spaces, while a different notion of topological entropy for uniformly continuous selfmaps of metric spaces was given by Bowen \cite{B} and Dinaburg \cite{Din} independently.

Yuzvinski \cite{Y} computed the topological entropy of continuous endomorphisms of compact metrizable groups, including the celebrated formula (carrying now his name, that is, Yuzvinski Formula) establishing that the topological entropy of a continuous endomorphism of a finite-dimensional universal solenoid coincides with the Mahler measure of their characteristic polynomial over $\Z$. 
Yuzvinski proved also the so-called Addition Theorem (usually called Yuzvinski's addition formula) for the topological entropy of continuous endomorphisms of compact metrizable groups, that was recently extended to all compact groups in \cite{Dik+Manolo}.  Later on, Stoyanov \cite{St} established uniqueness of the topological entropy of continuous endomorphisms of compact groups, imposing several natural axioms, in the so-called Uniqueness Theorem. As a by-product, this entails the coincidence of the topological and the measure entropy in the category of compact groups and surjective continuous endomorphisms (see \cite{B} for the metrizable case, \cite{Aoki} for the abelian case).

After a very brief and schematic introduction in the very end of \cite{AKM}, the algebraic entropy for endomorphisms of abelian groups was gradually developed by Weiss \cite{W} and Peters  \cite{P1,P2}. The interest in this direction increased after \cite{DGSZ}, where a rather complete description in the case of torsion abelian groups was obtained, including an Addition Theorem and a Uniqueness Theorem. These were generalized to all abelian groups in \cite{DGB}. Details and results can be found in \cite{AADGBH,AD,DGB,DGB4,GBshift,GBV,GBV2}, in \cite{GBSp,GBSp1} for the non-abelian case, in \cite{SVV,SZ1} for the algebraic entropy for modules; see also the surveys \cite{DGB3,DGB0,DSV,GS1,GS2}.
 
\medskip
As far as non-discrete dynamical systems are concerned, the measure entropy for actions of finitely generated groups on probability spaces by measure preserving transformations was defined by Kirillov \cite{Ki}; the case of abelian group actions was studied by Conze \cite{Con}, and by Katznelson and Weiss \cite{KW}. 
Lind, Schmidt, and Ward \cite{LSW} gave reference to Conze for the measure entropy of $\Z^d$-actions, while for $\Z^d$-actions on compact metrizable groups they  generalized to this setting both the definition of topological entropy by Bowen, as well as that by Adler, Konheim, and McAndrew, showing that they coincide.
They proved the Addition Theorem for $\Z^d$-actions on compact metrizable groups, and analogues of the Yuzvinski Formula (involving multidimensional Mahler measure).
 
The measure entropy for amenable group actions was introduced by Kieffer \cite{Kie}, while the topological entropy for amenable group actions on compact metric spaces by Stepin and Tagi-Zade \cite{ST}, and Ollagnier \cite{Oll} defined the topological entropy for amenable group actions on compact spaces using open covers as in \cite{AKM}.
A cornerstone in the theory of entropy of amenable group actions is the work by Ornstein and Weiss \cite{OW}, where in particular they proved the celebrated Ornstein-Weiss Lemma.
Apparently, not much was done for actions of genuine non-abelian amenable groups until Deninger's paper \cite{De}, followed by Chung and Thom \cite{CT}, and Li \cite{Li}. These authors established appropriate analogues of the Yuzvinski Formula in terms of the Fuglede-Kadison determinant in place of the Mahler measure, and analogues of the Addition Theorem. In particular, Li \cite{Li} proved the Addition Theorem for actions of countable amenable groups on compact metrizable abelian groups.

The algebraic and the topological entropy were extended by Virili \cite{V2} to actions on locally compact abelian groups. Another paper by Virili \cite{V3} concerns the algebraic entropy of amenable group actions on modules; there he proved also an Addition Theorem and finds applications to the Stable Finiteness Conjecture and the Zero Divisors Conjecture, originally stated by Kaplansky.
These ideas were pushed further by Li and Liang \cite{LL}. Various extensions of these entropies to the case of actions of sofic groups 
can be found, for example, in \cite{Bo0,Bo,KH}, and the survey \cite{Weisss}.



\medskip
Recently, Ceccherini-Silberstein, Coornaert, and Krieger \cite{CCK} extended Ornstein-Weiss Lemma to cancellative amenable semigroups (see Theorem~\ref{CCKLemma} below). 
Using this result, they introduced the measure entropy and the topological entropy for actions of cancellative amenable semigroups. In particular, a notion of topological entropy $h_{top}$ was defined for left actions of cancellative left amenable semigroups on compact topological spaces (see \S\ref{htop-sec} below) extending the one in \cite{AKM}.

Following this approach, we consider left actions $$S \overset{\alpha}\curvearrowright A$$ of cancellative right amenable semigroups $S$ on abelian groups $A$ by group endomorphisms  (i.e., $\alpha(s):A\to A$ is a group endomorphism for every $s\in S$). 
For such actions we define and investigate two variants of the algebraic entropy, denoted by $\ent$ and $\halg$, that coincide when $A$ is torsion. In case $S=\N$, these notions of algebraic entropy coincide with those mentioned above for discrete dynamical systems.
Details and basic properties of the algebraic entropies can be found in \S\ref{halgsec}.

Before that, in \S\ref{background} we provide the necessary background on amenable semigroups and their F\o lner nets, in particular we introduce the canonical F\o lner nets. These tools are used in \S\ref{sec:integral} to build a kind of integration theory for a class of real valued functions defined on the finite subsets of an amenable cancellative monoid. We show that this integral satisfies an appropriate version of Fubini's Theorem (see Theorem~\ref{Fubini}). 

This theory allows us to introduce the algebraic entropy in \S\ref{halgsec} as an integral of a suitable function. Moreover, our counterpart of Fubini's Theorem applies several times in \S\ref{restr:quot:sec}, where we compute the algebraic entropy of restriction and quotient actions.
More precisely, for a left action $G \overset{\alpha}\curvearrowright A$ of a cacellative right amenable monoid $G$ on an abelian group $A$, denote by $$N\overset{ \alpha\rest_N}{\curvearrowright} A$$ the restriction action of $\alpha$ with respect to a submonoid $N$ of $G$.
If $G$ is a an amenable group and $N$ is a normal subgroup of $G$ trivially acting on $A$, then the quotient $G/N$ of $G$ acts on $A$ and we denote by 
$$
G/N \overset{\overline{\alpha}_{G/N}}\curvearrowright A
$$
this quotient action of $\alpha$. 
We show that, in case $N$ is a normal subgroup of the amenable group $G$, the algebraic entropy of $\alpha\restriction_N$ is always a greater than the algebraic entropy of $\alpha$, and the same applies to $\overline{\alpha}_{G/N}$ when $N$ acts trivially on $A$ (see Theorems~\ref{teo:submonoid} and~\ref{quotient}).

Several corollaries of the theorems on restriction and quotient actions are given, showing that the algebraic entropy vanishes very often. 
In particular, many actions of $\N^d$ and $\Z^d$, for $d>1$ (see Corollary~\ref{Coro:June12}  and Corollary~\ref{Coro:application1}), as well as the natural action of various amenable subgroups of $\mathrm{GL}_n(K)$ on $K^n$ for an infinite field $K$ (see Corollary~\ref{field}) have zero algebraic entropy. 
The counterparts of this frequent vanishing of the algebraic entropy seem to be known for the topological entropy and the measure entropy
 (see \cite{HSt}). They represent a motivation to study an alternative kind of entropy,  
called measure receptive entropy and topological receptive entropy in \cite{BDGBS} (see also \cite{BDGBS-Car}), for actions of finitely generated, not necessarily amenable monoids.   
%

\medskip
We dedicate the whole \S\ref{AT:sec} to the following Addition Theorem for actions on torsion abelian groups $A$ with respect to an invariant subgroup $B$:

\begin{theorem}[Addition Theorem]\label{ATintro}
Let $S \overset{\alpha}\curvearrowright A$ be a left action of a cancellative right amenable monoid $S$ on a torsion abelian group $A$.
Let $B$ be an $\alpha$-invariant subgroup of $A$, and denote by $\alpha_B$ and $\alpha_{A/B}$ the induced actions of $S$ on $B$ and on $A/B$, respectively. Then
\[\ent(\alpha) = \ent(\alpha_B)  +\ent (\alpha_{A/B}).\]
\end{theorem}

As recalled above, this theorem was proved for $\N$-actions in \cite{DGSZ}, and its counterpart for $\N$-actions on arbitrary abelian groups in \cite{DGB0}. In both cases the proof was quite long, making  heavy use of the algebraic structure of $A$. Our proof avoids any recourse to the structure of $A$. 
In \cite{DFG-tiling} we consider a particular case of the Addition Theorem for monoids $S$ admitting tiling F\o lner sequences, which largely covers the case of $\N^d$-actions (and in particular, the case of $\N$-actions from \cite{DGSZ}), with a notably simplified proof.

The action $\alpha$ of a cancellative right amenable monoid $S$ on an abelian group $A$ provides a left $\Z[S]$-module structure on $A$. By Proposition~\ref{conju}, we have that $h_{alg}$ is an invariant of the category $\mathrm{Mod}_{\Z[S]}$ of left $\Z[S]$-modules, and by Proposition~\ref{contlim} $h_{alg}$ is upper-continuous, that is, continuous with respect to direct limits.
If we restrict to the subcategory $\mathfrak T$ of $\mathrm{Mod}_{\Z[S]}$ consisting of left $\Z[S]$-modules that are torsion as abelian groups, analogously we have that $\ent$ is an uppercontinuous invariant of $\mathfrak T$. Moreover, Theorem~\ref{ATintro} shows that $\ent$ is also additive, and so it is a length function of $\mathfrak T$ in the sense of Northcott and Reufel \cite{NR} and V\'amos \cite{Vamos}.

\medskip
In the case of discrete dynamical systems there is a remarkable connection, usually named Bridge Theorem, between the topological entropy and the algebraic entropy discovered by Weiss \cite{W} and Peters \cite{P1}, and proved in general in  \cite{DGB1}. More precisely, the topological entropy of a continuous endomorphism $\phi$ of a compact abelian group $K$ coincides with the algebraic entropy of its dual endomorphism $\widehat \phi$ of the Pontryagin dual $\widehat K$ of $K$, which is a discrete abelian group. 
This connection was extended to totally disconnected locally compact abelian groups in \cite{DGB2}.

The Bridge Theorem from \cite{P1} was recently extended by Virili \cite{V2} to the case of actions of amenable groups on locally compact abelian groups, while the one from \cite{DGB2} was extended in \cite{GB} to semigroup actions on totally disconnected locally compact abelian groups.
In \S\ref{BT-sec}, generalizing the main result of \cite{W}, we prove a Bridge Theorem for left actions of cancellative left amenable monoids on totally disconnected compact abelian groups (their Pontryagin dual groups are precisely the torsion abelian groups).

\begin{theorem}[Bridge Theorem]\label{BTintro}
Let $S \overset{\gamma}\curvearrowright K$ be a left action of a cancellative left amenable monoid on a totally disconnected compact abelian group $K$, which induces a dual right action $\widehat\gamma$ of $S$ on the Pontryagin dual of $A$.
Then the topological entropy of $\gamma$ coincides with the algebraic entropy of $\widehat \gamma$.
\end{theorem}

Combining Theorem~\ref{ATintro} and Theorem~\ref{BTintro}, one obtains an Addition Theorem for the topological entropy of left actions of cancellative left amenable monoids on totally disconnected compact abelian groups.


\bigskip
What we said above about the vanishing of the algebraic entropy may leave the wrong feeling that there are no significant instances of cancellative right amenable semigroups $S$, distinct from $\N$ and $\Z$, acting with positive algebraic entropy on discrete abelian groups. This is not the case; indeed, in a preliminary version of this manuscript we dedicated special attention to the computation of the algebraic entropy of the shifts, that is, standard actions 
$S \overset{\alpha_S}\curvearrowright A^{(S)}$ of a cancellative right amenable semigroup $S$ on the direct sum $A^{(S)}=\bigoplus_S A$ of $|S|$ copies of an abelian group $A$. The shifts, except trivial cases,  have always positive algebraic entropy. Moreover, one can define and study the generalized shifts $S \overset{\alpha_\lambda}\curvearrowright A^{(X)}$ associated to actions $S\overset{\lambda}\curvearrowright X$ of $S$ on a set $X$. The algebraic entropy of $\alpha_\lambda$ is related to the set-theoretic entropy of $\lambda$, introduced in \cite{AD} in the case of $\N$-actions. 
Gradually the material on generalized shifts and set-theoretic entropy grew more and more, so it will be published separately in \cite{DFG}.

\subsubsection*{Acknowledgments}

It is a pleasure to thank Benjamin Weiss for the numerous useful suggestions and advise, and for the fruitful discussions with the second named author.
We would like to thank also Luigi Salce, who suggested us to work together on this topic.

\smallskip
This work was partially supported by Programma SIR 2014 by MIUR  (Project GADYGR, Number RBSI14V2LI, cup G22I15000160008), 
and partially also by the ``National Group for Algebraic and Geometric Structures, and their Applications'' (GNSAGA - INdAM).

\subsubsection*{Notation and terminology}

For a subset $X$, we let $$\ell(X)=\log|X|,$$ using the convention that $\ell(X)=\infty$ if the set $X$ is infinite. Moreover, we denote by $\P(X)$ the family of all subsets of $X$ and by $\Pf(X)$ its subfamily consisting of all non-empty finite subsets of $X$.

For a monoid $S$, let $\Pf^0(S)=\{Y\in\Pf(S):1\in Y\}$. Clearly, $\Pf^0(S)\subseteq
\Pf(S)\subseteq \P(S)$.

If $A$ is an abelian group, let $\mathcal F(A)$ denote the family of all
finite subgroups of $A$.

For a subset $W$ of $A$ and $m\in\N_+$, let $$W_m= \underbrace{W+W+\ldots+W}_m.$$

Let $S$ be a semigroup. For $F,F'\in\P(S)$, we denote $$F F'= \set{gh: g \in F, h \in F'}.$$

A \emph{left semigroup action} $S\overset{\alpha}{\curvearrowright}X$ of $S$ on a set $X$ is defined by a map
$$\alpha:S\times X\to X,\quad (s,x)\mapsto \alpha(s)(x)=s\cdot x,$$
such that $\alpha(st)=\alpha(s)\circ \alpha(t)$ for every $s,t\in S$ (i.e., $(st)\cdot x=s\cdot(t\cdot x)$ for every $s,t\in S$ and every $x\in X$).
Given $s \in S$, $x \in X$, and $Y \subseteq X$, we denote $$s \cdot Y = \set{\alpha(s)(y): y \in Y}.$$ Similarly, a right action $X\curvearrowleft S$ can be defined. 

In case $S$ is a monoid with neutral element $1$, a left semigroup action $S\overset{\alpha}{\curvearrowright}X$ is a \emph{left monoid action} of $S$ on $X$ if $\alpha(1)(x) = x$ for all $x\in X$, i.e., $\alpha(1)$ is the identity map $id_X$. If $S$ is a group, then this condition implies that $\alpha(s)$ is a bijection for every $s\in S$.
Unless otherwise stated, all the actions of monoids considered in this paper are monoid actions (see also Remark~\ref{groupaction}).

\begin{remark}\label{Sop}
\begin{enumerate}[(a)]
\item In order to avoid reformulating everything for both left and right actions, it is convenient to consider the opposite semigroup $(S^{op}, \cdot^{op})$ of the semigroup $(S,\cdot)$ defined as follows: 
$$S^{op} = S\mbox{ and } x\cdot ^{op}y = y\cdot x\mbox{ for every }x, y \in S.$$ 
This allows one to associate to every left action $S\overset{\alpha}{\curvearrowright}X$ a right action $X\overset{\alpha^{op}}{\curvearrowleft}S^{op}$ by simply putting $\alpha^{op}(s)= \alpha(s)$ for every $s\in S$.

\item For a non-empty set $Y$ we denote by $\Se ym(Y)$ the monoid of all selfmaps $Y \to Y$ with operation $f\cdot g = f \circ g$ and identity $id_Y$.  
Note that one has a left action $\Se ym(Y) \overset{\pi}\curvearrowright Y$ given by $\pi(\sigma)(y) = \sigma(y)$ for every $\sigma \in \Se ym(Y)$.

If $S$ is a semigroup and $S \overset{\gamma}\curvearrowright Y$ is a left semigroup action, then equivalently there is a semigroup homomorphism $\gamma: S \to \Se ym(Y)$. 
Analogously, a right semigroup action $Y \overset{\eta}\curvearrowleft S$ is given by a semigroup homomorphism $\gamma: S^{op} \to \Se ym(Y)$. 

In case $S$ is a monoid and $S \overset{\gamma}\curvearrowright Y$ is a  left monoid action, $\gamma: S \to \Se ym(Y)$ is a monoid homomorphism, i.e., $\gamma(1) = id_Y$. The same holds for right monoid actions.
\end{enumerate}
\end{remark}


\begin{remark}\label{groupaction}
Let $G$ be a group, $A$ an abelian group, and consider the left monoid action $G\overset{\alpha}{\curvearrowright}A$. Then $G$ acts on $A$ by automorphisms, that is, $\alpha$ takes $G$ into the group $\mathrm{Aut}(A)$.

If $\alpha$ is a left semigroup action (i.e., the condition $\alpha(1)=id_A$ need not be verifed), then $\alpha$ simply takes $G$ into some subsemigroup of the multiplicative semigroup of the unitary endomorphism ring $\mathrm{End}(A)$ that is a group but not necessarily contained in $\mathrm{Aut}(A)$; so $1$ may end up just in some idempotent element of $\mathrm{End}(A)$, not necessarily $id_A$ (e.g., $\alpha(G)$ can simply be the singleton $\{0\}$ in $\mathrm{End}(A)$). 
\end{remark}

Let $S$ be semigroup, $A$ an abelian group, and consider the left action $S\overset{\alpha}{\curvearrowright} A$.
We say that a subset $B$ of $A$ is \emph{$\alpha$-invariant} if $\alpha(s)(B)\subseteq B$ for every $s\in S$.
Moreover, if $T$ is a subsemigroup of $S$, we say that $B$ is \emph{$T$-invariant} if $\alpha(t)(B)\subseteq B$ for every $t\in T$, that is, $B$ is $\alpha\restriction_T$-invariant.

If $B$ is an $\alpha$-invariant subgroup of $A$, then $\alpha$ induces in a natural way an action $S\overset{\alpha_B}{\curvearrowright} B$ and an action $S\overset{\alpha_{A/B}}{\curvearrowright} A/B$.

\section{Background on amenable semigroups and F\o lner nets}\label{background}

\subsection{Amenable semigroups}

Let $S$ be a semigroup. For every $s\in S$ denote by $L_s:S\to S$ the left multiplication $x\mapsto sx$ and by $R_s:S\to S$ the right multiplication $x\mapsto xs$. The semigroup $S$ is \emph{left cancellative} (respectively, \emph{right cancellative}) if $L_s$ (respectively, $R_s$) is injective for every $s\in S$.
Every finite semigroup which is left cancellative and right cancellative is a group. 

\medskip
The semigroup $S$ is \emph{left amenable} if there exists a left subinvariant finitely additive probability measure on $S$, that is, a function $\mu:\P(S)\to [0,1]$ such that:
\begin{enumerate}[(L1)]
 \item $\mu(S)=1$;
 \item $\mu(F\cup E)=\mu(F)+\mu(E)$ for every $F,E\in\P(S)$ with $F\cap E=\emptyset$;
 \item $\mu(L_s^{-1}(F))=\mu(F)$ for every $s\in S$ and every $F\in\P(S)$.
\end{enumerate}

The semigroup $S$ is \emph{right amenable} if $S^{op}$ is left amenable, and
 $S$ is \emph{amenable} if it is both left amenable and right amenable (see \cite{Day1,Day2, Day3}). Every commutative semigroup is amenable \cite{Day1}.

\begin{remark}\label{am-meas}
The condition in (L3) is in general weaker than requiring $\mu$ to be left invariant (i.e., $\mu(sX)=\mu(X)$ for every $s\in S$ and every $X\in\P(S)$). In fact, assume that $\mu$ is left invariant and let $s\in S$ and $X\in\P(S)$. Denoting $f=L_s: S \to S$, our hypothesis yields 
\begin{equation}\label{inva}
\mu(f(Z)) = \mu(Z)\ \text{for every}\ Z\subseteq S.
\end{equation}
Let $Y= f(f^{-1}(X))$. Then $Y = X \cap f(S) \subseteq X$ and $\mu(f^{-1}(X)) = \mu(Y)$ by Equation~\eqref{inva}.
As $\mu(S) = \mu (f(S)) = 1$ by Equation~\eqref{inva}, we deduce that $$\mu(X \setminus Y) = \mu (X \setminus f(S))  \leq \mu(S \setminus f(S)) = 0.$$
Hence, $\mu(Y) = \mu(X)$, and so $\mu(f^{-1}(X)) = \mu(Y) = \mu(X)$. 

\medskip
In \cite{GK} the semigroups admitting a left invariant finitely additive probability measure are called \emph{left measurable.}
If $S$ is cancellative, then the two conditions are equivalent, that is, $S$ is left amenable if and only if $S$ is left measurable. Indeed, let $s\in S$ and $F\in\P(S)$; since $S$ is cancellative, $L_s^{-1}(sF)=F$, so if (L1) holds for $\mu$, then $\mu(F)=\mu(L_s^{-1} (sF))=\mu(sF)$.
\end{remark}


While left amenability and right amenability are equivalent for groups, and every finite group is amenable, there exist finite semigroups that are left amenable but not right amenable as the next example shows. An example of a cancellative right amenable semigroup which is not left amenable can be   found in \cite{Kla}.

\begin{example}\label{p1p2} 
Let $S$ be a semigroup  and let $\mu$ be a left subinvariant finitely additive probability measure on $S$. Assume that $z\in S$ is a left zero element of $S$ (i.e., $zx=z$ for every $x\in S$). Then (L1) implies that $\mu(\{z\})=1$, so $\mu(S\setminus\{z\})=0$ by (L2).
Hence, a left amenable semigroup cannot admit more than one left zero element.

Let $S = \set{p_1, p_2}$, with multiplication $p_i x = p_i$, $i = 1, 2$. Then $p_1$ and $p_2$ are left zero elements of the semigroup $S$, so $S$ is not left amenable. On the other hand, $S$ is right amenable with the measure $\mu:S\to [0,1]$ defined by $\mu(\{p_1\})=a_1$ and $\mu(\{p_2\})=a_2$, where $a_1,a_2\in[0,1]$ and $a_1+a_2=1$.
\end{example}

An example of an amenable group containing a copy of the free semigroup in two generators (which is not right amenable) can be found in \cite{Hochster1969}. In particular, not every submonoid of an amenable group is amenable.

\medskip
We say that a semigroup $S$ satisfies the \emph{left F\o lner condition} (briefly, lFC) if for every $K\in\Pf(S)$ and every $\varepsilon>0$ there exists $F\in\Pf(S)$ such that 
$$|kF\setminus F|\leq \varepsilon|F|\ \text{for every}\ k\in K.$$
We say that $S$ satisfies the \emph{right F\o lner condition} (briefly, rFC) if $S^{op}$ satisfies the lFC.

Clearly every finite semigroup $S$ satisfies both lFC and rFC. 

\medskip
A \emph{left F\o lner net} of a semigroup $S$ is a net $(F_i)_{i\in I}$ in $\Pf(S)$ such that for every $s\in S$
$$\lim_{i\in I}\frac{|sF_i\setminus F_i|}{|F_i|}=0.$$
Analogously, a  \emph{right F\o lner net} is a left F\o lner net of $S^{op}$.

\begin{example}\label{iunion}
If $S$ is a semigroup such that $S$ is increasing union  of a sequence $(F_n)_{n\in\N}$ of its finite subsemigroups, then $(F_n)_{n\in\N}$ is both a left F\o lner sequence and a right F\o lner sequence of $S$. This condition is satisfied for example by all countable torsion abelian groups, and more in general by all countable locally finite groups.
\end{example}

A semigroup $S$ satisfies lFC (respectively, rFC) if and only if there exists a left F\o lner net (respectively, a right F\o lner net) of $S$ (see \cite{CCK} and \cite[Proposition 4.7.1]{CC}).
On the other hand, every left amenable semigroup satisfies lFC (respectively, every right amenable semigroup satisfies rFC) (see \cite{Frey} and \cite[Theorem 3.5]{N}). 

By F\o lner Theorem \cite{F}, a group $G$ is amenable if and only if $G$ satisfies lFC (respectively, rFC), 
%
%
but this cannot be extended to  semigroups, since there exist non left amenable semigroups $S$ satisfying lFC; indeed, a finite semigroup $S$ satisfies lFC, so it suffices to take a finite semigroup which is not left amenable as in Example~\ref{p1p2}. The equivalence holds again if we suppose that the semigroup is left cancellative:

\begin{theorem}[{see \cite[Corollary 4.3]{N}}]
Let $S$ be a left cancellative semigroup. Then the following conditions are equivalent:
\begin{enumerate}[(a)]
 \item $S$ is left amenable;
 \item $S$ satisfies lFC;
 \item $S$ admits a left F\o lner net.
\end{enumerate}
\end{theorem}

Clearly, the counterpart of this result holds also for right cancellative semigroups and right amenability.


\subsection{F\o lner nets}

\begin{remark}
Let $S$ be a right amenable semigroup. If $S$ is infinite and $(F_i)_{i\in I}$ is a right F\o lner net of $S$, then $\lim_{i\in I}\card{F_i}=\infty$. On the other hand, $(F_i)_{i\in I}$ need not be strictly increasing. For example, in $\N$ or $\Z$ consider $F_n=[n,n!]$ for every $n\in\N$, or $F'_n=[n,2n]$ for every $n\in\N$.
\end{remark}

The following known equivalent description of right F\o lner nets is a consequence of the definition.

\begin{lemma}\label{Feq}
Let $S$ be a cancellative right amenable semigroup. Then $(F_i)_{i\in I}$ is a right F\o lner net of $S$ if and only if, for every $s\in S$,
$$\lim_{i\in I}\frac{\card{F_is\SDiff F_i}}{\card{F_i}}=0.$$
\end{lemma}
\begin{proof} 
It suffices to note that, for $F\in\Pf(S)$ and $s \in S$, we have $$\card{Fs \SDiff F} = 2 \card{Fs \setminus F}.$$
Indeed, $\card F + \card{Fs \setminus F} = \card{F \cup Fs} =\card{Fs} + \card{F \setminus Fs}.$
Since $S$ is cancellative, $\card{Fs} = \card F$ is finite, and therefore $\card{Fs \setminus F} = \card{F \setminus Fs}.$
\end{proof}

%

\begin{lemma}\label{F_iF}
Let $S$ be a cancellative right amenable monoid and $(F_i)_{i \in I}$ a net in $\Pf(S)$. Then:
\begin{enumerate}[(a)]
\item $(F_i)_{i \in I}$ is a right F\o lner net of $S$ if and only if for every $E\in\Pf(S)$,
\begin{equation}\label{eq:Folner-finite}
\lim_{i \in I} \frac{\card{F_i E \SDiff F_i}}{\card{F_i}} = 0;
\end{equation}
 \item if $(F_i)_{i \in I}$ is a right F\o lner net of $S$ and $E\in\Pf(S)$, then 
\begin{equation}\label{?}
\lim_{i\in I}\frac{\card{F_iE}}{\card{F_i}}=1
\end{equation}
and $(F_i E)_{i \in I}$ is also a right F\o lner net of $S$.
\end{enumerate}
\end{lemma}
\begin{proof}
(a) The sufficiency of  the condition~\eqref{eq:Folner-finite} (when applied to singletons) is obvious.  

Conversely, assume that $(F_i)_{i \in I}$ is a right F\o lner net and let $E\in\Pf(S)$. Fix $\eps > 0$. By Lemma~\ref{Feq}, for every $s\in E$ there exists $i_s\in I$ such that, for every $j \in I$ with $j \geq i_s$,
$$\frac{\card{F_j s \SDiff F_j}}{\card{F_j}} < \eps.$$
Pick an $\imath \in I$ such that $\imath \geq i_s$ for every $s \in E$ (such an $\imath\in I$ exists since $E$ is finite and $I$ is a directed set).

Let $j > \imath$. Then
$$F_j E \SDiff F_j \subseteq \bigcup_{s \in E} F_j s \SDiff F_j,$$
and so
$$\frac{\card{F_j E \SDiff F_j}}{\card{F_j}} \leq \sum_{s \in E} \frac{\card{F_j s \SDiff F_j}}{\card{F_j}}< \eps \card E.$$
Thus,~\eqref{eq:Folner-finite} holds.

(b) The equality $\lim_{i\in I}\frac{\card{F_iE}}{\card{F_i}}=1$ can be easily deduced from (a), since it implies that $\lim_{i\in I}\frac{\card{F_iE\setminus F_i}}{\card{F_i}}=0$. Let $s \in S$; by Lemma~\ref{Feq} we have to show that
\begin{equation}\label{baba+}
\lim_{i \in I} \frac{\card{F_iEs \SDiff F_iE}}{\card{F_i E}}=0.
\end{equation}
The inclusion $F_iEs \SDiff F_iE\subseteq (F_iEs \SDiff F_i)\cup (F_iE \SDiff F_i)$, gives 
\begin{equation*}\label{baba}
\frac{\card{F_iEs \SDiff F_iE}}{\card{F_i E}} \leq \frac{\card{F_iEs \SDiff F_i} +\card{F_iE \SDiff F_i}}{\card{F_iE}}= \frac{\frac{\card{F_iEs \SDiff F_i}}{\card{F_i}} +\frac{\card{F_iE \SDiff F_i}}{\card{F_i}}}{\frac{\card{F_iE}}{\card{F_i}}}.
\end{equation*}
Item (a) yields $\lim_{i\in I} \frac{\card{F_iEs \SDiff F_i}}{\card{F_i}}=0$ and $\lim_{i\in I}\frac{\card{F_iE \SDiff F_i}}{\card{F_i}}=0$, while $\lim_{i\in I}\frac{\card{F_iE}}{\card{F_i}}=1$ by Equation~\eqref{?}. This proves the equality in Equation~\eqref{baba+}.
\end{proof}

\begin{remark}\label{Rem:Nov19} 
One may get the wrong impression that if  $(F_i)_{i\in I}$ is a right F\o lner net of a semigroup $S$, then the family $\{F_i: i\in I\}$ generates a ``big'' subsemigroup of $S$ (actually, the whole $S$). To see that this is wrong consider two examples. 
\begin{enumerate}[(a)]
\item Suppose that $S$ has a right absorbing element $a$ (i.e., such that $as =  a$ for all $s\in S$). Now the singleton $\{a\}$ is a subsemigroup of $S$ and forms a stationary right F\o lner net of $S$. 
\item Consider now the abelian group $S = \Z$. The F\o lner sequence $([0,n])_{n\in\N}$ is contained in the subsemigroup $\N$ of $S$.
\end{enumerate} 
Yet, it is true that the subgroup generated by any right F\o lner net of a group $S$ coincides with $S$. 
\end{remark}

The following well-known fact concerns right F\o lner nets in direct products. We omit its standard proof, since we prove a much more general result below in Theorem~\ref{lem:double-limit}.

\begin{lemma}\label{Fnprod}
Let $G = H\times K$ be an amenable group, let $(H_i)_{i\in I}$ be a right F\o lner net of $H$ and let $(K_j)_{j\in J}$ be a right F\o lner net of $K$. Then $(H_i\times K_j)_{(i,j)\in I\times J}$ is a right F\o lner net of $G$. 
\end{lemma}

In Example~\ref{Example: semi-direct product} we show that the na\"ive generalization of this lemma for semidirect products fails, while Theorem~\ref{lem:double-limit} gives the correct generalization.

\medskip

Recall that a semigroup $S$ is called \emph{Ore semigroup} (or \emph{right reversible semigroup} as in \cite{Paterson}
) if it satisfies the \emph{left Ore condition}, that is, 
\begin{center}
for every $a, b \in S$ there exist $f,g \in S$ such that $fa = gb$.
\end{center}
It is a well-known fact (see \cite[Proposition 1.23]{Paterson}) that a right amenable semigroup $S$ satisfies the left Ore condition, and that the group generated by $S$ is of the form $S^{-1} S$.

We use the following consequence of this property.

%

\begin{corollary}\label{cor:multi-Ore}
Let $S$ be a right amenable semigroup, and $s_1, \dotsc, s_k \in S$. Then there exist $t, r_1, \dotsc, r_k \in S$ such that $t = r_i s_i$ for every $i\in\{1, \dots, k\}$.
\end{corollary}

We conclude with the following relation between amenable monoids and amenable groups.

\begin{lemma}\label{S->G}
Let $G$ be a group and $S$ a submonoid of $G$ that generates $G$ as a group. 
\begin{enumerate}[(a)]
\item If $S$ is right amenable, then $G$ is amenable. 
\item If $(F_i)_{i\in I}$ is a right F\o lner net of $S$, then $(F_i)_{i\in I}$ is a right F\o lner net of $G$.
\end{enumerate}
\end{lemma}
\begin{proof} 
(a) Let $g \in G$. Then, there exist $a_1, \dots, a_n \in S$ such that $g = b_1 \ldots b_n$, where $b_i \in\{a_i, a_i^{-1}\}$.
Let $F\in\Pf(S)$, such that, for every $i \leq n$, $\card{F a_i \SDiff F} < \eps \card F$.
For every $m \leq n$, let $g_m = b_1 \dots b_m \in G$.
Then we prove that 
\begin{equation}\label{claim0}
\card {F g_m \SDiff F} < m \eps \card F.
\end{equation}
Then, taking $m = n$ we get the F\o lner condition for $G$.

Thus, we are left with the proof of Equation~\eqref{claim0}. If $m = 1$, then $g_1 = b_1 = a_1^{\pm 1}$.
If $b_1 = a_1$, then $\card {F a_1 \SDiff F} <\eps \card F$ by definition of $F$. If $b_1 = a_1^{-1}$, then
\[\card {F b_1 \SDiff F}  = \card {F \SDiff F b_1^{-1}} = \card {F \SDiff F{a_1}}<\eps \card F,\]
also by definition of $F$.

Assume, by inductive hypothesis, that we have proved the claim for $m$; we want to prove it also for $m + 1$. In view of the inclusion
$$F g_{m}b_{m+1} \SDiff F \subseteq (F g_{m}b_{m+1} \SDiff F b_{m+1})\cup (F b_{m+1} \SDiff F)$$
we have that
\begin{align*}
\card {F g_{m+1} \SDiff F} &= \card {F g_{m}b_{m+1} \SDiff F} \\ 
&\leq \card {F g_{m}b_{m+1} \SDiff F b_{m+1}} + \card {F b_{m+1} \SDiff F} \\ 
&=\card {F g_{m} \SDiff F} + \card {F b_{m+1} \SDiff F}\\
&\leq m \eps \card F  + \eps \card F \\ 
&= (m+1)\eps \card F,
\end{align*}
where we have used the inductive hypothesis and the case $m = 1$.

(b) Let $(F_i)_{i\in I}$ be a right F\o lner net of $S$. Then, for every $s\in S$, $\lim_{i\in I}\frac{\card{F_is\SDiff F_i}}{\card{F_i}}=0$.
Let $g\in G$. As above, there exist $a_1, \dots, a_n \in S$ such that $g = b_1 \ldots b_n$, where $b_i\in\{a_i, a_i^{-1}\}$. Moreover, for $\eps>0$, there exists a cofinal  subset $J\subseteq I$ such that, for every $i\in J$ and for every $l\in\{i,\ldots,n\}$, $$\frac{\card{F_i s_l\SDiff F_i}}{\card{F_i}}<\eps;$$
by \eqref{claim0}, for every $i\in J$, $$\frac{\card {F_i g\SDiff F_i}}{\card{F_i}}<n\eps.$$
Hence $\lim_{i\in I}\frac{\card {F_i g\SDiff F_i}}{\card{F_i}}=0$, and so $(F_i)_{i\in I}$ is a right F\o lner net of $G$.
\end{proof}

\begin{corollary}
If $S$ is a cancellative right amenable monoid, then it embeds in an amenable group.
\end{corollary}


\subsection{Canonically indexed F\o lner nets}

Fixed $\eps\in(0,1]$, we introduce the relation $$\sim_\eps$$ for finite subsets of a set $S$, although we use it in the sequel when $S$ is a monoid. 
For $F,F'\in\Pf(S)$, let
$$F\sim_\eps F'\ \text{if and only if}\ \card{F}=\card{F'}\ \text{and}\ \card{F\SDiff F'}\leq\eps\card F.$$

The relation $\sim_\eps$ is reflexive and symmetric; we see that in some sense it is also transitive:

\begin{lemma}\label{aequiv}
Let $S$ be a set, $F,F',F''\in\Pf(S)$ and $\eps,\eps'\in(0,1]$.
\begin{enumerate}[(a)]
\item $F\sim_\eps F$;
\item $F\sim_\eps F'$ implies $F'\sim_\eps F$;
\item $F\sim_\eps F'$ and $F'\sim_{\eps'} F''$ imply $F\sim_{\eps+\eps'} F''$.
\end{enumerate}
\end{lemma}
\begin{proof}
(a) is obvious, (b) follows from the fact that $\card F=\card{F'}$ when $F\sim_\eps F'$, and (c) follows from the inclusion $F\SDiff F''\subseteq (F\SDiff F')\cup(F'\SDiff F'')$.
\end{proof}

We see other useful properties of the relation $\sim_\eps$.

\begin{lemma}\label{sim}
Let $S$ be a set, $F,F',F'',E,E'\in\Pf(S)$ and $\eps\in(0,1]$. If $F\sim_\eps F'$ and $E\sim_\eps E'$, then:
\begin{enumerate}[(a)]
\item $F \sqcup E \sim_{\eps} F' \sqcup E'$;
\item $F \times E \sim_\eps F' \times E$;
\item $F \times E \sim_{4 \eps} F' \times E'$;
\item when $S$ is a right cancellative semigroup, $Fs\sim_\eps F's$ for every $s\in S$.
\end{enumerate}
\end{lemma}
\begin{proof} 
The inclusions 
$$(F\sqcup E)\SDiff (F'\sqcup E')\subseteq (F\SDiff F')\cup (E\SDiff E') \  \ \mbox{ and } \ \  (F\times E)\SDiff(F'\times E)\subseteq (F\SDiff F')\times E$$
imply (a) and (b) respectively, while (c) follows from the inclusion
$$ (F\times E)\SDiff(F'\times E')\subseteq ((F\SDiff F')\times E)\cup((F\SDiff F')\times E')\cup (F\times(E\SDiff E'))\cup (F'\times (E\SDiff E')).$$
The equality $Fs\SDiff F's=(F\SDiff F')s =R_s(F\SDiff F')$ for $s\in S$, and the injectivity of $R_s$, give (d).
\end{proof}

\medskip
Let $S$ be a monoid, and let $$\mathcal I(S)=\Pf^0(S)\times\N_+,$$ endowed with the partial order defined as follows. For every $E,E'\in\Pf^0(G)$ and $n,n'\in\N_+$,
$$(E,n)\leq (E',n')\ \text{if and only if}\ E\subseteq E'\ \text{and}\ n\leq n'.$$

\begin{definition}
Let $S$ be a cancellative monoid. A right F\o lner net a $(F_i)_{i\in \mathcal I(S)}$ of $S$ is \emph{canonically indexed} if
\begin{equation}\label{FEn}
F_{(E,n)}s\sim_{\frac{1}{n}} F_{(E,n)}
\end{equation}
for every $(E,n)\in\mathcal I(S)$ and every $s\in E$.
\end{definition}

We see several useful properties of the canonically indexed right F\o lner nets.

\begin{lemma}
Let $S$ be a cancellative right amenable monoid and $(F_i)_{i\in \mathcal I(S)}$ a canonically indexed right F\o lner net of $S$. Then:
\begin{enumerate}[(a)]
\item $(F_i)_{i\in \mathcal I(S)}$ is a right F\o lner net of $S$;
\item for every $E\in\Pf^0(S)$ and every $s\in E$,
$$\lim_{n\to \infty}\frac{\card{F_{(E,n)}s\SDiff F_{(E,n)}}}{\card{F_{(E,n)}}}=0.$$
\end{enumerate}
\end{lemma}
\begin{proof}
(a) Let $s \in S$ and $n\in\N_+$. We have to find $\bar i \in \mathcal I(S)$ such that, for every $j\in\mathcal I(S)$ with $j\geq \bar i$, 
$$\frac{\card{F_j s \SDiff F_j}}{\card{F_j}}\leq \frac 1 n.$$
Take $\bar i=(\set{e,s},n)$ and  $j=(X,m)\in\mathcal I(S)$ with $j\geq\bar i$; this means that $s\in X$ and $m \geq n$. Thus, Equation~\eqref{FEn} yields
$$\frac{\card{F_j s \SDiff F_j}}{\card{F_j}}=\frac{\card{F_{(X,m)} s \SDiff F_{(X,m)}}}{\card{F_{(X,m)}}} \leq \frac 1 m \leq \frac 1 n.$$

(b) Follows directly from Equation~\eqref{FEn}.
\end{proof}

Canonically indexed right F\o lner nets are always available in cancellative right amenable monoids:

\begin{proposition}\label{canonical}
Let $S$ be a cancellative right amenable monoid. Then $S$ admits a canonically indexed right F\o lner net.
\end{proposition}
\begin{proof} 
Let $(E,n)\in\mathcal I(S)$. By  Lemma~\ref{F_iF}(a), there exists $F\in\Pf(S)$ such that
$$\frac{\card{F E \SDiff F}}{\card{F}}\leq \frac 1 n.$$
Let $F_{(E,n)}=F$. Then $(F_{(E,n)})_{(E,n)\in\mathcal I(S)}$ is a canonically indexed right F\o lner net of $S$.
\end{proof}

\subsection{Good sections and canonical \Folner nets}

Let $S$ be a cancellative monoid, $C$ a monoid, and $\pi: S \to C$ a surjective homomorphism. 
For $s,s'\in S$ let $$s\sim s'\ \ \text{if}\ \ \pi(s) = \pi(s').$$ 
It is well-known that $\sim$ is a congruence, that is, an equivalence relation compatible with the semigroup operation of $S$. In particular, $$N= [1]_\sim = \pi^{-1}(1)$$ is a submonoid of $S$.  Obviously, every fibre $$[s]_\sim = \pi^{-1}(\pi(s))$$ satisfies $$Ns \subseteq [s]_\sim \supseteq sN;$$ nevertheless, these inclusions need not (simultaneously) become equalities. 

\begin{definition} 
Let $S$ be a cancellative monoid, $C$ a monoid, and $\pi: S \to C$ a surjective homomorphism.
\begin{enumerate}[(a)] 
\item An element $s\in S$ is \emph{good} (respectively, \emph{semi-good}) if  $Ns= [s]_\sim = sN$ (respectively, $Ns= sN$).
\item A fibre $[s]_\sim$ is \emph{good} (respectively, \emph{semi-good}) if it admits a good (respectively, semi-good) representative.
\end{enumerate}
\end{definition}

The existence of good elements of a specific fibre $[s]_\sim$ does not mean that all elements of $[s]_\sim$ are good. Indeed, $N$ is a good fiber as $1$ is a good element, and the set of all good elements of $N$ coincides with the group $U(N)$ of all invertible elements of $N$.

\begin{lemma}\label{lem:good-5}  
Let $S$ be a cancellative monoid, $C$ a monoid, and $\pi: S \to C$ a surjective homomorphism.  If $s\in S$ is semi-good, then for every $n \in N$ there exists a unique element $h_s(n) \in N$ such that $$n s = sh_s(n).$$ 
Then $h_s:N\to N$, $n\mapsto h_s(n)$, is an automorphism of $N$. 
\end{lemma}
\begin{proof} 
Obviously, $h(n)$ exists by definition, and it is unique because $S$ is cancellative.

Consider $h:N\to N$, $n\mapsto h_s(n)$. Then $h$ is bijective since it has an inverse, defined by associating to every $m \in N$ the (unique) $n\in N$ such that $ns = sm$.
Moreover, for $n, n' \in N$, we have that
\[s h_s(n) h_s(n') = n s h_s(n') = n n' s = s h_s(n n');\]
since $S$ is cancellative, we conclude that $h_s(n n') = h_s(n)h_s(n')$, and therefore $h_s$ is an automorphism of $N$.
\end{proof}


\smallskip
Now we face the situation when all fibers are good. Recall that, if $\pi: S \to C$ a surjective monoid homomorphism, a \emph{section} for $\pi$ is a map $\sigma: C \to S$ such that $$\pi(\sigma(c)) = c\quad \text{for every}\ c \in C.$$

\begin{definition}
Let $S$ be a cancellative monoid, $C$ a monoid, and $\pi: S \to C$ a surjective homomorphism.
A section $\sigma$ for $\pi$ is \emph{good} if $\sigma(c)$ is good for every $c\in C$.
\end{definition}

Clearly, $\sigma$ is a good section for $\pi$ if and only if
\begin{equation}\label{def:good}
\pi^{-1}(c) = N \sigma(c) = \sigma(c) N\quad \text{for every}\ c \in C.
\end{equation}
Moreover, $\pi$ has a good section if and only if all fibers of $\pi$ are good; the section is defined, for every $c\in C$, by $\sigma(c) = s_c$, where $s_c$ is any of the good elements of $\pi^{-1}(c)$.

\medskip
Several possibilities can occur; indeed, it may happen that all sections are good, some sections are  good, or no section is good.
First, if $S$ and $C$ are groups and $\pi: S \to C$ is a surjective homomorphism, then every section for $\pi$ is good.

\begin{examples}\label{ex:good}
\begin{enumerate}[(a)]
\item If $S = N \times C$ for some cancellative monoids $N$ and $C$, and $\pi: S \to C$ is the canonical projection, then the canonical section $\sigma: c \mapsto (1,c)$ is good.

\item Let $N$ and $C$ be cancellative monoids, and let $\phi: C \to \Aut(N)$ be a homomorphism of monoids. Let $S$ be the semidirect product of $N$ and $C$ with respect to $\phi$ and $\pi: S \to C$ be the canonical projection. Then, the canonical section $\sigma: c \mapsto (1,c)$ is good.

\item Let $S = (\N, +)$. Fix $n \geq 1$ and let $C= \Z/ n\Z$ with the usual addition. Let $\pi: \N \to C$ be the canonical projection. Then, $\pi$ has a unique good section.

\item Fix $n \in \N$ and let $S= \set{0, 1, \dotsc, n-1} \cup [n, + \infty) \subseteq \R$. Then $S$ is a submonoid of $(\R, +)$. Let $C= \R/\Z$ and $\pi: S \to C$ be the restriction of the canonical projection $\R \to C$. Then $\pi$ has a unique good section.

\item Let $S = (\N, +)$. Let $C = \set{0,1,2}$ with commutative monoid operation $\oplus$ given by $x \oplus y = \min(2, x + y)$ (notice that $0$ is the neutral element of $C$). Let $\pi: \N \to C$ be the homomorphism defined by $x\mapsto \min(2, x)$. Then no section for $\pi$ is good.

\item A variant of the above example is the following. Let $S= \R_{\geq 0}$, and $C= \Pa{[0,1], \oplus}$ be the standard MV-algebra:
$x \oplus y := \min(1, x + y)$. Let $\pi: S \to C$ be the homomorphism $\pi(x)= \min(1, x)$. Then, no section for $\pi$ is good.

\item  Let $S = (\R_{\geq 0}, +)$, $C= \R/\Q$, and let $\pi : S \to C$ be the restriction of the canonical projection $\R \to C$. Even if $S$ is cancellative and $C$ is a group, there is no good section for $\pi$.
\end{enumerate}
\end{examples}

\begin{lemma}\label{lem:good-1}
Let $S$ be a cancellative monoid, $C$ a monoid, and $\pi: S \to C$ a surjective homomorphism admitting a good section $\sigma$.
Then every section for $\pi$ is good if and only if $N = \pi^{-1}(1)$ is a group. 
\end{lemma}
\begin{proof}
Suppose that $N$ is a group and let $\sigma'$ be another section for $\pi$. Fix $c \in C$ and let $s \in \pi^{-1}(c)$.
Then, there exist $m,n \in N$ such that $s = n \sigma(c)$ and $\sigma'(c) = m \sigma(c)$. Therefore, $s = n m^{-1} \sigma'(c) \in N \sigma'(c)$. Similarly one proves that $s \in \sigma'(c) N$. 

Suppose that every section for~$\pi$ is good. Let $n\in N$ and define a new section $\sigma'$ for $\pi$ by putting $\sigma'(c) = \sigma(c)$ for all $c\in C \setminus \{1\}$ and $\sigma'(1) = n$. Equation~\eqref{def:good} applied to $c=1$ and the good section $\sigma'$ gives $N=Nn=nN$, which yields that $n$ is an invertible element of $N$. 
\end{proof}

Item (b) of the next lemma makes use of the assumption that $S$ is cancellative.

\begin{lemma}\label{lem:good-2}\label{lem:good-3} 
Let $S$ be a cancellative monoid, $C$ a monoid, and $\pi: S \to C$ a surjective homomorphism admitting a good section $\sigma$. Then:
\begin{enumerate}[(a)]
\item $\sigma(x) \sigma(y) N \subseteq \sigma(xy) N = \pi^{-1}(xy)$ for every $x, y \in C$; 
\item the map $f: N \times C \to S$, $(n,c)\mapsto n \sigma (c)$ is a bijection.
\end{enumerate}
\end{lemma}
\begin{proof} 
(a) The equality is immediate by the assumption that $\sigma$ is good. To verify the (first) inclusion, let $n \in N$. Then $$\pi(\sigma(x) \sigma(y) n) ) = \pi(\sigma(x)) \pi(\sigma(y)) \pi(n) = xy,$$ and hence $$\sigma(x) \sigma(y) n \in \pi^{-1}(xy) = \sigma(xy) N.$$

(b) To see that $f$ is injective assume that $f(n,c) = f(n', c')$ for some $(n,c), (n',c')\in N\times C$. Then $n \sigma(c) = n' \sigma(c')$, and so $$c= \pi(n \sigma(c)) = \pi(n' \sigma(c')) = c'.$$ Since $S$ is cancellative, we conclude that $n = n'$.

To check that $f$ is surjective take $s \in S$; as $\sigma$ is good, we have that $s \in N \sigma(\pi(s)) \subseteq  f(N \times C)$.
\end{proof}

The following technical lemma applies in the next section.

\begin{lemma}\label{lem:good-exchange} 
Let $S$ be a cancellative monoid, $C$ a monoid, and $\pi: S \to C$ a surjective homomorphism admitting a good section $\sigma$. Assume that $N=\pi^{-1}(1)$ is right amenable. If $\bar C\in\Pf(C)$ and $y \in C$, there exist $Z\in\Pf(N)$ and $u \in N$ such that
\[u \sigma(y \bar C) \subseteq Z \sigma(y) \sigma(\bar C).\]
\end{lemma}
\begin{proof} 
Let $\bar C= \set{c_1, \dotsc, c_k}$. Since $\sigma$ is a good section, for every $i \in\{ 1, \dots, k\}$ there exists $t_i \in N$ such that $$\sigma(y)\sigma(c_i) = t_i \sigma(y c_i).$$ By Corollary~\ref{cor:multi-Ore}, there exist $u \in N$ and $z_1, \dotsc, z_k \in N$ such that $$u = z_i t_i\quad\text{for every}\ i\in\{ 1, \dots, k\}.$$ 
Let $Z= \set{z_1, \dots, z_k}$. Therefore,
\[u \sigma(y c_i) = z_i t_i \sigma(y c_i) = z_i \sigma(y) \sigma(c_i) \in Z \sigma(y) \sigma(\bar C)\]
for every $i\in\{ 1, \dots, n\}$.
\end{proof}

We see now that it may occur that $S$ is a non cancellative monoid even if $C$ is a cancellative monoid and $\pi:S\to C$ is a surjective homomorphism admitting a good section and such that $N=\pi^{-1}(1)$ is a subgroup of $S$.

\begin{example}\label{ex:non-cancellative}
Let $S= \Pa{\Z \times \set 0} \cup \Pa{\set 0 \times {\N_+}}$, with the operation
\[(x,n) + (y, m) = \begin{cases} (x + y, 0) & \text{if } n = m = 0\\ (0, m + n) & \text{otherwise}.\end{cases}\]
Then $(S, +)$ is a commutative non-cancellative monoid. Let $C=\N$ and let $\pi: S \to C$, $(a,b)\mapsto b$, notice that $\pi$ is a surjective homomorphism of monoids. Let $\sigma: \N \to S$, $\sigma(n)= (0,n)$. Then $C$ is cancellative, $\sigma$ is a good section and $N=\pi^{-1}(0)$ is a group.
\end{example}

We show that in presence of a good section for the surjective monoid homomorphism $\pi:S\to C$, with $S$ cancellative, the amenability of $S$ implies the amenability also of $C$ and $N=\pi^{-1}(1)$.

\begin{lemma}\label{lem:amenable-subgroup} 
Let $S$ be a cancellative monoid, $C$ a monoid, $\pi: S \to C$ a surjective homomorphism, and let $N=\pi^{-1}(1)$. If $S$ is right amenable, then:
\begin{enumerate}[(a)]
\item $C$ is right amenable;
\item in case $\pi$ admits a good section~$\sigma$, $N = \pi^{-1}(1)$ is right amenable as well.
\end{enumerate}
\end{lemma}
\begin{proof}
Remember that the right amenability of $S$ is equivalent to the existence of a right invariant finitely additive probability measure on $S$ (see Remark~\ref{am-meas}). Let $\mu$ be a right invariant finitely additive probability measure on~$S$.

(a) That $C$ is right amenable is a known fact (see \cite{Day3} -- the (short) proof is to define a right invariant finitely additive probability measure $\lambda$ on $C$ by $\lambda(Y)= \mu(\pi^{-1}(Y))$, as $Y$ varies among subsets of $C$).

(b) Let $T = \sigma(C)$, which is a transversal of $N$ in $S$. Let $X \subseteq N$, and define $\nu(X)=\mu(TX)$. It is easy to see that $\nu$ is a right invariant finitely additive probability measure on~$N$. The only place where we need that $\sigma$ is good is in proving that $\nu(N) = \mu(TN) = \mu(S) = 1$.
\end{proof}

We cannot drop the assumption that $\pi$ has a good section in the above lemma; indeed, in \cite{Hochster1969} one can find an example of a solvable amenable group containing a submonoid which is not right amenable.

\medskip
The following technical result is a kind of converse of Lemma~\ref{lem:amenable-subgroup} illustrating the impressing utility of the canonically indexed right F\o lner nets. Given cancellative monoids $S$ and $C$, a surjective homomorphism $\pi:S \to C$ and a good section $\sigma: C \to S$ for $\pi$, we prove that $S$ is right amenable whenever $N=\pi^{-1}(1)$ and $C$ are right amenable. More precisely, starting from canonically indexed right F\o lner nets of $N$ and $C$, we provide a canonically indexed right F\o lner net of $S$.

\begin{theorem}\label{lem:double-limit}
Let $S$ and $C$ be cancellative monoids, $\pi:S \to C$ be a surjective homomorphism admitting a good section $\sigma$, and $N = \pi^{-1}(1)$.  Let $(N_i)_{i\in\mathcal I(N)}$ and $(C_i)_{i\in\mathcal I(C)}$ be canonically indexed right F\o lner nets of $N$ and $C$, respectively. Then there exists a map $\zeta:\mathcal I(C)\to \Pf^0(N)$ such that 
\begin{enumerate}[(a)] 
\item for every $(X,m)\in\mathcal I(N)$ and $(Y,n)\in\mathcal I(C)$ with $m\geq n$, letting $$F_{((X,m),(Y,n))}= N_{(\zeta(Y,n)\cup X,m)}\sigma(C_{(Y,n)}),$$ 
we have that, for every $x\in X$ and $y\in Y$,
\begin{equation}\label{8/n}
F_{((X,m),(Y,n))}x\sigma(y)\sim_{\frac{3}{n}} F_{((X,m),(Y,n))}.
\end{equation}
\item Hence, the net
\[(F_{((X,m),(Y,n))})_{(X,m)\in\mathcal I(N),(Y,n)\in\mathcal I(C),m\geq n}\]
is a right F\o lner net of $S$  such that $$|F_{((X,m),(Y,n))}| = |N_{(\zeta(Y,n)\cup X,m)}|\cdot |C_{(Y,n)}|.$$
\end{enumerate}
\end{theorem}
\begin{proof} 
(a) Let $(Y,n)\in \mathcal I(C)$ and denote $\bar C=C_{(Y,n)}$. We have to produce $Z=\zeta(Y,n)\in \Pf^0(N)$ such that Equation~\eqref{8/n} holds, for every $(X,m)\in\mathcal I(N)$ with $m\geq n$, and for every $x\in X$ and $y\in Y$; that is, denoting 
\begin{equation}\label{barN}
\bar N=N_{(Z\cup X,m)},
\end{equation}
for every $x\in X$ and $y\in Y$ we have to verify that
\begin{equation}\label{8/nbis}
\bar N\sigma(\bar C) x\sigma(y)\sim_{\frac{3}{n}} \bar N\sigma(\bar C).
\end{equation}

In view of Lemma~\ref{lem:good-2}, for every $c\in\bar C$ and every $x\in X$, there exists $z_{c,x}\in N$ such that $$\sigma(c)x=z_{c,x}\sigma(c).$$
Moreover, for every $c\in\bar C$ and every $y\in Y$, there exists $z_{c,y}\in N$ such that $$\sigma(c)\sigma(y)=z_{c,y}\sigma(cy)$$ by the same lemma. Let 
$$Z=\{z_{c,x}: c\in\bar C,x\in X\}\cup\{z_{c,y}: c\in\bar C,y\in Y\}.$$
First we see that for every $x\in X$ and every $m\geq n$,
\begin{equation}\label{first}
\bar N\sigma(\bar C)x\sim_{\frac{1}{n}} \bar N\sigma(\bar C).
\end{equation}
To this end, fix $x \in X$ and $m\geq n$. Then, using the notation in~\eqref{barN}, since $\bar N\subseteq N$,
$$\bar N \sigma(\bar C) x =\bigsqcup_{c \in\bar C} \bar N \sigma(c)x=\bigsqcup_{c\in\bar C} \bar N z_{c,x}\sigma(c).$$
According to Equation~\eqref{FEn}, since $m\geq n$, we have that $\bar Nz_{c,x} \sim_{\frac{1}{n}} \bar N$ for every $c\in\bar C$; therefore, by Lemma~\ref{sim}(a,d) 
$$\bar N \sigma(\bar C)x=\bigsqcup_{c\in\bar C}\bar Nz_{c,x}\sigma(c)\sim_{\frac{1}{n}}\bigsqcup_{c\in\bar C}\bar N\sigma(c)=\bar N\sigma(\bar C).$$
This settles Equation~\eqref{first}.

Now we verify that for every $y\in Y$ and every $m\geq n$.
\begin{equation}\label{second}
\bar N\sigma(\bar C)\sigma(y)\sim_{\frac{2}{n}} \bar N\sigma(\bar C).
\end{equation}
To this end, fix $y\in Y$ and $m\geq n$. Then, according to Equation~\eqref{FEn} and Lemma~\ref{sim}(a,d), since $m\geq n$,
$$\bar N\sigma(\bar C)\sigma(y)=\bigsqcup_{c\in\bar C}\bar N\sigma(c)\sigma(y)=\bigsqcup_{c\in\bar C}\bar Nz_{c,y}\sigma(cy)\sim_{\frac{1}{n}} \bigsqcup_{c\in\bar C}\bar N\sigma(cy)=\bar N\sigma(\bar C y).$$
To complete the proof of Equation~\eqref{second}, we see now that $$\bar N\sigma(\bar C y)\sim_{\frac{1}{n}}\bar N\sigma(\bar C).$$
Indeed, $\bar Cy\sim_{\frac{1}{n}}\bar C$ by Equation~\eqref{FEn}, so $\sigma(\bar Cy)\sim_{\frac{1}{n}}\sigma(\bar C)$ since $\sigma$ is injective. Hence, $$\bar N\times \sigma(\bar Cy)\sim_{\frac{1}{n}}\bar N\times \sigma(\bar C)$$ by Lemma~\ref{sim}(b). Since the map $N\times C\to S$ defined by $(x,y)\mapsto x\sigma(y)$ is a bijection by Lemma~\ref{lem:good-2}(b), 
we conclude that $$\bar N\sigma(\bar C y)\sim_{\frac{1}{n}}\bar N\sigma(\bar C),$$ as required. In view of Lemma~\ref{aequiv}(c), this settles Equation~\eqref{second}.

Equation~\eqref{first} and~\eqref{second}, in view of Lemma~\ref{aequiv}(c), imply Equation~\eqref{8/nbis}, and so the thesis.

\medskip
(b) Let $$\mathcal I(N,C)=\set{((X,m),(Y,n))\in \mathcal I(N)\times\mathcal I(C): m\geq n},$$ with the partial order induced by the partial orders of $\mathcal I(N)$ and $\mathcal I(C)$.

Let $g\in S$ and $\eps>0$. We have to find $\imath \in \mathcal I(N,C)$ such that, for every $j=((X,m),(Y,n))\in \mathcal I(N,C)$ with $j\geq \imath$, $$\frac{\card{F_j g \SDiff F_j}}{\card{F_j}}\leq \eps.$$
Let $n\in\N_+$ such that $\frac{3}{n}\leq\eps$ and write $g=x\sigma(y)$ for some $x\in N$ and $y\in C$. Define $$\imath=((\set{e,x},n),(\set{e,y},n))\in \mathcal I(N,C)$$ and let $j=((X,m),(Y,k))\in \mathcal I(N,C)$ with $j\geq \imath$; this means that $x\in X$, $y\in Y$ and $m \geq k\geq n$.
Thus, by item (a),
$$\frac{\card{F_j g \SDiff F_j}}{\card{F_j}}\leq\frac 3 n\leq\eps,$$
as required.
\end{proof}

The map $(X,Y)\mapsto X\sigma(Y)$ in the above theorem gives a cofinal embedding, as partially ordered sets, of $\Pf^0(N)\times \Pf^0(C)$ in $\Pf^0(G)$.

\medskip
We show an explicit example of the above construction. 

\begin{example}\label{Example: semi-direct product}\label{quest:new1}
Let $A$ and $C$ be groups, $\varphi$ an action of $C$ on $A$, and
$G = A \rtimes_\varphi C$ be their semidirect product. That is, as a set $G = A \times C$, and the product of elements is given by
\[(a_1, c_1) * (a_2, c_2) = (a_1 \cdot a_2^{c_1}, c_1 \cdot c_2), \]
where $a^{c} = \varphi(c)(a)$.
Then (under the identifications $a \mapsto (a, e)$ and $c \mapsto (e, c)$),
$A$ and $C$ are subgroups of $G$, with $A$ normal in $G$ and $C$ isomorphic to
$G/A$. The map $\sigma: C \to G$ given by $c \mapsto (e, c)$ is a section of the quotient map $G \to C$, and, as said before, we identify $c$ with $\sigma(c)$.

\smallskip
Let $(A_i)_{i \in I}$ and $(C_j)_{j \in J}$ be right \Folner nets of $A$ and $C$ respectively; define $G_{i,j}= A_i * C_j$. We show that $(G_{i,j})_{i,j}$ need not be a right \Folner net in general. 

Let $x \in A$ and $y \in C$. Notice that $x * y = (x, y)$ and $y * x = (x^y, y)$. We have that
\[G_{i,j} * y = \set{(a,c) * (e,y): a \in A_i, c \in C_j} = \set{(a, c \cdot y): a \in A_i, c \in C_j} = A_i *(C_j \cdot y)\]
Thus,
\[\frac{\card{G_{i,j} * y  \setminus G_{i,j}}}{\card{G_{i,j}}} = \frac{\card{C_j \cdot y \setminus C_j}}{\card{C_j}},\]
which converges to $0$ as $j$ goes to infinity, uniformly in $i$. However, 
\[G_{i,j} * x = \set{(a,c) * (x,e): a \in A_i, c \in C_j} = \set{(a \cdot x^c, c): a \in A_i, c \in C_j} = \bigsqcup_{c \in C_j} (A_i \cdot x^c) * \set{c}.\]
Thus,
\[\frac{\card{G_{i,j} * x  \setminus G_{i,j}}}{\card{G_{i,j}}} = \frac 1 {\card{C_j}}\sum_{c \in C_j} \frac{\card{A_i \cdot x^c \setminus A_i}}{|A_i|},\]
which in general does not converge to $0$ as $i,j$ go to infinity. 

\medskip 
The following example was suggested to us by B. Weiss. 
Let $A = \Z^2$, $C = \Z$, and, for every $n\in\Z$ and $(v_1,v_2)\in\Z^2$,
\[\varphi(n)(v_1,v_2)= (v_1 + n v_2,v_2).\]
Take $x=(0,1)$, $C_n= [0,n-1]$, and $A_m = [0,m-1]^2$. Then
\[\delta_{n,m}(x) := \frac{\card{G_{m,n} * x  \setminus G_{m,n}}}{\card{G_{m,n}}} =\frac 1 {n}\sum_{c = 0}^{n-1} \frac{\card{A_m + x^c \setminus A_m}}{m^2}.\]
We have that
\[\delta_{n,n}(x) \geq  \frac 1 {n^3}\sum_{c = n/2}^{n-1} \card{A_n +(c,1) \setminus A_n} \geq \frac{n n^2/4}{n^3} = 1/4.\]
Therefore, $(G_{i,j})_{i,j}$ is not a right \Folner net. 

On the other hand, for every $\eps > 0$, $n \in \N$  and $x \in \Z^2$, if we take $m$ large enough (depending on $n$, $x$, and $\eps$), we have that
\[\frac{\card{A_m + x^c \setminus A_m}}{\card{A_m}} < \eps\]
for every $c \in C_n$, and therefore $\delta_{n,m}(x) < \eps$. Thus, for a function $f(n)$ growing fast enough, the sequence $(G_{f(n),n})_{n \in \N}$ is a right \Folner sequence of $G$. 

For a similar example see \cite[Example 0.5]{Paterson}.
\end{example}

\section{An integral for subadditive functions}\label{sec:integral}

\subsection{Definition of the integral}

Let $S$ be a semigroup and let $f:\Pf(S)\to \R$ be a function. Following \cite{CCK}, we say that $f$ is:
\begin{enumerate}[(1)]
 \item \emph{subadditive} if $f(F_1\cup F_2)\leq f(F_1)+f(F_2)$ for every $F_1,F_2\in\Pf(S)$;
 \item \emph{right subinvariant} (respectively, \emph{left subinvariant}) if $f(Fs)\leq f(F)$ (respectively, if $f(sF)\leq f(F)$) for every $s\in S$ and every $F\in\Pf(S)$;
 \item \emph{right invariant} (respectively, \emph{left invariant}) if $f(Fs)=f(F)$ (respectively, if $f(sF)= f(F)$) for every $s\in S$ and every $F\in\Pf(S)$;
 \item \emph{uniformly bounded on singletons} if there exists a real number $M\geq 0$ with $f(\{s\})\leq M$ for every $s\in S$.
\end{enumerate}

A subadditive function is automatically non-negative; in fact, if $f:\Pf(S)\to \R$ is subadditive, then $f(F) = f(F \cup F) \leq f(F) + f(F)$ for every $F\in\Pf(S)$.

If $S$ is a monoid, then from the fact that $f$ is right subinvariant, it
follows that $f(\{s\})\leq f(\{e\})$ for every $s\in S$, that is, $f$ is uniformly bounded on singletons. If $S$ is a group, then from the fact that $f$ is right subinvariant, it follows that $f$ is right invariant, that is, $f(Fs)=f(F)$ for every $s\in S$ and every $F\in\Pf(S)$.

The obvious counterparts for left subinvariance and left invariance hold true. 

\medskip
The following is the counterpart of Ornstein-Weiss Lemma for cancellative left amenable semigroups.

\begin{theorem}[{see \cite[Theorem 1.1]{CCK}}]\label{CCKLemma} 
Let $S$ be a cancellative left amenable semigroup and let $f:\Pf(S)\to \R$ 
be a subadditive right subinvariant function uniformly bounded on singletons. Then there exists $\lambda\in\R_{\geq 0}$ such that, for every left F\o lner net $(F_i)_{i\in I}$ of $S$, $$\lim_{i\in I}\frac{f(F_i)}{|F_i|}=\lambda.$$
\end{theorem}

By applying Theorem~\ref{CCKLemma} to $S^{op}$, one has the following ``dual" version that we formulate here for reader's convenience.

\begin{corollary}\label{CCKLemmar}
Let $S$ be a cancellative right amenable semigroup and let $f:\Pf(S)\to \R$ be a subadditive left subinvariant function uniformly bounded on singletons. Then there exists $\lambda\in\R_{\geq 0}$ such that, for every right F\o lner net $(F_i)_{i\in I}$ of $S$, $$\lim_{i\in I}\frac{f(F_i)}{|F_i|}=\lambda.$$
\end{corollary}

Let $S$ be a cancellative right amenable semigroup and let $$\mathcal S(S)$$ be the family of all functions $f:\Pf(S)\to \R_+$ 
that are increasing (i.e., such that $f(F)\leq f(F')$ whenever $F\subseteq F'$), subadditive, left subinvariant, and uniformly bounded on singletons. 

\medskip
By Corollary~\ref{CCKLemmar}, the limit in the following definition exists and it does not depend on the choice of the right F\o lner net. The mere existence of the limit defining $\mathcal H_S(f)$ does not require that the function in $\mathcal S(S)$ is increasing. Yet this very mild property is always present in all cases of interest, so it is harmless to impose it as a blanket condition in the definition of $\mathcal S(S)$, since it is needed in some proofs in the sequel (e.g., Lemma~\ref{pre-Fubini}, Theorem~\ref{Fubini}, etc.).

If $N$ is a right amenable subsemigroup of a cancellative right amenable semigroup $S$ and $f\in \mathcal S(S)$, then $f\restriction_{\Pf(N)}\in \mathcal S(N)$.

\begin{definition}\label{Integral:definition}
Let $S$ be a cancellative right amenable semigroup and $f\in\mathcal S(S)$.
Define $$\mathcal H_S(f)=\lim_{i\in I}\frac{f(F_i)}{\card{F_i}},$$
where $(F_i)_{i\in I}$ is a right F\o lner net of $S$.
 
If $N$ is a right amenable subsemigroup of $S$, let
$$\mathcal H_N(f)=\mathcal H_N(f\restriction_{\Pf(N)}).$$
\end{definition}

The following results follow directly from the definition.

\begin{lemma}\label{constantFubini} 
Let $S$ be a cancellative right amenable semigroup and $f\in\mathcal S(S)$.
\begin{enumerate}[(a)]
\item If $f$ is the constant function $f\equiv a\in\R_+$, then 
$$\mathcal H_S(f)=\begin{cases}0&\text{if $S$ is infinite},\\ \frac{a}{\card{S}}&\text{if $S$ is finite (and hence a group)}.\end{cases}$$
\item If $f$ is bounded and $S$ is infinite, then $\mathcal H_S(f)=0$.
\end{enumerate}
\end{lemma}

For the remainder of this section $S$ is a cancellative right amenable monoid, so every $f\in\mathcal S(S)$ is automatically bounded on singletons as observed above. 

\begin{lemma}\label{rem:decreasing}
Let $S$ be a cancellative right amenable monoid and $f\in\mathcal S(S)$.
Then \[\mathcal H_S(f) \leq f(\set 1).\]
\end{lemma}
\begin{proof}
For every $F \in \Pf(S)$, we have that
\[\frac{f(F)}{\card F} \leq \frac{\sum_{s \in F} f(\set g)}{\card F} \leq  \frac{\sum_{s \in F}  f(\set 1)}{\card F} = f(\set 1),\]
hence the thesis.
\end{proof}

For a cancellative right amenable monoid $S$, $\mathcal S(S)$ with the pointwise addition is a submonoid of the partially ordered vector space $\R^{\Pf(S)}$, so $\mathcal S(S)$ is a commutative cancellative monoid. 

Let $\hat{\mathcal S}(S)$ be the subgroup of $\R^{\Pf(S)}$ generated by $\mathcal S(S)$. 
Obviously, $\hat{\mathcal S}(S)$ is a vector space over $\R$ (actually, a subspace of $\R^{\Pf(S)}$) and $\mathcal H_S$ extends uniquely to a linear functional on $\hat{\mathcal S}(S)$, by imposing $\mathcal H_S(f_1-f_2)=\mathcal H_S(f_1)-\mathcal H_S(f_2)$ (this definition does not depend on the choice of $f_1$ and $f_2$ representing $f_1-f_2$).

Moreover, $\hat{\mathcal S}(S)$ is a partially ordered vector space with the partial order given by $$f \leq f'\quad \text{if}\ f(F) \leq f'(F)\ \text{for every}\ F \in \Pf(S),$$ and $\hat{\mathcal S}(S)$ is also endowed with the seminorm 
$$\norm f = \limsup_{F\in\Pf(S)} \frac{\abs{f(F)}}{\card{F}};$$ 
clearly, $\mathcal H_S(f) \leq \norm f$ for every $f\in\hat{\mathcal S}(S)$. For example, the function $\cardm:F\mapsto\card F$ is in $\hat{\mathcal S}(S)$, and $\Aver_S(\cardm) = \norm f = 1$.

\smallskip
Furthermore, $\mathcal H_S(f)$ behaves like an average for $f$  and its extension to $\hat{\mathcal S}(S)$ is a positive linear functional on the partially ordered vector space $\hat{\mathcal S}(S)$. 
Next, we show that this average is invariant under the right action $f \mapsto f^{F}$ of $\Pf(S)$ on $\mathcal S(S)$, where for every $f\in\mathcal S(S)$ and every $F\in\Pf(S)$,
$$f^F:\Pf(S)\to\R_+,\ X\mapsto f(XF);$$
clearly, $f^F \in\mathcal S(S)$. (This should be compared to the well-known fact that there is a right invariant mean on the  set of all  bounded real-valued functions of an amenable group $G$.)

\begin{lemma}\label{fF}
Let $S$ be a cancellative right amenable monoid, $f\in\mathcal S(S)$, and $F\in\Pf(S)$. Then 
\[\mathcal H_S(f)=\mathcal H_S(f^F).\]
\end{lemma}
\begin{proof}
Let $(F_i)_{i\in I}$ be a right F\o lner net of $S$. By Lemma~\ref{F_iF}(b) $(F_iF)_{i\in I}$ is a right F\o lner net of $S$, too. Then the assertion follows from Lemma~\ref{F_iF}(b).
\end{proof}

The next lemma shows that this average is invariant also under the obvious action of $\mathrm{Aut}(S)$ on $\mathcal S(S)$.

\begin{lemma}\label{varphi}
Let $S$ be a cancellative right amenable monoid, $\varphi:S\to S$ an automorphism, and $f\in\mathcal S(S)$. Then
\[\mathcal H_S(f)=\lim_{i\in I}\frac{f(\varphi(F_i))}{\card{F_i}} = \mathcal H_S(f \circ \varphi),\]
for any right F\o lner net $(F_i)_{i\in I}$ of $S$.
\end{lemma}
\begin{proof}
The second equality is by definition of $\mathcal H_S$.

Since $(\varphi(F_i))_{i\in I}$ is a right F\o lner net of $S$ and $\card{\varphi(F_i)}=\card{F_i}$ for every $i\in I$, we have that 
$$\mathcal H_S(f)=\lim_{i\in I}\frac{f(\varphi(F_i))}{\card{\varphi(F_i)}}=\lim_{i\in I}\frac{f(\varphi(F_i))}{\card{F_i}},$$
hence the first equality in the thesis holds.
\end{proof}

\begin{remark}\label{1in} \label{1in-bounded}
Let $S$ be a cancellative right amenable monoid, let $f\in\mathcal S(S)$ and $C\in\Pf(S)$, and denote $$\P_C(S)=\{X\in\Pf(S): C\subseteq X\}.$$
\begin{enumerate}[(a)]
  \item Assume that $f$ is bounded on $\P_C(S)$, namely $f\rest_{\P_C(S)} \leq r \in \R$. If $S$ is infinite, then $\mathcal H_S(f)= 0$. Indeed, there exists a right F\o lner net $(F_i)_{i\in I}$ of $S$ such that $C\subseteq F_i$ for every $i\in I$. Then $$\mathcal H_S(f)=\lim_{i\in   I}\frac{f(F_i)}{\card{F_i}} \leq \lim_{i\in I}\frac{r}{\card{F_i}}=0.$$ 
  \item If $f$ is constant on $\P_C(S)$, namely $f\rest_{\P_C(S)}=r\in\R$, then $$\mathcal H_S(f)=\begin{cases}0 & \text{if $S$ is infinite,}\\ \frac{r}{\card{S}} & \text{if $S$ is finite.}\end{cases}$$
Indeed, if $S$ is finite, then $\mathcal H_S(f)=\frac{f(S)}{\card{S}}=\frac{r}{\card{S}}$. If $S$ is infinite, then item (a) applies. 
\end{enumerate}
\end{remark}


Fix a surjective homomorphism $\pi: S \to C$ of cancellative monoids, let $N=\ker \pi$, and fix a good section $\sigma$ for $\pi$ with $\sigma(1) =1$. Define the map 
$$\Theta_\sigma: \mathcal S(S)\to \mathcal S(C)$$
by setting, for $f\in\mathcal S(S)$,
$$\Theta_\sigma(f)(X)=\mathcal H_N(f^{\sigma(X)})\quad \text{for all}\ X\in\Pf(C).$$ 

In the next lemma we show that indeed $\Theta_\sigma(f)\in \mathcal S(C)$ and $\Theta_\sigma$ does not depend on the choice of the good section $\sigma$. This is why most often we  write simply $\Theta$ in place of $\Theta_\sigma$, and $\theta$ in place of  $\Theta_\sigma(f)$ when the function $f$ is clear from the context. 

\begin{lemma}\label{pre-Fubini}
Let $S$ and $C$ be cancellative monoids,  $\pi: S \to C$ a surjective homomorphism with a good section $\sigma$, and assume that $C$ and $N= \pi^{-1}(1)$ are right amenable. 
Then $\Theta_\sigma$ does not depend on the choice of $\sigma$, and $\Theta_\sigma(f)\in\mathcal S(C)$ for every $f\in\mathcal S(S)$. 
\end{lemma}
\begin{proof} 
First we see that $\Theta_\sigma$ does not depend on the choice of $\sigma$.
Assume that $\sigma':C\to G$ is another good section;
we have to prove that, for every $\bar C\in\Pf(C)$,
$$\mathcal H_N(f^{\sigma'(\bar C)})=\mathcal H_N(f^{\sigma(\bar C)}).$$ 
For a fixed $\bar C\in\Pf(C)$, we verify that $\mathcal H_N(f^{\sigma'(\bar C)})\leq\mathcal H_N(f^{\sigma(\bar C)})$; the converse inequality can be proved similarly exchanging the roles of $\sigma$ and $\sigma'$. For every $c\in\bar C$ let $z_c\in N$ such that $\sigma'(c)=z_c\sigma(c)$, and let $Z=\{z_c: c\in \bar C\}$; then $\sigma'(\bar C)\subseteq Z\sigma(\bar C)$. Let $(N_i)_{i\in I}$ be a right F\o lner net of $N$. Then, applying Lemma~\ref{fF} in the last equality,
$$\mathcal H_N(f^{\sigma'(\bar C)})=\lim_{i\in I}\frac{f(N_i\sigma'(\bar C))}{\card{N_i}}\leq\lim_{i\in I}\frac{f(N_i Z\sigma(\bar C))}{\card{N_i}}=\mathcal H_N((f^{\sigma(\bar C)})^Z)=\mathcal H_N(f^{\sigma(\bar C)}).$$

\medskip 
Let $f\in\mathcal S(S)$ and put $\theta= \Theta_\sigma(f)$. We prove that $\theta\in\mathcal S(C)$. First, $\theta$ is increasing, since, if $X,X'\in \Pf(C)$ and $X\subseteq X'$, then 
$$\theta(X)=\lim_{i\in I}\frac{f^{\sigma(X)}(N_i)}{\card{N_i}}=\lim_{i\in I}\frac{f(N_i\sigma(X))}{\card{N_i}}\leq\lim_{i\in I}\frac{f(N_i\sigma(X'))}{\card{N_i}}=\lim_{i\in I}\frac{f^{\sigma(X')}(N_i)}{\card{N_i}}=\theta(X').$$
Analogously, one proves that $\theta$ is subadditive.
To verify that $\theta$ is left subinvariant, we need to prove that for $y \in C$ and $\bar C\in\Pf(C)$,
\begin{equation}\label{yC}
\theta(\bar C) \leq \theta(y\bar C).
\end{equation}
Let $(N_i)_{i\in I}$ be a right F\o lner net of $N$. By definition, $$\theta(y \bar C)=\mathcal H_N(f^{y\bar C})=\lim_{i\in I}\frac{f(N_i\sigma(y\bar C))}{\card{N_i}}.$$
By Lemma~\ref{lem:good-exchange}, there exist $u \in N$ and $Z \in \Pf(N)$ such that
\[u \sigma(y\bar C)\subseteq Z\sigma(y)\sigma(\bar C).\]
Let $h=h_{\sigma(y)}:N\to N$ be the automorphism of $N$ from Lemma~\ref{lem:good-5}. 
Let also $f'= f^{\sigma(\bar C)} \in \mathcal S(S)$. 
Therefore,
\begin{align*}
\theta(\bar C) &= \lim_{i \in I} \frac{f'(N_i)}{\card{N_i}} = & \text{by  Lemma~\ref{fF}}\\
& =\lim_{i \in I} \frac{f'(N_i h(Z))}{\card{N_i}} = & \text{by Lemmas~\ref{lem:good-5} and~\ref{varphi}}\\
&= \lim_{i \in I} \frac{f'(h(N_i) h(Z))}{\card{N_i}} = \lim_{i \in I} \frac{f'(h(N_iZ))}{\card{N_i}} \geq  & \text{since $f'$ is left subinvariant}\\
&\geq \lim_{i \in I} \frac{f'(\sigma(y) h(N_i Z))}{\card{N_i}} = & \text{by definition of } h\\
&= \lim_{i \in I} \frac{f'(N_i Z \sigma(y))}{\card{N_i}} = \lim_{i \in I} \frac{f(N_i Z \sigma(y)\sigma(\bar C))}{\card{N_i}} \ge &\text{since }  Z \sigma(y) \sigma(\bar C) \supseteq u \sigma(y \bar C) \\
&\geq  \lim_{i \in I} \frac{f(N_i u \sigma(y \bar C))}{\card{N_i}} =  \lim_{i \in I} \frac{f^{\sigma(y \bar C)}(N_i u)}{\card{N_i}} = &\text{ by Lemma~\ref{fF}}\\
&=\theta(y \bar C).
\end{align*} 
Therefore, $\theta$ is left subinvariant, hence $\theta\in\mathcal S(C)$.
\end{proof} 

Clearly, in the above notation, the obvious extension  $\hat \Theta:  \hat{\mathcal S}(S)\to \hat{\mathcal S}(C)$ of the map $ \Theta$ is a vector space homomorphism.

\subsection{Fubini's Theorem for monoids}\label{sec:Fubini-monoid}

Considering $\mathcal H_S$ introduced in the previous subsection as an integral, the next theorem reminds Fubini's Theorem.

\medskip
For $r,s\in\R$, denote 
\begin{align*}
&r \eqeps s\quad \text{if}\ \abs{r-s}\leq \eps\\
&r \leqeps s \quad \text{if}\ r \leq s + \eps.
\end{align*}

\begin{theorem}\label{Fubini} 
Let $S$ and $C$ be cancellative right amenable monoids, $\pi: S \to C$ a surjective homomorphism admitting a good section $\sigma$, and let $N= \pi^{-1}(1)$. Then, for every $f \in \mathcal S(S)$,
\begin{equation}\label{HGHC}
\mathcal H_S(f)=\mathcal H_C(\Theta_\sigma(f)).
\end{equation}
\end{theorem}
\begin{proof}
Fix $f\in\mathcal S(S)$ and put $\theta=\Theta_\sigma(f)$ for brevity. 
Let $(N_i)_{i \in \mathcal I(N)}$ and $(C_j)_{j\in\mathcal I(C)}$ be canonically indexed right F\o lner nets of $N$ and $C$ respectively (they exist in view of Proposition~\ref{canonical}). Given $i\in\mathcal I(N)$ and $j\in\mathcal I(C)$, define $$\rho(i,j)= \frac{f(N_{i} \sigma(C_j))}{\card{N_{i}} \card{C_j}}.$$
By definition, 
\begin{align*}
\mathcal H_C(\theta) &= \lim_{j\in \mathcal I(C)}\frac{\theta(C_j))}{\card{C_j}} \\
&= \lim_{j\in \mathcal I(C)}\frac{\mathcal H_C(f^{\sigma(C_j)})}{\card{C_j}} \\
&= \lim_{j\in \mathcal I(C)}  \frac{1}{\card{C_j}} \lim_{i\in \mathcal I(N)}  \frac{f(N_{i} \sigma(C_j))}{\card{N_i}} \\ 
&= \lim_{j\in \mathcal I(C)}\lim_{i\in\mathcal I(N)} \rho(i,j).
\end{align*}
Fix $\eps>0$. Then there exists $\bar i\in\mathcal I(N)$  such that, for every $i\geq\bar i$ there exists $\bar j(i)\in\mathcal I(C)$ such that, for every $j\geq\bar j(i)$,
\begin{equation}\label{eq:rho'}
\mathcal H_C(\theta)\eqeps\rho(i,j).
\end{equation}

Let $\zeta:\mathcal I(C)\to\Pf^0(N)$ and
\[\mathcal I(N,C) = \set{((X,m),(Y,n))\in \mathcal I(N)\times\mathcal I(C): m\geq n}
\subseteq \mathcal I(N)\times\mathcal I(C)\]
be as in the proof of Theorem~\ref{lem:double-limit}. Let also 
$$
\mu:\mathcal I(N,C)\to \mathcal I(N),\ ((X,m),(Y,n))\mapsto (\zeta(Y,n)\cup X,m),
$$
and note that, for every $(i,j)\in\mathcal I(N,C)$,
\begin{equation}\label{eq:mu}
\mu(i,j)\geq i.
\end{equation}
By Theorem~\ref{lem:double-limit}, $(N_{\mu(i,j)}\sigma(C_j))_{(i,j)\in\mathcal I(N,C)}$ is a right F\o lner net of $S$, so we have that
$$\mathcal H_S(f) = \lim_{(i,j)\in\mathcal I(N,C)}  \rho(\mu(i, j),j).$$
Thus, there exists $(i_0,j_0)\in\mathcal I(N,C)$ such that $i_0\geq\bar i$, $j_0\geq \bar j(i_0)$, and for every $(i,j)\in\mathcal I(N,C)$ with $(i,j)\geq (i_0,j_0)$,
\begin{equation}\label{eq:rho-G}
\mathcal H_S(f)\eqeps\rho(\mu(i,j),j).
\end{equation}
Since $\mu(i_0, j_0) \geq j_0$ by~\eqref{eq:mu}, in view of~\eqref{eq:rho'} we have that $$\mathcal H_S(f)\eqeps \rho(\mu(i_0, j_0),j_0)\eqeps \mathcal H_C(\theta).$$
Since $\eps>0$ is arbitrary, this gives Equation~\eqref{HGHC}.
\end{proof}


\begin{corollary}\label{cor:subgroup}
Let $S$ and $C$ be cancellative right amenable monoids, let $\pi: S \to C$ be a surjective homomorphism admitting a good section $\sigma$, and let $N= \pi^{-1}(1)$. If $f \in \mathcal S(S)$, then 
\[\mathcal H_S(f) \leq \mathcal H_N(f).\]
\end{corollary}
\begin{proof} 
By Theorem~\ref{Fubini}, we have
\(\mathcal H_S(f)=\mathcal H_C(\theta),\)
where $\theta=\Theta_\sigma(f) \in \mathcal S(C)$. Therefore, by Lemma~\ref{rem:decreasing}, we have that
\[\mathcal H_S(f) = \mathcal H_C(\theta) \leq \theta(\set 1) = \mathcal H_N(f^{\set 1})  = \mathcal H_N(f),\]
hence the required inequality.
\end{proof}

As mentioned in Definition~\ref{Integral:definition}, the limit defining $\mathcal H_S(f)$ does not depend on the choice of the right F\o lner net $(F_i)_{i\in I}$ of the cancellative right amenable monoid $S$, so we can assume when necessary that all members $F_i$ contain the neutral element $1$ of $S$. 

\begin{proposition}\label{prop:Fubini-restriction}
Let $S$ and $C$ be cancellative right amenable monoids, $\pi: S \to C$ a surjective homomorphism admitting a good section $\sigma$, and let $N= \pi^{-1}(1)$.
Given $f \in \mathcal S(C)$, define $$f_\pi : \Pf(S) \to \R,\ X\mapsto f(\pi(X)).$$
Then $f_\pi \in \mathcal S(S)$ and  \[\mathcal H_S(f_\pi) = \frac{\mathcal H_C(f)}{\card{N}}.\]
\end{proposition}
\begin{proof}
Let us first note that $N$ is amenable by Lemma~\ref{lem:amenable-subgroup}, therefore $\mathcal H_N(f_\pi)$ is well-defined.
It is easy to check that $f_\pi\in \mathcal S(S)$.

Let $\theta = \Theta_\sigma(f_\pi)$. 
Given $X\in\Pf(C)$, we have that
\[\theta(X) = \mathcal H_N(f_\pi^{\sigma(X)}).\]
By definition, for every $Y \subseteq N$,
\[f_\pi^{\sigma(X)}(Y) = f\Pa{\pi(\sigma(X)Y)} = f\Pa{\pi(\sigma(X))\pi(Y)} = f(\pi(\sigma(X))) = f(X)\]
does not depend on $Y$. Thus,
\[\theta(X) = \mathcal H_N(f_\pi^{\sigma(X)}) = \frac{f(X)}{\card N},\]
and, by Theorem~\ref{Fubini},
\[\mathcal H_S(f_\pi) = \mathcal H_C(\theta) =  \frac{\mathcal H_C(f)}{\card N},\]
as required.
\end{proof}

\begin{corollary}\label{prop:pre:Fubini}
Let $S$ and $C$ be cancellative right amenable monoids, let $\pi: S \to C$ be a surjective homomorphism admitting a good section $\sigma$, and consider the function 
$$\cardm_\pi: \Pf(S)\to \R_+,\ F\mapsto \card{\pi(F)}.$$ Then $\cardm_\pi\in\mathcal S(S)$ and $$\mathcal H_S(\cardm_\pi) = \frac{1}{| \pi^{-1}(1)|}.$$
\end{corollary}
\begin{proof}
Apply Proposition~\ref{prop:Fubini-restriction} to the function $\cardm \in \mathcal S(C)$. 
\end{proof}

\section{Algebraic entropy for amenable semigroup actions}\label{halgsec}

We define two notions of algebraic entropy for left actions of cancellative right amenable semigroups on discrete abelian groups. They extend respectively the algebraic entropy $\ent$ introduced by Weiss \cite{W} for endomorphisms of torsion abelian groups and the algebraic entropy $h_{alg}$ introduced in \cite{DGB0} following the work of Peters \cite{P1} for endomorphisms of abelian groups.

For amenable group actions on discrete abelian groups our definition of algebraic entropy coincides with that given in \cite{V2} for locally compact abelian groups.

\subsection{Definitions}

Let $S$ be a cancellative right amenable semigroup, $A$ an abelian group, and consider the left action $S\overset{\alpha}{\curvearrowright} A$. For a non-empty subset $X$ of $A$ and for every $F\in\Pf(S)$, let 
$$T_F(\alpha,X)=\sum_{s\in F}\alpha(s)(X)=\sum_{s\in F}s\cdot X$$
be the \emph{$\alpha$-trajectory of $X$ with respect to $F$}. 
Note that $T_F(\alpha,X)$ is finite, whenever $X$ is finite. 

When there is no danger of confusion we simply write $T_F(X)$ in place of $T_F(\alpha,X)$.

\medskip
For $X\in \Pf(A)$, consider the function $$f_X:\Pf(S)\to \R,\quad F\mapsto \ell(T_F(\alpha,X)).$$
Note that $f_X(F)=0$ for all $F\in \Pf(S)$ whenever $X\subseteq A$ is a singleton (as $T_F(\alpha,X)$ is a singleton as well).

In the next lemma we see, in particular, that  $f_X$ is subadditive for every $X\in\Pf(A)$.

\begin{lemma}\label{OWconditions} 
Let $S$ be a cancellative right amenable semigroup, $A$ an abelian group, $S\overset{\alpha}{\curvearrowright} A$ a left action, and $X\in\Pf(A)$.
Then $f_X\in\mathcal S(S)$. 
\end{lemma}
\begin{proof}
First we verify that $f_X$ is increasing. Let $F,F'\in\Pf(S)$ with $F\subseteq F'$ and
$F\neq F'$.
Since $L := T_{F'\setminus F}(\alpha,X) \ne \emptyset$, we have that
$$|T_F(\alpha,X)| = |T_F(\alpha,X) + l| \leq |T_F(\alpha,X)  + L|=\card{T_{F'}(\alpha,X)} $$ for every $l \in L$, as $T_F(\alpha,X) + l \subseteq T_F(\alpha,X)  + L$.  

\smallskip
Let $F_1,F_2\in\Pf(S)$.  In case $F_1\cap F_2 = \emptyset$,
$$T_{F_1\cup F_2}(X)=\sum_{s\in F_1\cup F_2}\alpha(s)(X)= \sum_{s\in F_1} \alpha(s)(X)+\sum_{s\in F_2}\alpha(s)(X)=T_{F_1}(X)+T_{F_2}(X).$$ 
In the general case, let  $F_2^*= F_2 \setminus F_1$, so that $F_1$ and $F_2^*$ are disjoint, yet we have the same union $F_1 \cup F_2$. By the first case, 
$$ T_{F_1 \cup F_2}(X) = T_{F_1 \cup F_2^*}(X) = T_{F_1}(X) + T_{F_2^*}(X). $$ 
Therefore, 
$$f_X({F_1 \cup F_2})=\ell(T_{F_1\cup F_2^*}(X))=\ell(T_{F_1}(X)+T_{F_2^*}(X))\leq \ell(T_{F_1}(X))+\ell(T_{F_2^*}(X)) =f_X(F_1) + f_X(F_2^*).$$
As the function $f_X$ is increasing, we have that $f_X(F_2^*) \leq   f_X(F_2)$.  This proves the desired inequality 
$$f_X({F_1 \cup F_2}) \leq  f_X(F_1) + f_X(F_2),$$  i.e.,  $f_X$ is subadditive.

\smallskip
Let now $F\in\Pf(S)$ and $s\in S$. Then 
$$T_{sF}(X) = \sum_{f\in F}\alpha(sf)(X)=\alpha(s)\left(\sum_{f\in F}\alpha(f)(X)\right)=\alpha(s)(T_{F}(X)),$$
and so $$f_X(sF)=\ell(T_{sF}(X))= \ell(\alpha(s)(T_{F}(X))) \leq \ell(T_{F}(X))=f_X(F).$$ Therefore, $f_X$ is left subinvariant.

\smallskip
Finally, for every $s\in S$, $$f_X(\{s\})=\log|\alpha(s)(X)|\leq\log|X|,$$ so $f_X$ is uniformly bounded on singletons.
\end{proof}

In view of Lemma~\ref{OWconditions}, by applying Corollary~\ref{CCKLemmar}, we can give the following definition.

\begin{definition}\label{def:entropy}
Let $S$ be a cancellative right amenable semigroup, $A$ an abelian group, and $S\overset{\alpha}{\curvearrowright}A$ a left action. For $X\in\Pf(A)$, the \emph{algebraic entropy of $\alpha$ with respect to $X$} is 
$$H_{alg}(\alpha,X)=\mathcal H_S(f_X).$$
(In other words, $H_{alg}(\alpha,X)=\lim_{i\in I}\frac{\ell(T_{F_i}(\alpha,X))}{|F_i|},$ where $(F_i)_{i\in I}$ is a right F\o lner net of $S$.)

\smallskip
The \emph{algebraic entropy of $\alpha$} is 
$$h_{alg}(\alpha)=\sup\{H_{alg}(\alpha,X): X\in \Pf(A)\}.$$
Let also $$\ent(\alpha)=\sup\{H_{alg}(\alpha,X): X\in\mathcal F(A)\}.$$
\end{definition}

The definition of $H_{alg}(\alpha,X)$ does not depend on the choice of the right F\o lner net $(F_i)_{i\in I}$ in view of Corollary~\ref{CCKLemmar}.
Moreover, for every $X\in\Pf(A)$, $H_{alg}(\alpha,X) \leq \ell(X)$ is bounded, according to Lemma~\ref{rem:decreasing}. 


\begin{remark}\label{SSop} One can introduce the algebraic entropy also for right actions of cancellative left amenable semigroups on abelian groups. 
(see also Remark~\ref{GGop} for the case of right actions of amenable groups). Namely, consider a right action $$A\overset{\beta}{\curvearrowleft}S$$ of a cancellative left amenable semigroup $S$ on an abelian group $A$.
To define the algebraic entropy of $\beta$, 
consider the left action $\beta^{op}$ of the cancellative right amenable semigroup $S^{op}$ on $A$ (see Remark~\ref{Sop}) and let 
$$h_{alg}^r(\beta)= h_{alg}(\beta^{op}).$$
\end{remark}

Now we see  that the function $H_{alg}(\alpha,-)$ is monotone increasing. As a consequence of this fact, it is possible to restrict the family of finite subsets of $A$ on which we compute the algebraic entropy. In particular, we can always assume that $X\in\Pf^0(A)$.

\begin{lemma}\label{cofinal}
Let $S$ be a cancellative right amenable semigroup, $A$ an abelian group, and $S\overset{\alpha}{\curvearrowright}A$ a left action.
\begin{enumerate}[(a)]
\item If $X,Y\in \Pf(A)$ and $X\subseteq Y$, then $H_{alg}(\alpha,X)\leq H_{alg}(\alpha,Y)$.
\item If $\mathcal F\subseteq\Pf(A)$ is cofinal with respect to $\subseteq$, then $h_{alg}(\alpha)=\sup\{H_{alg}(\alpha,X): X\in\mathcal F\}.$
\end{enumerate}
\end{lemma}
\begin{proof} 
(a) is clear from the definition and (b) follows from (a).
\end{proof}

Since the torsion part $t(A)$ of an abelian group $A$ is a fully invariant subgroup of $A$, and so in particular $t(A)$ is $\alpha$-invariant,  it clearly follows from the definition that 
\begin{equation}\label{entt}
\ent(\alpha)=\ent(\alpha_{t(A)}).
\end{equation}

In view of Lemma~\ref{cofinal} we clarify the relation between $h_{alg}$ and $\ent$.

\begin{proposition}\label{h=ent}
Let $S$ be a cancellative right amenable semigroup, $A$ an abelian group, and $S\overset{\alpha}{\curvearrowright}A$ a left action. Then 
$$\ent(\alpha)=\ent(\alpha_{t(A)})=h_{alg}(\alpha_{t(A)}).$$
\end{proposition}
\begin{proof}
In view of Equation~\eqref{entt}, it suffices to prove that, if $A$ is torsion, then $\ent(\alpha)=h_{alg}(\alpha)$. This is true since in this case $\mathcal F(A)$ is cofinal in $\Pf^0(A)$, and so Lemma~\ref{cofinal} applies.
\end{proof}

In the next remark we see that for $S=\N$ we find the classical case of the algebraic entropy of a single endomorphism. This is why we keep the same notation.

\begin{remark}\label{Naction}
Assume that $A$ is an abelian group and fix an endomorphism $\phi:A\to A$. Then $\phi$ induces the action $\alpha_\phi$ of $\N$ on $A$ defined by $\alpha_\phi(n)=\phi^n$ for every $n\in\N$.
In \cite{DGB0}, the algebraic entropy of $\phi$ is defined exactly as $$h_{alg}(\phi)=h_{alg}(\alpha_\phi),$$ using the special right F\o lner sequence $(F_n)_{n\in\N_+}$ of $\N$ with $F_n=[0,n-1]$ for every $n\in\N_+$, and consequently the $n$-th $\phi$-trajectories $$T_n(\phi,X)=T_{F_n}(\alpha_\phi,X)$$ for every $X\in\Pf(A)$.
\end{remark}



On the other hand, there is a relevant difference with respect to the case of a single endomorphism when the acting semigroup $S$ is cyclic and finite. In fact, if $A$ is an abelian group and $\phi:A\to A$ is an endomorphism such that $\phi^n=\phi^m$ for some distinct $n,m\in\N$, then $h_{alg}(\phi)=0$ (see \cite{DGB,DGB0}).
This is no more true in general for the algebraic entropy defined in this paper, as we see in item (a) of the next example.

\begin{example}\label{Esempio:AG} 
Let $S$ be a cancellative right amenable monoid, $A$ an abelian group, and $S\overset{\alpha}{\curvearrowright}A$ a left action.
\begin{enumerate}[(a)]
\item If $S$ is finite,
then $$h_{alg}(\alpha)=\frac{\ell(A)}{\card{S}}\quad\text{and}\quad\ent(\alpha)=\frac{\ell(t(A))}{\card{S}};$$
in particular, $h_{alg}(\alpha)=\infty$ and $\ent(\alpha)=\infty$ whenever  $t(A)$ is infinite.

Indeed, take the constant right F\o lner sequence $(S)_{n\in\N}$ of $S$. If $A$ is finite, since $T_S(\alpha,A)=A$, by Lemma~\ref{cofinal}(a)  and Proposition~\ref{h=ent} we get that
$$h_{alg}(\alpha)=H_{alg}(\alpha,A)=\frac{\ell(T_S(\alpha,A))}{|S|}=\frac{\ell(A)}{|S|} \ \mbox{ and }\ \ent(\alpha) = \frac{\ell(t(A))}{|S|}.$$
If $A$ is infinite, for every $X\in\Pf^0(A)$, since $X\subseteq T_S(\alpha,X)$, we have that
\begin{equation*}\label{dag}
H_{alg}(\alpha,X)=\frac{\ell(T_S(\alpha,X))}{|S|}\geq \frac{\ell(X)}{|S|}.
\end{equation*}
Then $h_{alg}(\alpha)\geq\sup\left\{\frac{\ell(X)}{|S|}: X\in\Pf^0(A)\right\}$ is infinite since $A$ is infinite.
Similarly, $\ent(\alpha)$ is infinite whenever $t(A)$ is infinite. 

In particular, if $S=\{e\}$, then $H_{alg}(\alpha,X)=\ell(X)$ for every $X\in\Pf^0(A)$.

%
%

\item Consider the trivial left monoid action $S\overset{\tau}{\curvearrowright}A$, defined by $\tau(s)=id_A$ for every $s\in S$. If $S$ is infinite, then $h_{alg}(\tau)=0$. 

Indeed, for a right F\o lner net $(F_i)_{i\in I}$ of $S$ and $X\in\Pf(A)$, one has $$T_{F_i}(\tau,X)= X + \ldots + X.$$ So, 
$$f_X(F_i) \leq |X|\log (1+ |F_i|).$$ This yields $$H_{alg}(\tau,X) \leq \lim_{i\in I} \frac{|X|\log (1+ |F_i|)}{|F_i|}= 0,$$ as $A$ is infinite and so $\lim_{i\in I} |F_i| = \infty$.
\end{enumerate}
\end{example}

In item (a) of the above example, we assume $S$ to be a finite monoid and $S\overset{\alpha}{\curvearrowright}A$ a left monoid action, otherwise the conclusion could be false. Indeed, consider the action such that $\alpha(s)=0_A$ for every $s\in S$, which is a left semigroup action but not a left monoid action. Then, $T_S(\alpha,X)=\{0\}$ for every $X\in\Pf^0(A)$, so $h_{alg}(\alpha)=\sup\{H_{alg}(\alpha,X): X\in\Pf^0(A)\}=0$.

\begin{remark}\label{S->Gh}
Let $G$ be a group and $S$ a cancellative right amenable submonoid of $G$ that generates $G$ as a group; by Lemma~\ref{S->G}(a) the group $G$ is necessarily amenable. Let $A$ be an abelian group and consider the left action $G\overset{\alpha}\curvearrowright A$. Then $$h_{alg}(\alpha)=h_{alg}(\alpha\restriction_S).$$
In fact, by Lemma~\ref{S->G}(b), a right F\o lner net $(F_i)_{i\in I}$ of $S$ is also a right F\o lner net of $G$. So, for every $X\in\Pf^0(A)$, we have that $H_{alg}(\alpha,X)=H_{alg}(\alpha\restriction_S,X)$ by definition.
\end{remark}

\begin{remark}
Let $G$ be an amenable group, $A$ an abelian group, and $G\overset{\alpha}{\curvearrowright}A$ a left action. Consider $$\ker\alpha=\{g\in G: \alpha(g)=id_A\}.$$ 
It should be natural to expect that the induced action $G/\ker\alpha\overset{\bar\alpha}{\curvearrowright}A$ would have the same algebraic entropy of $G\overset{\alpha}{\curvearrowright}A$. Actually, this is not the case in view of the above examples. Indeed, $h_{alg}(\tau)=0$, while $\ker\tau=G$, so $G/\ker\tau=\{1\}$ and hence $h_{alg}(\bar\tau)=\infty$ whenever $A$ is infinite. 

\smallskip
For a non-trivial example, in which $G/\ker\alpha$ is infinite, witnessing that $h_{alg}(\alpha)$ is not equal to $h_{alg}(\bar\alpha)$, see Example~\ref{gamma=alpha} as well as the general Theorem~\ref{quotient}.
\end{remark}

\subsection{Basic properties}

We start showing that $h_{alg}$ coincides for weakly conjugated actions, defined as follows:  

\begin{definition}\label{New:Def}
For cancellative right amenable semigroups $S$ and $T$, abelian groups $A$ and $B$, and left actions $S\overset{\alpha}{\curvearrowright} A$ and $T\overset{\beta}{\curvearrowright}B$, we say that $\alpha$ and $\beta$ are \emph{weakly conjugated} if there exist an isomorphism $\eta:S\to T$ and an isomorphism $\xi:A\to B$ such that 
\begin{equation}\label{conjeq}
\xi\circ\alpha(s)=\beta(\eta(s))\circ\xi
\end{equation} 
for every $s\in S$.
\end{definition}

Our leading example is when $S=T$ and $\eta = id_S$, yet Lemma~\ref{alphaeta} provides a relevant instance when $S=T$ is an amenable group, yet $\eta$ is an arbitrary automorphism of $S$.  The next proposition justifies our attention to weak conjugacy. 

\begin{proposition}\label{conju}
Let $S,T$ be cancellative right amenable semigroups, $A,B$ abelian groups, and $S\overset{\alpha}{\curvearrowright} A$, $T\overset{\beta}{\curvearrowright}B$ left actions. If $\alpha$ and $\beta$ are weakly conjugated, then $h_{alg}(\alpha)=h_{alg}(\beta)$.
\end{proposition}
\begin{proof}
Let $(F_i)_{i\in I}$ be a right F\o lner net of $S$, and let $X\in\Pf(A)$. Then $(\eta(F_i))_{i\in I}$ is a right F\o lner net of $T$.  For every $i\in I$, 
$$T_{\eta(F_i)}(\beta,\xi(X))=\sum_{s\in F_i}\beta(\eta(s))(\xi(X))=\sum_{s\in F_i}\xi(\alpha(s)(X))=\xi\left(\sum_{s\in F_i}\alpha(s)(X)\right)=\xi(T_{F_i}(\alpha,X)).$$ Hence, $\ell(T_{\eta(F_i)}(\beta,\xi(X)))=\ell(T_{F_i}(\alpha,X))$, and so $H_{alg}(\beta,\xi(X))=H_{alg}(\alpha,X)$. Since $\xi$ induces a bijection between $\Pf(A)$ and $\Pf(B)$, this implies $h_{alg}(\alpha)=h_{alg}(\beta)$.
\end{proof}

\begin{remark}\label{wconj}
Since the only automorphism of $\N$ is $id_\N$, two left actions $\N\overset{\alpha}{\curvearrowright} A$ and $\N\overset{\beta}{\curvearrowright}B$ on abelian groups $A$ and $B$ are weakly conjugated precisely when $$\xi\circ\alpha(n)=\beta(n)\circ\xi$$ for all $n\in \N$. In other words, letting $\phi = \alpha(1)$ and $\psi = \beta(1)$, one has 
\begin{equation}\label{Wconjeq}
\phi=\xi^{-1}\circ\psi\circ\xi. 
\end{equation} 

Usually, a pair of endomorphisms $\phi:A\to A$ and $\psi:B\to B$ are said to be \emph{conjugated} if the condition in Equation~\eqref{Wconjeq} holds for some isomorphism $\xi:A\to B$ (see \cite{DGB0}). Clearly, a pair of conjugated endomorphisms  $\phi:A\to A$ and $\psi:B\to B$ gives rise to weakly conjugated $\N$-actions $\N\overset{\alpha_\phi}{\curvearrowright}A$ and $\N\overset{\alpha_\psi}{\curvearrowright}B$. Indeed, $\phi=\xi^{-1}\circ\psi\circ\xi$ implies $\phi^n=\xi^{-1}\circ\psi^n\circ\xi$ for every $n\in\N$, that is, $\xi\circ\phi^n=\psi^n\circ\xi$ for every $n\in\N$; in terms of actions this means that $\xi\circ\alpha_\phi(n)=\alpha_\psi(n)\circ\xi$ for every $n\in\N$, that is Equation~\eqref{conjeq} is satisfied with $\eta=id_\N$. 

It is known that $h_{alg}(\phi)=h_{alg}(\psi)$ whenever the endomorphisms $\phi$ and $\psi$ are conjugated (see \cite{DGB0}). 
\end{remark}

On the other hand, weakly conjugated automorphisms of $A$, viewed as $\Z$-actions, need not be conjugated in the sense of Equation~\eqref{Wconjeq} (see Remark~\ref{New:RmerK}). 

\smallskip
In the case of an automorphism $\phi:A\to A$ of an abelian group $A$, it is known that $$h_{alg}(\phi)=h_{alg}(\phi^{-1}).$$
In order to obtain this equality as a consequence of Proposition~\ref{conju} in Remark~\ref{New:RmerK}, we see that $\phi$ and ${\phi^{-1}}$ are weakly conjugated. To this end we consider, more generally, an action $G\overset{\alpha}\curvearrowright A$ of an amenable group $G$ on an abelian group $A$, and an automorphism $\eta:G\to G$. Define the action $$G\overset{\alpha_\eta}{\curvearrowright}A$$ by $\alpha^\eta(g)=\alpha(\eta(g))$ for every $g\in G$.

\medskip
We omit the immediate proof of the next lemma. 

\begin{lemma}\label{alphaeta}
Let $G$ be an amenable group, $A$ an abelian group, and $G\overset{\alpha}\curvearrowright A$ a left action. 
If $\eta:G\to G$ is an automorphism, then $\alpha$ and $\alpha^\eta$ are weakly conjugated and $h_{alg}(\alpha)=h_{alg}(\alpha^\eta)$.
\end{lemma}

\begin{remark}\label{New:RmerK} 
Let $G$ be an abelian group, $\eta=-id_G\in \Aut(G)$, and $G\overset{\alpha}\curvearrowright A$ a left action. 
Then $\alpha^\eta$ is now given by $\alpha^\eta(g)= \alpha(g)^{-1}$ for every $g\in G$, and it is weakly conjugated to $\alpha$, so $h_{alg}(\alpha)=h_{alg}(\alpha^\eta)$ by Lemma~\ref{alphaeta}. In particular, for $G=\Z$ we obtain $h_{alg}(\phi)=h_{alg}(\phi^{-1})$ for any automorphism $\phi:A\to A$. 

To finish the comparison between conjugacy and weak conjugacy in the realm of automorphisms, consider the automorphism $m_2:\Q\to \Q$, defined by $x\mapsto 2x$, inducing an obvious $\Z$-action on $\Q$, which is weakly conjugated to the action induced by $(m_2)^{-1}=m_{\frac{1}{2}}$. Nevertheless, $m_2$ is not conjugated to $(m_2)^{-1}$ in the sense of Equation~\eqref{Wconjeq}.
\end{remark}

Resuming Remark~\ref{wconj} and Remark~\ref{New:RmerK}, the notion of weak conjugacy provides a convenient umbrella covering two relevant cases of coincidence of algebraic entropies: of a pair of conjugated endomorphisms, or a pair of two mutually inverse automorphisms.

\medskip
Using Proposition~\ref{conju} we can specify the relation between the algebraic entropy of $\alpha(g)$ and the algebraic entropy of the restriction of $\alpha$ to the semigroup generated by $g$.

\begin{remark}\label{mayday}
Let $S$ be a cancellative right amenable monoid, $A$ an abelian group, and $S\overset{\alpha}{\curvearrowright}A$ a left action. Let $g\in S$ and let $T$ be the submonoid of $S$ generated by $g$.
\begin{enumerate}[(a)]
\item If $T$ is infinite, then $T\cong \N$ and $\alpha\restriction_{T}$ is weakly conjugated to the action $\alpha_{\alpha(g)}$ in the notation of Remark~\ref{Naction}. 
So, $h_{alg}(\alpha(g))=h_{alg}(\alpha\restriction_{T})$ by Remark~\ref{Naction} and Proposition~\ref{conju}. 
\item 
If $T$ is finite, then a straightforward computation shows that $h_{alg}(\alpha(g))=0$, while $h_{alg}(\alpha\restriction_{T})$ is always positive (actually $\infty$ if $A$ is infinite) by Example~\ref{Esempio:AG}(a). 
\end{enumerate}
The same properties hold when $S$ is a group and $T=\langle g\rangle$ is the subgroup of $S$ generated by $g$.
\end{remark}

The following observation using Proposition~\ref{conju} is related to Remark~\ref{SSop}.

\begin{remark}\label{GGop} 
In case of a right action $A\overset{\beta}{\curvearrowleft}G$ of an amenable group $G$ acting on an abelian group $A$, one has also an alternative option to define a left action $\beta'$ that leads to the same algebraic entropy $h_{alg}^r$. Indeed, define the left action $\beta'$ by putting $\beta'(g)=\beta(g^{-1})$ for every $g\in G$. 
The left actions $$G^{op}\overset{\beta^{op}}{\curvearrowright}A\quad \text{and}\quad G\overset{\beta'}{\curvearrowright}A$$ are conjugated since the map $\iota: g \mapsto g^{-1}$ provides a group isomorphism $G \to G^{op}$. By Proposition~\ref{conju}, $$h_{alg}^r(\beta)=h_{alg}(\beta^{op})=h_{alg}(\beta').$$
This shows that we can use both ways to pass from right actions to left ones obtaining the same result. 
\end{remark}


Now we consider the monotonicity of $h_{alg}$ with respect to invariant subgroups and quotients over invariant subgroups.

\begin{proposition}\label{restriction_quotient} 
Let $S$ be a cancellative right amenable semigroup, $A$ an abelian group, and $S\overset{\alpha}\curvearrowright A$ a left action. If $B$ is an $\alpha$-invariant subgroup of $A$, then
$$h_{alg}(\alpha)\geq \max\{h_{alg}(\alpha_B),h_{alg}(\alpha_{A/B})\}.$$
\end{proposition}
\begin{proof} 
For every $X\in\Pf(B)$,  $$H_{alg}(\alpha_B,X)=H_{alg}(\alpha,X),$$ so $H_{alg}(\alpha_B, F)\leq h_{alg}(\alpha)$. Hence, $h_{alg}(\alpha_B)\leq h_{alg}(\alpha)$.

Assume that $X\in\Pf(A/B)$ and $X=\pi(X_0)$ for some $X_0\in\Pf(A)$, where $\pi:A\to A/B$ is the canonical projection. Let $(F_i)_{i\in I}$ be a right F\o lner net of $S$. Then, for every $i\in I$, 
$$\pi(T_{F_i}(\alpha,X_0))=\sum_{s\in F_i}\pi(\alpha(s)(X_0))=\sum_{s\in F_i}\alpha_{A/B}(s)(\pi (X_0))=T_{F_i}(\alpha_{A/B},X).$$ 
It follows that $h_{alg}(\alpha)\geq H_{alg}(\alpha,X_0)\geq H_{alg}(\alpha_{A/B}, X)$, and hence $h_{alg}(\alpha)\geq h_{alg}(\alpha_{A/B})$.
\end{proof}

Next we verify the continuity for direct limits.

\begin{proposition}\label{contlim}
Let $S$ be a cancellative right amenable semigroup, $A$ an abelian group, and $S\overset{\alpha}\curvearrowright A$ a left action. If $A$ is a direct limit of $\alpha$-invariant subgroups $\{A_i: i\in I\}$, then $h_{alg}(\alpha)=\sup_{i\in I}h_{alg}(\alpha_{A_i})$.
\end{proposition}
\begin{proof}
By Proposition~\ref{restriction_quotient}, $h_{alg}(\alpha)\geq h_{alg}(\alpha_{A_i})$ for every $i\in I$ and so $h_{alg}(\alpha)\geq\sup_{i\in I}h_{alg}(\alpha_{A_i})$.

To check the converse inequality, let $X\in\Pf(A)$. Since $A=\varinjlim\{A_i: i\in I\}$ and $\{A_i: i\in I\}$ is a directed family, there exists $j\in I$ such that $X\subseteq A_j$. Then 
$$H_{alg}(\alpha,X)=H_{alg}(\alpha_{A_j},X)\leq h_{alg}(\alpha_{A_j}).$$ This proves that $h_{alg}(\alpha)\leq\sup_{i\in I}h_{alg}(\alpha_{A_j})$.
\end{proof}

The following is a basic instance of the Addition Theorem.

\begin{proposition}\label{wAT}
Let $S$ be a cancellative right amenable semigroup, $A$ an abelian group, and $S\overset{\alpha}\curvearrowright A$ a left action. If $A=A_1\times A_2$, with $A_1$, $A_2$ $\alpha$-invariant subgroups of $A$, then $h_{alg}(\alpha)=h_{alg}(\alpha_{A_1})+h_{alg}(\alpha_{A_2})$.
\end{proposition}
\begin{proof}
Note that, $\alpha(s)=\alpha_{A_1}(s)\times\alpha_{A_2}(s)$ for every $s\in S$.

Let $(F_i)_{i\in I}$ be a right F\o lner net of $S$. For every $i\in I$, and for $X_1\in\Pf(A_1)$, $X_2\in\Pf(A_2)$,
\[\begin{split} T_{F_i}(\alpha,X_1\times X_2)=\sum_{s\in F_i}(\alpha_{A_1}(s)(X_1)\times \alpha_{A_2}(s)(X_2))&=\\=\left(\sum_{s\in F_i}\alpha_{A_1}(s)(X_1)\right)\times \left(\sum_{s\in  F_i}\alpha_{A_2}(s)(X_2)\right)&=T_{F_i}(\alpha_{A_1},X_1)\times T_{F_i}(\alpha_{A_2}, X_2).\end{split}\]
Hence, 
\begin{equation}\label{times-eq}
H_{alg}(\alpha,X_1\times X_2)=H_{alg}(\alpha_{A_1},X_1)+H_{alg}(\alpha_{A_2},X_2).
\end{equation} 
Consequently, $h_{alg}(\alpha)\geq h_{alg}(\alpha_{A_1})+h_{alg}(\alpha_{A_2})$. 

Since $\{X_1\times X_2: X_i\in\Pf(A_i),i=1,2\}$ is cofinal in $\Pf(A)$, in view of Lemma~\ref{cofinal}, Equation~\eqref{times-eq} proves also that $h_{alg}(\alpha)\leq h_{alg}(\alpha_{A_1})+h_{alg}(\alpha_{A_2})$.
\end{proof}

\subsection{Properties of the trajectories}

\begin{lemma}\label{lemma?}
Let $S$ be a semigroup, $A$ an abelian group, and $S\overset{\alpha}\curvearrowright A$ a left action.
If $F\in\Pf(S)$ and $X,X'\in\Pf^0(A)$, then:
\begin{enumerate}[(a)]
\item $T_F(X + X') = T_F(X)+T_F(X')$;
\item $T_F(X)\cup T_F(X')\subseteq T_F(X\cup X')\subseteq T_F(X)+T_F(X')$.
\end{enumerate}
\end{lemma}
\begin{proof} 
(a) We have
\begin{equation*}\begin{split}
T_F(X + X')=\set{\sum_{f\in F}\alpha(f)(x_f+x'_f)\colon x_f\in X, x'_f\in X'}=\\ =\set{\sum_{f \in F}\alpha(f)(x_f)+\sum_{f \in F}\alpha(f)(x'_f)\colon x_f\in X,x'_f\in X'}=T_F(X)+T_F(X').
\end{split}
\end{equation*}

(b) can be proved analogously to (a).
\end{proof}

Recall, that if $W$ is a subset of an abelian group $A$ and $m\in\N_+$, we use the notation $$W_m=\underbrace{W+W+\ldots+W}_m.$$

\begin{lemma}\label{next-to-the-last}
Let $S$ be a right amenable semigroup, $A$ an abelian group, $S\overset{\alpha}\curvearrowright A$ a left action, $F\in\Pf(S)$, and $B\in\Pf(A)$. Then, for every $m\in\N_+$, $$T_{F}(\alpha,B_m)=T_F(\alpha,B)_m.$$
\end{lemma}
\begin{proof}
It suffices to compute that
$$T_{F}(\alpha,B_m)= \sum_{f\in F}\alpha(f)(B_m) =  \sum_{f\in F}\alpha(f)(B)_m=\left( \sum_{f\in F}\alpha(f)(B)\right)_m=  T_F(\alpha,B)_m$$
for every $m\in\N_+$.
\end{proof}

\begin{lemma}\label{TEF}
Let $S$ be a right amenable semigroup, $A$ an abelian group, $S\overset{\alpha}\curvearrowright A$ a left action, and $F,F'\in\Pf(S)$.  
If $X$ is a subset of $A$, and $X'=X_{\card{F'}}$, then  
\begin{equation}\label{Neew}
T_{FF'}(X)\subseteq T_F(T_{F'}(X))\subseteq T_{FF'}(X').
\end{equation}
If $X$ is a subgroup of $A$, then $X' = X$ and $T_{FF'}(X)= T_F(T_{F'}(X)).$
\end{lemma}
\begin{proof} For a subset $X$ of $A$, we have that
$$T_{FF'}(X)=\sum_{g\in FF'}\alpha(g)(X)$$
and
$$
T_F(T_{F'}(X))=\sum_{f\in F}\alpha(f)\left(\sum_{f'\in F'}\alpha(f')(X)\right)=\sum_{f\in F}\sum_{f'\in F'}\alpha(f)(\alpha(f')(X))=\sum_{f\in F,f'\in F'}\alpha(ff')(X).
$$
Hence, $T_{FF'}(X)\subseteq T_F(T_{F'}(X))$.

To prove the second containment in Equation~\eqref{Neew} we use as above that $T_F(T_{F'}(X))=\sum_{f\in F,f'\in F'}\alpha(ff')(X).$ Since 
$$T_{FF'}(X')=\sum_{g\in FF'}\alpha(g)(X')=\sum_{g\in FF'}\alpha(g)(X)_{\card{F'}},$$
and for every $g\in FF'$, by the right cancellation property of $S$,
$$\card{V_g=\set{(f,f')\in F\times F': ff'=g}}\leq \card{F'},$$ we can conclude that $T_{F}(T_{F'}(X))\subseteq T_{FF'}(X')$.

If $X$ is a subgroup of $A$, then $X' = X$, so Equation~\eqref{Neew} yields $T_{FF'}(X)= T_F(T_{F'}(X)).$
\end{proof}

The following examples show that the containments in the above lemma can be strict.

\begin{example} Let $A = \Q$.
\begin{enumerate}[(a)]
\item Let $S = \Z$ and consider the left action $S\overset{\alpha}\curvearrowright A$ defined by $\alpha(n)(q) = 2^n q$ for every $n\in\Z$ and every $q\in\Q$.
Let $F = F' = \set{0,1}=X$.  Then $T_F(X)=\set{0,1,2,3}$ and $FF=\set{0,1,2}$, so
$$T_{FF}(X)=\set{0,\ldots,7}\subsetneq T_F(T_F(X))=\set{0,\ldots,9}\subsetneq T_{FF}(X') = \set{0,1, \ldots,14}.$$
\item Let $S= \mathrm{Aut}(\Q)$, and consider the left action $S\overset{\alpha}\curvearrowright A$ defined by $\alpha(\phi)(q)=\phi(q)$ for every $\phi\in S$ and every $q\in\Q$.
Let $F=F'= \set{\pm id_\Q}$ and $X=\set{0,\pm 1}$. Then $T_F(X)=X+X=\set{0,\pm 1,\pm 2}$ and $FF=F$, hence
$$T_{FF}(X)=T_F(X)=\set{0,\pm1,\pm2}\subsetneq T_F(T_F(X))=T_F(X)+T_F(X)=\set{0,\pm1,\pm2,\pm3,\pm4}.$$
\end{enumerate}
\end{example}

\begin{lemma}\label{HX=HTX}
Let $S$ be a cancellative right amenable monoid, $A$ an abelian group, and $S\overset{\alpha}\curvearrowright A$ a left action.
If $X\in\mathcal F(A)$ and $F \in \Pf^0(S)$, then $$H_{alg}(\alpha,X)=H_{alg}(\alpha,T_F(\alpha,X)).$$
\end{lemma}
\begin{proof} 
Since $X\subseteq T_F(X)$, we have that $H_{alg}(\alpha,X)\leq H_{alg}(\alpha, T_F(X))$ by Lemma~\ref{cofinal}(a).

To prove the converse inequality, let $(F_i)_{i \in I}$ be a right F\o lner net of $S$ and assume without loss of generality that $1\in F_i$ for every $i\in I$. By Lemma~\ref{TEF}, we have that $$H_{alg}(\alpha, T_F(X)) = \lim_{i\in I} \frac{\ell(T_{F_i}(T_F(X)))}{\card{F_i}}=\lim_{i\in I} \frac{\ell(T_{F_i F}(X))}{\card{F_i}}.$$

For every $i\in I$, let $D_i= (F_i F) \setminus F_i$. Fix $\eps > 0$. Since $(F_i)_{i \in I}$ is a right F\o lner net of $S$, $\lim_{i\in I}\frac{\card{D_i}}{\card{F_i}}=0$ by Lemma~\ref{F_iF}(b), and so there exists $J\subseteq I$ cofinal such that $\ell(X) \card{D_j} \leq \eps \card{F_j}$ for every $j\in J$. Thus, for every $j\in J$,
$$\ell(T_{F_j F} (X))  \leq \ell(T_{F_j}(X)) + \ell(T_{D_j}(X)) \leq \ell(T_{F_j}(X)) + \card{D_j} \ell(X) \leq \ell(T_{F_j}(X)) + \eps \card{F_j},$$
and so
$$\frac{\ell(T_{F_j F}(X))}{\card{F_j}} \leq \frac{\ell(T_{F_j}(X))}{\card{F_j}} + \eps.$$
Therefore, we get $$H_{alg}(\alpha, T_{F}(X)) \leq H_{alg}(\alpha, X) + \eps.$$ Since the above inequality is true for every $\eps > 0$, we have the conclusion.
\end{proof}

The above lemma does not hold true in case $X$ is not a subgroup of $A$:

\begin{example}\label{4x} 
Let $S=\N$, $A=\Z$, and consider the left action $S\overset{\alpha_\phi}{\curvearrowright}A$, where $\phi=m_4:\Z\to \Z$ is defined by $\phi(x)=4x$ for every $x\in\Z$. It is known (and straightforward to prove) that $h_{alg}(\phi)=\log4$ (see \cite{DGB0}).

Let $X=\set{0,1}$. Then $T_2(\phi,X)=X+4X=\set{0,1,4,5}$ and $$T_n(\phi,X)=\set{a_0+a_14+\ldots+a_{n-1}4^{n-1}: a_i\in\set{0,1}}$$ for every $n\in\N_+$.
Then $$\bigcup_{n\in\N_+}T_n(\phi,X)\subsetneq\Z=\langle X\rangle.$$

Fix $n\in\N_+$ and consider the map $$j_n:\set{0,1,2,3}^n\to \set{0,1,2,\ldots,4^n-1},\ (a_0,\ldots,a_{n-1})\mapsto a_0+a_14+\ldots+a_{n-1}4^{n-1},$$
which is easily verified to be a bijection; since $j_n(\set{0,1}^n)=T_n(\phi,X)$, we have that $\card{T_n(\phi,X)}=2^n$. Hence, $$H_{alg}(\phi,X)=\log 2.$$

We see now that, for $X'=T_2(\phi,X)=\set{0,1,4,5}$, $$H_{alg}(\phi,X')=\log3.$$ One can prove by induction that, for every $n\in\N_+$,
\begin{equation}\label{Tn4}
T_n(\phi,X')=\{b_0+4 b_1+
\ldots+4^{n-1} b_{n-1}+4^n b_n : b_1,\ldots ,b_{n-1}\in\{0,1,2\},b_0,b_n\in\{0,1\}\}.
\end{equation}
It implies that for every $n\in\N_+$, $$T_n(\phi,X')=j_n(\set{0,1}\times\set{0,1,2}^n\times\set{0,1}),$$
so that we can conclude that $$\card{T_n(\phi,X')}=4\cdot 3^n.$$ Clearly, this implies that $H_{alg}(\phi,X')=\log 3$.
\end{example}

\subsection{Computing entropy using generators}

Let $S$ be a cancellative right amenable semigroup, $A$ an abelian group, and $S\overset{\alpha}\curvearrowright A$ a left action. For a subset $X$ of $A$, the \emph{full $\alpha$-trajectory} of $X$ is $$T_S(\alpha,X)=\bigcup_{F\in\Pf(S)}T_F(\alpha,X).$$

If $X$ is a subgroup of $A$, then $T_F(\alpha,X)=\langle \alpha(s)(X): s\in F\rangle$ is a subgroup of $A$ for every $F\in\Pf(S)$, and so also $T_S(\alpha,X)=\langle \alpha(s)(X): s\in S\rangle$ is a subgroup of $A$. If furthermore $X\in\mathcal F(A)$, then $T_S(\alpha,X)\subseteq t(A)$.
On the other hand, if $X$ is a subset of $A$, in general $T_S(\alpha,X)$ could be strictly contained in $\langle \alpha(s)(X): s\in S\rangle$ as Example~\ref{4x} shows, so the condition $A=T_S(\alpha,X)$ is stronger than $A=\langle\alpha(s)(X): s\in S\rangle$.

\medskip
We recall that the action $S\overset{\alpha}\curvearrowright A$ induces on $A$ a structure of left $\Z[S]$-module. So, a subgroup $B$ of $A$ generates $A$ as a $\Z[S]$-module if and only if $A=\bigcup_{X\in\Pf^0(B)}T_S(\alpha,X)$.

\begin{proposition}\label{alphagensgi}
Let $S$ be a right amenable semigroup, $A$ an abelian group and $S\overset{\alpha}\curvearrowright A$. If $B$ is a subgroup of $A$ such that $t(B)$ generates $t(A)$ as a left $\Z[S]$-module, then 
\begin{equation}\label{****}
\ent(\alpha) = \sup\set{H_{alg}(\alpha, X): X \in\mathcal F(B)}.
\end{equation}
In particular, if $t(A)=T_S(\alpha,X)$ for some $X\in\mathcal F(A)$, then $\ent(\alpha)=H_{alg}(\alpha,X).$
\end{proposition}
\begin{proof} 
By definition, $\ent(\alpha) \geq \sup\set{H_{alg}(\alpha, X)\colon X \in\mathcal F(B)}$. To prove the opposite inequality, let $Y\in\mathcal F(A)$. Since $$t(A)=\bigcup_{X\in\mathcal F(B)}T_S(X),$$ there exist $F\in\Pf^0(S)$ and $X\in\mathcal F(B)$ such that $Y\subseteq T_F(X)$. Then, by Lemma~\ref{cofinal}(a) and Lemma~\ref{HX=HTX}, 
\begin{equation}\label{176}
H_{alg}(\alpha, Y) \leq H_{alg}(\alpha,T_F(X))= H_{alg}(\alpha, X).
\end{equation}
As $Y\in\mathcal F(A)$ was chosen arbitrarily, this proves \eqref{****}.

If $t(A)=T_S(\alpha,X)$ for some $X\in\mathcal F(A)$, for every $Y\in\mathcal F(A)$ there exists $F\in\Pf^0(S)$ such that $Y\subseteq T_F(X)$, so the last assertion follows again from \eqref{176}.
\end{proof}

In case $X$ is a finite subset of an abelian groups $A$, it is not clear whether $A=T_S(\alpha,X)$ would imply $h_{alg}(\alpha)=H_{alg}(\alpha,X)$.
On the other hand, we have at least the following result.

\begin{proposition}
Let $S$ be a cancellative right amenable monoid, $A$ an abelian group, and $S\overset{\alpha}\curvearrowright A$ a let action. 
If $B$ is a subgroup of $A$ that generates $A$ as a left $\Z[S]$-module, then
$$h_{alg}(\alpha)=\sup\set{H_{alg}(\alpha,X): X\in \Pf^0(B)}.$$
\end{proposition}
\begin{proof} 
If $S$ is finite, then $h_{alg}(\alpha)=\infty$ and $$\sup\set{H_{alg}(\alpha,X): X\in \Pf^0(B)}=\infty,$$ 
by Example~\ref{Esempio:AG}(a). So assume that $S$ is infinite and let $(F_i)_{i\in I}$ be a right F\o lner net of $S$. Clearly, 
$$h_{alg}(\alpha)\geq\sup\set{H_{alg}(\alpha,X): X\in \Pf^0(B)}.$$ 
To prove the converse inequality, let $Z\in\Pf^0(A)$. Since $A=\bigcup_{X\in\Pf^0(B)}T_S(\alpha,X)$, there exist $X\in\Pf^0(B)$ and $F\in\Pf^0(S)$ such that $Z\subseteq T_F(X)$.  Let $X'=X_{|F|}$; since $B$ is a subgroup of $A$, we have that $X'\in\Pf^0(B)$. 
Since $(F_iF)_{i\in I}$ is a right F\o lner sequence of $S$ by Lemma~\ref{F_iF}(b), and since $(F_i)_{i\in I}$ is strictly increasing, we have that 
$$H_{alg}(\alpha,X')=\lim_{i\in I}\frac{\ell(T_{F_iF}(X))}{\card{F_iF}}=\lim_{i\in I}\frac{\ell(T_{F_i F}(X))}{\card{F_i}}\cdot\frac{\card{F_i}}{\card{F_iF}}=\lim_{i\in I}\frac{\ell(T_{F_iF}(X))}{\card{F_i}}.$$
By Lemma~\ref{cofinal}(a) and Lemma~\ref{TEF},
$$H_{alg}(\alpha,Z)=\lim_{i\in I}\frac{\ell(T_{F_i}(Z))}{\card{F_i}}\leq \lim_{i\in I}\frac{\ell(T_{F_i}(T_F(X)))}{\card{F_i}}\leq \lim_{i\in I}\frac{\ell(T_{F_i F}(X'))}{\card{F_i}}=H_{alg}(\alpha,X').$$
Hence, we can conclude that $h_{alg}(\alpha)\leq\sup\set{H_{alg}(\alpha,X): X\in \Pf^0(B)}$.
\end{proof}

\begin{lemma}\label{lem:H+} Let $S$ be a cancellative right amenable monoid, $A$ an abelian group, $S\overset{\alpha}\curvearrowright A$ e left action, and $B, C\in\Pf(A)$. Then
\[H_{alg}(\alpha, B + C) \leq H_{alg}(\alpha, B) + H_{alg} (\alpha, C) \quad\text{and}\quad H_{alg}(\alpha, - B) = H_{alg}(\alpha, B).\]
\end{lemma}
\begin{proof}
For every $F \in\Pf(S)$,
\[\ell(T_F(\alpha, B +C)) \leq \ell(T_F(\alpha, B)) + \ell(T_F(\alpha, C))  \quad\text{and}\quad \ell(T_F(\alpha, -B)) = \ell(T_F(\alpha, B)).\]
Let $(F_i)_{i \in I}$ be a right \Folner net of $S$. Then
\[\frac{\ell(T_{F_i}(\alpha, B +C))}{\card{F_i}} \leq \frac{\ell(T_F(\alpha, B))}{\card{F_i}} + \frac{\ell(T_F(\alpha, C))}{\card{F_i}} \quad\text{and}\quad \frac{\ell(T_{F_i}(\alpha, -B))}{\card{F_i}} = \frac{\ell(T_{F_i}(\alpha, -B))}{\card{F_i}}.\]
Taking the limit as $i \to \infty$, we have the conclusion.
\end{proof}

The following is a useful technical consequence of the above lemma. 

\begin{lemma}\label{lem:halg-0}
Let $S$ be a cancellative right amenable monoid, $A$ an abelian group, and $S\overset{\alpha}\curvearrowright A$ a left action.
Let $B$ be a subset of $A$ generating $A$ as a group. For every $a \in A$, let $W_a= \set{0, a}\in\Pf^0(A)$.
If $H_{alg}(\alpha, W_b) = 0$ for every $b \in B$, then $\halg(\alpha) = 0$.
\end{lemma}
\begin{proof}
Let $C \in\Pf^0(A)$. Since $B$ generates $A$, there exist $b_1, \dotsc, b_m \in B$ (each of the $b_i$ can appear more than once) such that
\[C \subseteq \pm W_{b_1} + \ldots + \pm W_{b_m}\]
By Lemma~\ref{lem:H+},
\[H_{alg}(\alpha, C) \leq H_{alg}(\alpha, W_1) + \ldots + H_{alg}(\alpha, W_m) = 0,\]
so we have the thesis.
\end{proof}

Following \cite{DGB,DGB0}, call an action $S\overset{\alpha}\curvearrowright A$ of a cancellative right amenable monoid $S$ on an abelian group $A$
\begin{enumerate}[(a)]
\item \emph{locally nilpotent} if  for every $a \in A$ there exists $s \in S$ such that $\alpha(s)(a) = 0$;
\item \emph{weakly locally nilpotent} if the same condition is satisfied for all elements $a$ taken from a set $B$ of generators of the group $A$. 
\end{enumerate}
If $S$ is commutative, then these two conditions are obviously equivalent. 

We show that the algebraic entropy of every weakly locally nilpotent action is zero (for $\N$-actions this can be found in \cite{DGB,DGB0}).

\begin{corollary}\label{cor:halg-annihilator}
Let $S$ be a cancellative right amenable monoid, $A$ an abelian group, and $S\overset{\alpha}\curvearrowright A$ a left action. If $\alpha$ is weakly locally nilpotent, then $\alpha$ is locally nilpotent, and $\halg(\alpha) = 0$.
\end{corollary}
\begin{proof}
Let $B$ be a subset generating $A$ as a group and witnessing the weak local nilpotency, i.e., such that for every $b \in B$, there exists $s \in S$ such that $\alpha(s)(b) = 0$. Fix $b \in B$ and let $s \in S$ such that $\alpha(s)(b) = 0$.

To show that $\alpha$ is locally nilpotent, let $a \in A$. There exist $b_1, \dotsc ,b_k \in B$ such that $a = b_1 + \dotsc+ b_k$. For every $i\in\{1,\ldots,k\}$, let $s_i \in S$ such that $\alpha(s_i)(b_i) = 0$. By Corollary~\ref{cor:multi-Ore}, there exist $t, r_1, \dotsc, r_k \in S$ such that $t = r_1 s_1 = \ldots = r_k s_k$. For every $i\in\{1,\ldots,k\}$, we have that
\[\alpha(t)(b_i) = \alpha(r_i s_i)(b_i) = \alpha(r_i)(\alpha(s_i)(b_i)) = \alpha(r_i)(0) = 0.\]
Therefore, $$\alpha(t)(a) = \sum_{i=1}^k\alpha(t)(b_i) = 0,$$ that shows that $\alpha$ is locally nilpotent.

Let us prove now that $\halg(\alpha) = 0$. By Lemma~\ref{lem:halg-0}, it suffices to show that $\Halg(\alpha,W_b) = 0$ for every $b\in B$. Let $(F_i)_{i \in I}$ be a right \Folner net of $S$. Then $(F_i s)_{i \in I}$ is also a right \Folner net of $S$ by Lemma~\ref{F_iF}. Moreover, $T_{F_i s}(\alpha, W_b) = \set{0}$. Therefore,
\[\Halg(\alpha,W_b) = \lim_{i \in I} \frac{\ell(T_{F_i s}(\alpha, W_b))}{\card{F_i s}} = 0,\]
that concludes the proof.
\end{proof}

No group $S$ admits a weakly locally nilpotent action. Indeed, if $s \in S$ is invertible and $0 \neq a \in A$, then $\alpha(s)(a) \neq 0$.

\section{Entropy of restriction and quotient actions of amenable group actions}\label{restr:quot:sec}

\subsection{Entropy of restriction actions}

Let $G$ be a cancellative right amenable monoid, $A$ an abelian group, and $G\overset{\alpha}\curvearrowright A$ a left action. If $H$ is a submonoid of $G$ we call \emph{restriction action} of $\alpha$ the left action $H\overset{\alpha\restriction_H}\curvearrowright A$ induced by $\alpha$. 
In this section we point out some basic relation between the algebraic entropy of $\alpha$ and the algebraic entropy of $\alpha\restriction_H$. 

\medskip
The next example shows immediately that the restriction action of an action of finite (actually, zero) algebraic entropy may have infinite algebraic entropy.

\begin{example}\label{exconj}
Let $G$ be an infinite amenable monoid, $A$ an infinite abelian group, and consider the trivial action $G\overset{\tau}{\curvearrowright}A$. By Example~\ref{Esempio:AG}(b), we have that $$h_{alg}(\tau)=\ent(\tau)=0,$$ while for $H=\{1\}$, $$h_{alg}(\tau\restriction_H)=\infty$$ and $\ent(\tau\restriction_H)=\infty$ when $t(A)$ is infinite.
\end{example}

For a non-trivial example with $h_{alg}(\alpha)=0$ while $h_{alg}(\alpha\restriction_H)=\infty$ see Example~\ref{gamma=alpha}.

\medskip
Next we give a result for monoids showing that many actions of $\N^d$ have zero  algebraic entropy (see Corollary~\ref{Coro:June12}).
We will state many other results on the vanishing of the algebraic entropy for amenable group actions after Theorem~\ref{teo:submonoid}.

\begin{lemma}\label{N2N}
Let $G$ and $H$ be infinite countable cancellative right amenable monoids, $A$ an abelian group and consider the left action $G\times H\overset{\alpha}\curvearrowright A$. If $h_{alg}(\alpha\restriction_G)<\infty$, then $h_{alg}(\alpha)=0$.
\end{lemma}
\begin{proof}
Let $(F_n)_{n\in \N}$ and $(E_m)_{m\in \N}$ be right F\o lner sequences of $G$ and $H$ respectively. Then $(F_n\times E_m)_{n,m\in\N}$ is a right F\o lner net of $G\times H$ by Lemma~\ref{Fnprod}.
Let $X\in\Pf^0(A)$. For every $n,m\in\N$, $$T_{F_n\times E_m}(\alpha,X)=T_{F_n}(\alpha\restriction_{G},T_{E_m}(\alpha\restriction_H,X)).$$
Fixed $m\in\N$, we have that 
$$\lim_{n\to \infty}\frac{\ell(T_{F_n}(\alpha\restriction_{G},T_{E_m}(\alpha\restriction_H,X)))}{|F_n|}=H_{alg}(\alpha,\restriction_G,T_{E_m}(\alpha\restriction_H,X))\leq h_{alg}(\alpha\restriction_G)<\infty.$$
Therefore,
$$H_{alg}(\alpha,X)=\lim_{n,m\to \infty}\frac{\ell(T_{F_n}(\alpha\restriction_{G},T_{E_m}(\alpha\restriction_H,X)))}{\card{F_n}\card{E_m}}\leq \lim_{m\to\infty}\frac{h_{alg}(\alpha\restriction_G)}{\card{E_m}}=0,$$
and we conclude that $h_{alg}(\alpha)=0$.
\end{proof}

\begin{corollary}
Let $G$ and $H$ be infinite countable cancellative right amenable monoids, $A$ an abelian group and consider the left action $G\times H\overset{\alpha}\curvearrowright A$. Then $h_{alg}(\alpha)\leq h_{alg}(\alpha\restriction_G)$.
\end{corollary}
\begin{proof}
If $h_{alg}(\alpha\restriction_G)=\infty$ clearly $h_{alg}(\alpha)\leq h_{alg}(\alpha\restriction_G)=\infty$, and when $h_{alg}(\alpha\restriction_G)<\infty$ Lemma~\ref{N2N} gives $h_{alg}(\alpha)=0\leq h_{alg}(\alpha\restriction_G)$.
\end{proof}


\begin{corollary}\label{N2Ncor}
Let $d>1$, $A$ an abelian group and $\N^d\overset{\alpha}\curvearrowright A$ a left action. If $h_{alg}(\alpha(e_i))<\infty$ for some $e_i=(0,\ldots,0,1,0,\ldots,0)\in\N^d$ where $1$ is in the $i$-th entry, then $h_{alg}(\alpha)=0$.
\end{corollary}

\begin{corollary}\label{Coro:June12}
Let $d>1$, $A$ a torsion-free abelian group of finite rank and $\N^d\overset{\alpha}\curvearrowright A$ a left action. Then $h_{alg}(\alpha)=0$.
\end{corollary}
\begin{proof}
In view of Corollary~\ref{N2Ncor}, it is enough to check that no endomorphism of $A$ may have infinite algebraic entropy. In fact, every endomorphism 
$\phi$ of $A$ extends to an endomorphism $\widetilde \phi$  of the divisible hull $D\cong\Q^k$, with $k\leq n$, of $A$, and $\halg(\widetilde \phi) = \halg(\phi)$ (see \cite{DGB0}). 
Moreover, all endomorphisms of $\Q^k$ have finite algebraic entropy by the algebraic Yuzvinski Formula (see \cite{GBV2}).
\end{proof}

From now on we consider amenable group actions and their restriction actions.
The following is a generalization of what is called Logarithmic Law in the case $G=\mathbb Z$.

\begin{proposition}\label{finind} 
Let $G$ be an amenable group, $A$ an abelian group, and $G\overset{\alpha}{\curvearrowright}A$ a left action. If $H$ is a subgroup of $G$ of finite index $[G:H]=k$, then 
$$h_{alg}(\alpha\restriction_H)=k\cdot h_{alg}(\alpha)\quad\text{and}\quad\ent(\alpha\restriction_H)=k\cdot \ent(\alpha).$$
In particular:
\begin{enumerate}[(a)]
\item $h_{alg}(\alpha)=0$ if and only if $h_{alg}(\alpha\restriction_H)=0$, and $\ent(\alpha)=0$ if and only if $\ent(\alpha\restriction_H)=0$;
\item $h_{alg}(\alpha)=\infty$ if and only if $h_{alg}(\alpha\restriction_H)=\infty$, and $\ent(\alpha)=\infty$ if and only if $\ent(\alpha\restriction_H)=\infty$.
\end{enumerate}
\end{proposition}
\begin{proof}
Let $(F_i)_{i\in I}$ be a right F\o lner net of $H$. Let $1\in L\subseteq G$ be a transversal of the family $G/H=\{Hg:g\in G\}$ of the right cosets of $H$ in $G$; in particular $|L|=k$, so we can write $L=\{t_1,t_2,\ldots,t_k\}$. By Lemma~\ref{F_iF}(b), $(F_iL)_{i\in I}$ is a right F\o lner net of $G$.

Observing that $|F_iL|=|F_i|\cdot|L|$, we can compute, for $X\in\Pf^0(A)$, $$H_{alg}(\alpha,X)=\lim_{i\in I}\frac{\ell(T_{F_iL}(\alpha,X))}{|F_iL|}.$$
For $Y=T_L(\alpha,X)$, we have that $$T_{F_iL}(\alpha,X)=T_{F_i}(\alpha,Y)=T_{F_i}(\alpha\restriction_H,Y),$$
so $$H_{alg}(\alpha,X)=\lim_{i\in I}\frac{\ell(T_{F_i}(\alpha,Y))}{|F_i|\cdot|L|}=\frac{1}{|L|}\lim_{i\in I}\frac{\ell(T_{F_i}(\alpha,Y))}{|F_i|}=\frac{1}{k}H_{alg}(\alpha,Y)=\frac{1}{k}H_{alg}(\alpha\restriction_H,Y).$$
Therefore, we have the inequality $h_{alg}(\alpha)\leq\frac{1}{k}h_{alg}(\alpha\restriction_H)$.

\medskip
To prove the converse inequality, let $Y\in\Pf^0(A)$. Then 
$$\frac{1}{k}H_{alg}(\alpha\restriction_H,Y)=\lim_{i\in I}\frac{\ell(T_{F_i}(\alpha,Y))}{|L|\cdot|F_i|}= \lim_{i\in I}\frac{\ell(T_{F_i}(\alpha,Y))}{|F_iL|}\leq 
\lim_{i\in I}\frac{\ell(T_{F_iL}(\alpha,Y))}{|F_iL|}=H_{alg}(\alpha,Y),$$
and so $\frac{1}{k}h_{alg}(\alpha\restriction_H)\leq h_{alg}(\alpha)$. This concludes the proof that $h_{alg}(\alpha\restriction_H)=k\cdot h_{alg}(\alpha)$.

\smallskip
By Proposition~\ref{h=ent} and the first equality, we have that 
$$\ent(\alpha\restriction_H)=h_{alg}(\alpha_{t(A)})=k\cdot h_{alg}(\alpha_{t(A)})=k\cdot \ent(\alpha),$$
hence the thesis.
\end{proof}

Let us recall that a group $G$ is \emph{virtually cyclic} if $G$ has a cyclic subgroup of finite index. The above proposition allows us to completely determine the algebraic entropy of actions of virtually cyclic groups:

\begin{remark}
Let $G$ be a virtually cyclic group, $A$ an abelian group, and $G\overset{\alpha}{\curvearrowright}A$ a left action. The case when $G$ is finite was already discussed in Example~\ref{Esempio:AG}(a), so we can assume without loss of generality that $G$ is infinite. Then $G$ has an infinite cyclic
subgroup $N$ of finite index. Then its normal core $N_G$ is still an infinite cyclic subgroup of finite index $k=[G:N_G]\in\N_+$.
As the subgroup $N_G$ is normal, we deduce from Proposition~\ref{finind} that $\halg(\alpha\rest_{N_G})= k\halg(\alpha)$. 
Denoting $\phi=\alpha(1)$, we have that $\halg(\alpha\rest_{N_G}) = \halg(\phi)$, by Remark~\ref{mayday}(a). Therefore, $\halg(\alpha) = k\halg(\phi)$. 
\end{remark}

%

The next proposition can be deduced from Theorem~\ref{teo:submonoid} and Proposition~\ref{finind}, nevertheless we anticipate it here since its proof is much easier than that of Theorem~\ref{teo:submonoid}.  

\begin{proposition}\label{restr:act} 
Let $G$ be an amenable group, $A$ an abelian group, and $G\overset{\alpha}{\curvearrowright}A$ a left action. If $N$ is a normal subgroup of $G$, then
$$h_{alg}(\alpha)\leq h_{alg}(\alpha\restriction_N)\quad \text{and}\quad \ent(\alpha)\leq\ent(\alpha\restriction_N).$$
\end{proposition}
\begin{proof} 
Let $X\in\Pf^0(A)$ and $f=f_X$ (we recall that $f_X(Y)=\ell(T_Y(\alpha,X))$ for every $Y \in \Pf(G)$). By Corollary~\ref{cor:subgroup},
$$H_{alg}(\alpha,X) = \mathcal H_G(f) \leq \mathcal H_N(f) = H_{alg}(\alpha \restriction_N,X).$$
Since the above is true for every $X\in\Pf^0(A)$, we have that
$$h_{alg}(\alpha) \leq h_{alg}(\alpha \restriction_N).$$
By Proposition~\ref{h=ent} the second assertion follows from the first one.
\end{proof}

 Example~\ref{exconj} shows that the inequalities $h_{alg}(\alpha)\leq h_{alg}(\alpha\restriction_H)$ and $\ent(\alpha)\leq \ent(\alpha\restriction_H)$ in Proposition~\ref{restr:act} can be strict.


\begin{corollary}
Let $G$  be a group and $G\overset{\alpha}{\curvearrowright}A$ be a left action on an abelian group $A$. If an element $g\in G$ is non-torsion and central, then \[h_{alg}(\alpha)\leq h_{alg}(\alpha(g)).\] 
In particular, if $h_{alg}(\alpha(g))=0$, then $h_{alg}(\alpha)=0$.
\end{corollary}
\begin{proof} 
The subgroup $H=\langle g\rangle$ is central, hence normal. By Proposition~\ref{restr:act}, $h_{alg}(\alpha)\leq h_{alg}(\alpha\restriction_H)$. Now, $h_{alg}(\alpha\restriction_H)=h_{alg}(\alpha(g))$ by Remark~\ref{mayday}(a).
\end{proof}

The next corollary should be compared with Corollary~\ref{h<infty} and Corollary~\ref{Coro:application} where the condition $h_{alg}(\phi)=0$ is relaxed to a milder one. 

\begin{corollary}\label{h=0all} 
Let $G$ be a non-torsion abelian group and let $G\overset{\alpha}{\curvearrowright}A$ be a left action on an abelian group $A$. If $h_{alg}(\phi)=0$ for every $\phi\in\Aut(A)$, then $h_{alg}(\alpha)=0$.
\end{corollary}

The groups $\Z$, $\Z(p^{\infty})$ or $\Q/\Z$ satisfy the hypothesis imposed on $A$ in the above corollary (see \cite{DGSZ}).

\medskip
In the rest of this subsection we consider exclusively appropriate sufficient conditions that entail zero algebraic entropy for a left action $G\overset{\alpha}{\curvearrowright}A$. Most of these conditions are in terms of the restricted action $\alpha \rest_N$ with respect to appropriate subgroups
(necessarily of infinite index)  $N$ of $G$. An exception to this tendency is the hypothesis of the next proposition, verified in Example~\ref{gamma=alpha}(a).

\begin{proposition}
Let $G$ be an amenable group and let $G\overset{\alpha}{\curvearrowright}A$ be a left action an abelian group $A$. Let $N$ be a normal subgroup of $G$ such that $G/N$ is infinite. If every $N$-invariant subgroup of $A$ is also $\alpha$-invariant, then $$\ent(\alpha)=0.$$
\end{proposition}
\begin{proof} 
For the sake of brevity, let $\alpha'= \alpha \! \restriction_N$. Fix $A_0\in\mathcal F(A)$. Then $L=T_N(\alpha',A_0)$ is an $\alpha'$-invariant (i.e., $N$-invariant) subgroup of $A$  containing $A_0$, and this gives an action $N \overset{\alpha'_L}\curvearrowright L$. By hypothesis, $L$ is also $\alpha$-invariant. 
Hence, 
for every $Y\in\Pf^0(G)$ the subgroup $L$ contains $T_Y(\alpha,A_0)$, which obviously generates $L$ as a left $\Z[N]$-module. By Proposition~\ref{alphagensgi} applied to $N\overset{\alpha'_L}\curvearrowright L$ and the subgroup $T_Y(\alpha,A_0)$ of $L$, we deduce that 
\begin{equation}\label{22Sept}
\ent(\alpha'_L)=H_{alg}(\alpha'_L,T_Y(\alpha,A_0)).
\end{equation}

Let $f=f_{A_0}$, and let $C= G/N$, which is infinite by virtue of our assumption. 
Fix a section $\sigma: C \to G$ and consider $\theta=\Theta_\sigma(f)$. 
Since $A_0$ is a subgroup of $A$, for every $X\in\Pf^0(C)$, $$f^{\sigma(X)}=f_{T_{\sigma(X)}(\alpha,A_0)}$$ by Lemma~\ref{TEF}. Hence, Equation~\eqref{22Sept} entails 
$$\theta(X) = \mathcal H_N(f^{\sigma(X)}) = H_{alg}(\alpha'_L, T_{\sigma(X)}(\alpha,A_0)) =\ent(\alpha'_L)$$
for every $X \in \Pf^0(C)$. As $\ent(\alpha'_L)$ is obviously independent from $X$, this proves that $\theta$ is constant on $\Pf^0(C)$.  
Since $C$ is infinite, we can apply Theorem~\ref{Fubini} and Lemma~\ref{constantFubini} to deduce that 
$$H_{alg}(\alpha,A_0)= \mathcal H_G(f) =\mathcal H_C(\theta)=0.$$
Taking the supremum over all $A_0\in\mathcal F(A)$, we get the thesis.
\end{proof}


The next is one of the main result that we have on restriction actions and it has some impressive consequences. 

\begin{theorem}\label{teo:submonoid}
Let $G$ be an amenable group, $A$ an abelian group, and $G\overset{\alpha}{\curvearrowright}A$ a left action. Let $N$ be a non-trivial normal subgroup of $G$ such that $G/N$ is infinite.
\begin{itemize}
\item[(a)] If $\ent(\alpha \rest_N)<\infty$, then $\ent(\alpha)=0$.
\item[(b)] If $\halg(\alpha \rest_N)<\infty$, then $\halg(\alpha)=0.$
\end{itemize}
\end{theorem}
\begin{proof} 
In view of Proposition~\ref{h=ent}, (a) trivially follows from (b), since $\ent(\alpha) = \halg(\alpha_{t(A)})$ and $\ent(\alpha\rest_N) = \halg((\alpha\rest_N)_{t(A)})$. 

The proof of (b) follows the proof of the above proposition, so we keep the same notation, i.e., let $\alpha'= \alpha \! \restriction_N$, $C= G/N$ (infinite, by our hypothesis) and $\sigma: C \to G$ is a section. There is a subtle difference though, now we let $A_0\in \Pf(A)$. For $f=f_{A_0}$ we define again $\theta=\Theta_\sigma(f)$. 

Our first aim is to see that $\theta$ is bounded by $\halg(\alpha')$. To this end fix $X \in \Pf^0(C)$ and let $A_1= T_{\sigma(X)}(\alpha, A_0)$. 
Clearly, $A_1 \in \Pf(A)$. Then, for an arbitrary \Folner net $(F_i)_{i \in I}$ for~$N$ one has 
\begin{equation}\label{newEq}
\theta(X) = \mathcal H_N(f^{\sigma(X)}) =\lim_{i \in I} \frac{\ell(T_{F_i \sigma(X)}(A_0) )}{\card{F_i}} \buildrel{(*)}\over\leq 
\lim_{i \in I} \frac{\ell(T_{F_i}(A_1))}{\card{F_i}} =H_{alg}(\alpha', A_1)\leq \halg(\alpha'), 
\end{equation}
where the inequality ($*$) is due to Lemma~\ref{TEF}. 
We have proved in this way that our assumption $\halg(\alpha') < \infty$ implies  boundedness of  $\theta\restriction_{\Pf^0(C)}$ for all $A_0\in \Pf(A)$. 

Since $C$ is infinite, Theorem~\ref{Fubini}, Lemma~\ref{constantFubini}, and Remark~\ref{1in-bounded} imply that $$H_{alg}(\alpha,A_0)= \mathcal H_G(f) =\mathcal H_C(\theta)=0.$$ After taking the supremum over all $A_0\in\Pf(A)$ we get $\halg(\alpha)=0$. 
\end{proof}

We conjecture that Theorem~\ref{teo:submonoid} can be proved without the assumption that the subgroup $N$ is normal (see Conjecture~\ref{Conj2}). 

\medskip
The following is another interesting consequence of Theorem~\ref{teo:submonoid}.

\begin{corollary}
Let $G_1$, $G_2$ be infinite amenable groups, $A_1$, $A_2$ abelian groups, and consider left actions $G_1\overset{\alpha_1}\curvearrowright A_1$ and $G_2\overset{\alpha_2}\curvearrowright A_2$. Let $G=G_1\times G_2$, $A=A_1\times A_2$, and let the left action $G\overset{\alpha}\curvearrowright A$ be defined by $\alpha(g_1,g_1)=\alpha_1(g_1)\times\alpha_2(g_2)$ for every $(g_1,g_2)\in G$. If either $h_{alg}(\alpha_1)$ or $h_{alg}(\alpha_2)$  is finite, then $h_{alg}(\alpha)=0$.
\end{corollary}
\begin{proof} 
Assume that $h_{alg}(\alpha_1)< \infty$.  Since $\alpha\restriction_{G_1}=\alpha_1$, it suffices to apply Theorem~\ref{teo:submonoid} with $N=G_1$.
\end{proof}

As an example one can take $G=G_1=G_2$, $A=A_1=A_2$ and an action $G\overset{\alpha}\curvearrowright A$ with $\halg(\alpha)< \infty$.
Then the action $G^2\curvearrowright A^2$ defined as above has zero algebraic entropy.


\medskip
The following corollary of Theorem~\ref{teo:submonoid} has very interesting consequences.

\begin{corollary}\label{cyclic0}
Let $G$ be an amenable group and $G\overset{\alpha}{\curvearrowright}A$ a left action on an infinite abelian group $A$. 
If $G$ has a non-torsion central element $g$ with $h_{alg}(\alpha(g))<\infty$ and $[G:\langle g\rangle]$ is infinite, then $h_{alg}(\alpha) = 0$.
\end{corollary}
\begin{proof}
The subgroup $N=\langle\alpha(g)\rangle$ of $G$ is central, hence normal. Since $N=\langle g\rangle$ is infinite by hypothesis,
Remark~\ref{mayday}(a) gives that $h_{alg}(\alpha\restriction_N)=h_{alg}(\alpha(g))<\infty$. Hence, $h_{alg}(\alpha)=0$ by Theorem~\ref{teo:submonoid}.
\end{proof}

The next example shows that in Corollary~\ref{cyclic0} one cannot replace the hypothesis $g$ non-torsion by simply $g$ non-trivial, since $h_{alg}(\alpha(g))$ does not always coincide with $h_{alg}(\alpha\restriction_{\langle g\rangle})$.

\begin{example}\label{NEW:EXAMPLE!}
Let $G = \Z(2)\times \Z$ and let $g\ne 0$ be the generator of $\Z(2)$. Hence, $N=\langle g\rangle= \Z(2)$ and $[G:\langle g\rangle]$ is infinite. Then $o(g) =2  = o(\alpha(g))$.
 
Consider an action $G\overset{\alpha}{\curvearrowright}A$ on an infinite abelian group $A$, such that $\alpha\restriction_{\Z(2)}$ is trivial, while $\alpha\restriction_{\Z}$ has positive algebraic entropy, e.g., $A=\Q$ and, for every $(kg,m)\in G$, $$\alpha(kg,m)(x) = 2^mx\ \text{for every}\ x\in \Q.$$ 
Then $h_{alg}(\alpha(g))=0 <\infty$; nevertheless, $h_{alg}(\alpha\restriction_{\Z(2)}) = \infty$, as $A$ is infinite, according to Example~\ref{Esempio:AG}(a) and Remark~\ref{mayday}(b). 

Moreover, $\halg(\alpha)=\frac{1}{2}\halg(\alpha\restriction_{\Z})= \frac{1}{2}\log 2 > 0$, according to Proposition~\ref{finind} and Remark~\ref{mayday}(a),  since the multiplication $m_2:\Z\to \Z$, $x\mapsto 2x$ has $h_{alg}(m_2)=\log 2$ (see \cite{DGB0}).
\end{example}

In the next corollary, $\Z^d$ can be replaced by any non-torsion abelian group $G$ that is not virtually ciclic (so that $G$ has an infinite cyclic subgroup of infinite index). 

\begin{corollary}\label{Coro:application1} 
Let $n,d\in \N_+$, let $A$ be a subgroup of $\Q^n$, and $\Z^d\overset{\alpha}{\curvearrowright}A$ a left action. If $d> 1$, then $\halg(\alpha) = 0$. 
\end{corollary}
\begin{proof} 
Follows from Corollary~\ref{Coro:June12} and Remark~\ref{S->Gh}. 

Since $A$ has no automorphism of infinite algebraic entropy 
(see the proof of  Corollary~\ref{Coro:June12}), an alternative proof can be obtained from Corollary~\ref{cyclic0}, as
since $\Z^d$ with $d>1$ has no cyclic subgroups of finite index.
\end{proof}

We show now that the action of a field on itself by multiplication has zero algebraic entropy.

\begin{proposition}\label{casen=1}
Let $K$ be an infinite field  and consider the natural action $K^*\overset{\lambda}{\curvearrowright} K$ by multiplication. Then $h_{alg}(\lambda)=0$.
\end{proposition}
\begin{proof}  
Assume first that $K$ has characteristic zero. The underlying abelian group $A=(K,+)$ is a $\Q$-vector space and we can assume that $\Q\subseteq K$. By Proposition~\ref{restr:act}, $h_{alg}(\lambda)\leq h_{alg}(\varrho)$, where $\varrho=\lambda\restriction_{\Q^*}$, so it remains to prove that $h_{alg}(\varrho)=0$.

Let $X\in\Pf(A)$ and let $V$ be the $\Q$-linear subspace of $A$ generated by $X$, which is $\varrho$-invariant. Since $V$ is a torsion-free abelian group of finite rank, and $\Q^*$ as a group is isomorphic to $\Z(2)\times \Z^{(\Z)}$, by (the extended version of) Corollary~\ref{Coro:application1} and Proposition~\ref{finind}, we have that $h_{alg}(\varrho_V)=0$. Hence, $H_{alg}(\varrho,X)=H_{alg}(\varrho_V,X)\leq h_{alg}(\varrho_V)=0$.
Therefore, $h_{alg}(\varrho)=0$.

\medskip
Suppose that $K$ has characteristic $p>0$. We consider two cases.

First, assume that $K$ is an algebraic extension of $\F_p$. Then there exists a sequence $(a_n)_{n\in\N}$ in $\N_+$
such that $a_n|a_{n+1}$ for every $n \in\N$ and 
$$K= \bigcup_{n\in\N} K_n,\ \text{where}\ K_n = \F_{p^{a_n}}.$$
Then we can consider $(K^*_{n})_{n\in\N}$ as a right F\o lner sequence of $K^*$. Let $X\in\Pf^0(K)$. Then there exists $m\in\N$ such that $X\subseteq K_m$. For every $n\in\N$ with $n\geq m$, one has 
$$T_{K_n^*}(\lambda, K_m)=K_n^*K_m=K_n.$$
Hence, $$H_{alg}(\lambda,X)\leq H_{alg}(\lambda, K_m)=\lim_{n\to \infty}\frac{\ell(T_{K_n^*}(\lambda, K_m))}{|K_n^*|}=\lim_{n\to \infty}\frac{\ell(K_n)}{|K_n^*|}=0.$$
We can conclude that $h_{alg}(\lambda)=0$.

 Now assume that there exists a non-algebraic element $c\in K$. Let $S=\{c^i(c+1)^j\colon i,j\in\N\}$ be the submonoid of $K^*$ generated by $c$ and $c+1$.  Obviously, $S\cong\N^2$, since $c$ is transcendental over $\F_p$.  This determines the restricted action $S \overset{\lambda\restriction_S}{\curvearrowright} K$.  To conclude, it sufficed to prove that 
\begin{equation}\label{Laaaaast:Eq}
h_{alg}(\lambda\restriction_S)=0.
\end{equation}
Indeed, for the subgroup  $G$ of $K^*$ generated by $S$, \eqref{Laaaaast:Eq} implies that $h_{alg}(\lambda\restriction_G)=h_{alg}(\lambda\restriction_S)=0$ in view of Remark~\ref{S->Gh}, and hence $h_{alg}(\lambda)=0$ by Proposition~\ref{restr:act}.

To prove \eqref{Laaaaast:Eq} we have to verify that for every $X=\{x_0, x_1,\ldots,x_k\}\in\Pf^0(K)$ with $x_0=0$ one has 
\begin{equation}\label{Laaaaast:EqH}
H_{alg}(\lambda\restriction_S,X)=0.
\end{equation} 
Let $V$ be  the smallest $S$-invariant $\F_p$-linear subspace of $K$ that contains $X$ and let $S \overset{(\lambda\restriction_S)_V}{\curvearrowright} K$ the restricted action of $S$ on $V$. Then $H_{alg}(\lambda\restriction_S,X)=H_{alg}((\lambda\restriction_S)_V,X)$, and so \eqref{Laaaaast:EqH} is equivalent to $H_{alg}((\lambda\restriction_S)_V,X)=0$, so it is enough to prove that 
\begin{equation}\label{Laaaaast:Eqh}
h_{alg}((\lambda\restriction_S)_V)=0.
\end{equation} 

Once we limit the computation to $(\lambda\restriction_S)_V$, we can make use of Corollary~\ref{N2Ncor}. To this end we consider the submonoid $T$ of $S$ generated by $c$.  Obviously, $T\cong \N$ and we have the restricted actions
$$T \overset{\lambda\restriction_T}{\curvearrowright} K \ \mbox{ and } \ T \overset{(\lambda\restriction_T)_V}{\curvearrowright} V.$$ 

To apply Corollary~\ref{N2Ncor}, it is enough to prove that 
\begin{equation}\label{Laaaaast:EqT}
h_{alg}((\lambda\restriction_T)_V)<\infty.
\end{equation} 
To this end, note that $T$ is a cyclic monoid generated by $c$, so by Remark~\ref{mayday}(a)  
\begin{equation}\label{Tmc}
h_{alg}((\lambda\restriction_T)_V)=h_{alg}(m_c),
\end{equation}
where $$m_c=\lambda\restriction_T(c):V\to V$$ is the multiplication by $c$.
Let $U$ be the $\F_p$-linear span of $\{c^ix\colon i\in\N_+,x\in X\}$. Obviously, $U$ is $T$-invariant, and then also $S$-invariant; moreover, the transcendence of $c$ over $\F_p$ implies that $U\cong \F_p^{(\N)}$. 
Similarly, for $i\in\{1,\ldots,k\}$, each $V_i = Ux_i \cong \F_p^{(\N)}$ is $T$-invariant and $S$-invariant, and $V = V_1+\ldots + V_k$.
This determines the restricted actions $$T \overset{(\lambda\restriction_T)_{V_i}}{\curvearrowright} V_i$$
and the endomorphisms $$(\lambda\restriction_T)_{V_i}(c)=m_c\restriction_{V_i}.$$
Note that $V$ is a quotient of the vector space $W=V_1\times\ldots\times V_k$,
that $m_c$ is induced on $V$ by the multiplication $\mu_c^W$ by $c$ in $W$, and $$\mu_c^W=m_c\restriction_{V_1}\times \ldots\times m_c\restriction_{V_k}.$$ 
The multiplication $m_c\restriction_{V_i}$ by $c$ in $V_i$ acts on $V_i\cong\F_p^{(\N)}$ as the right Bernoulli shift, so $h_{alg}(m_c\restriction_{V_i})=\log p$ for every $i\in\{1,\ldots,k\}$. Hence, by Proposition~\ref{restriction_quotient}, Proposition~\ref{conju} and Proposition~\ref{wAT}, 
\begin{equation}\label{mc<infty}
h_{alg}(m_c)\leq h_{alg}(\mu_c^W)=k\log p \leq |X|\log p<\infty. 
\end{equation}

By \eqref{Tmc} and \eqref{mc<infty} we have \eqref{Laaaaast:EqT}, and so Corollary~\ref{N2Ncor} 
gives the desired equality \eqref{Laaaaast:Eqh}.
This proves \eqref{Laaaaast:EqH}, and so also \eqref{Laaaaast:Eq}. 
\end{proof}

The following result extends Proposition~\ref{casen=1}, which is  applied in its proof to cover the case $n=1$, while the general case follows from Theorem~\ref{teo:submonoid} and the case case $n=1$ (that is, Proposition~\ref{casen=1}).

\begin{corollary}\label{field}
Let $K$ be an infinite field, $n \in \N_+$ and let $G$ be the subgroup of $\mathrm{GL}_n(K)$ of upper triangular matrices. Denote by $A$ the vector space $K^n$ and by $G\overset{\alpha}\curvearrowright A$ the natural action. Then $h_{alg}(\alpha)=0$.
\end{corollary}
\begin{proof} 
Note that $G$ is solvable, hence amenable.
The case $n=1$ is covered by Proposition~\ref{casen=1}. 
Assume that $n>1$. Then the center $N=Z(G)$ of $G$ has infinite index $[G:N]$. 
 As $N$ is the subgroup of $G$ of all non-zero scalar matrices, and $N\cong K^*$ so the action $\alpha\restriction_N$ is conjugated to the action $K^*\overset{\lambda}\curvearrowright K^n$ by multiplication. Then, from the case $n=1$, Proposition~\ref{conju} and Proposition~\ref{wAT}, we obtain that $h_{alg}(\alpha\restriction_N)=0$. Therefore, $h_{alg}(\alpha)=0$ by Theorem~\ref{teo:submonoid}. 
\end{proof}

\begin{remark} Let $K$ be a field of characteristic zero having finite degree $n>1$ over $\Q$. Then the underlying abelian group $A=(K,+)$ is a 
$\Q$-vector space and $A\cong\Q^n$. 
The group $L:=\Aut(A)=\Aut_\Q(A)\cong \mathrm{GL}_n(\Q)$ is not amenable. 
Let $L\overset{\alpha}\curvearrowright A$ denote the natural action. 
By Corollary~\ref{field}, for $N= Z(L)$, we have that $h_{alg}(\alpha\restriction_N)=0$.
Hence, by Theorem~\ref{teo:submonoid}, for every amenable subgroup $G$ of $L$ that contains $N$ and with $[G:N]$ infinite, $h_{alg}(\alpha\restriction_G)=0$.
\end{remark}

The subgroups of $\Q^n$, and more generally the abelian groups with finite ranks introduced below, have the property required in the following result (see Lemma~\ref{narrow}). Under the assumption on $G$ to be non-virtually cyclic, Corollary~\ref{h<infty} strengthens Corollary~\ref{h=0all} above.

\begin{corollary}\label{h<infty}
Let $G$ be a non-torsion non-virtually-cyclic abelian group, $A$ an abelian group, and $G\overset{\alpha}{\curvearrowright} A$ a left action.  If $\halg(\phi) < \infty$ for all $\phi\in\Aut(A)$, then $\halg(\alpha)=0$.
\end{corollary}
\begin{proof}
Since $G$ is non-torsion and $G$ is not virtually cyclic, there exists an infinite cyclic subgroup $N=\langle g\rangle$ of $G$ of infinite index in $G$. Since $\halg(\alpha\restriction _N) = \halg(\alpha(g))$, by Remark~\ref{mayday}(a), and $h_{alg}(\alpha(g))<\infty$ by hypothesis, we conclude that $\halg(\alpha)=0$ by Theorem~\ref{teo:submonoid}.
\end{proof}

In order to discuss further applications of Theorem~\ref{teo:submonoid}, we recall that for an abelian group $A$ the free-rank $r_0(A) $ of $A$ is $$r_0(A)= \dim_\Q (\Q\otimes A)$$ and the $p$-rank $r_p(A)$ of $A$ is $$r_p(A) = \dim_{\Z/p\Z}A[p],$$ where $A[p] = \{x\in A\colon px = 0\}$. 

In the sequel we are interested in the class $\mathfrak N$ of \emph{abelian groups with finite ranks}, namely, abelian groups $A$ with $r_0(A) < \infty$ and $r_p(A) < \infty$ for every prime $p$. The class $\mathfrak N$ is stable under taking subgroups, quotients, and extensions (hence, finite products). Moreover, every $A\in\mathfrak N$ splits as $A= t(A) \times A_1$, where $A_1\in\mathfrak N$ is torsion-free, so $A_1$ is isomorphic to a subgroup of $\Q^n$ for some $n \in \N$.

\smallskip
We see that the abelian groups in $\mathfrak N$ have the property required in Corollary~\ref{h<infty}.

\begin{lemma}\label{narrow}
Let $A\in\mathfrak N$ and $\phi\in\Aut(A)$. Then $h_{alg}(\phi) < \infty$.
\end{lemma}
\begin{proof}
Since the Addition Theorem holds for $\halg$ in the case of single endomorphisms of $A$ (see \cite{DGB}), we have that
\begin{equation*}\label{LLL}
\halg(\phi) = \halg(\phi\restriction _{t(A)})  + \halg(\bar \phi),
\end{equation*}
where $\bar \phi: A/t(A) \to A/t(A)$ is the automorphism induced by $\phi$. Hence, it suffices to see that $\halg(\phi\restriction _{t(A)})< \infty$ and $h_{alg}(\bar \phi)<\infty$.

Since the quotient $A/t(A)$ is torsion free and has finite rank, it is isomorphic to a subgroup of $\Q^n$ for some $n\in \N$. As recalled above in the proof of Corollary~\ref{Coro:application1}, $\halg(\bar \phi)<\infty$. 

Now we show that $\halg(\phi\restriction _{t(A)})=0$. For every prime $p$ consider the $p$-primary component $t_p(A)$ of $t(A)$, which is fully invariant.
Since $r_p(A)<\infty$, the subgroup $A[p^n]$ of $t_p(A)$ is finite and fully invariant for every $n\in\N$; therefore, every finite subset of $t_p(A)$ is contained in a finite fully invariant subgroup of $t_p(A)$ (see \cite{DGSZ}). This yields that $\halg(\phi\restriction_{t_p(A)})=0$ for every prime $p$, so we can make use of the Addition Theorem for torsion abelian groups from \cite{DGSZ} to get that $\halg(\phi\restriction _{t(A)})=0$.
\end{proof}

The next result generalizes Corollary~\ref{Coro:application1}, it follows directly from Corollary~\ref{h<infty} and Lemma~\ref{narrow}.
  
\begin{corollary}\label{Coro:application}
Let $G$ be a non-torsion non-virtually cyclic abelian group, $A\in\mathfrak N$, and $G\overset{\alpha}{\curvearrowright} A$ a left action. Then $\halg(\alpha) = 0$.
\end{corollary}

The class of abelian groups with all automorphisms of finite algebraic entropy, which contains $\mathfrak N$ according to Lemma~\ref{narrow}, was studied in \cite{DGSZ} for torsion abelian groups and in \cite{DS} in the general case, yet no complete description of this class is known so far.


\subsection{Entropy of quotient actions}\label{sec:quotient-action}

In this section we consider the following general setting. Let $G \overset{\alpha}{\curvearrowright} A$ be a left action of an amenable group $G$ on an abelian group $A$, and let $N$ be a normal subgroup of $G$ such that $N\subseteq\ker \alpha$, with $\pi: G \to G/N$ be the canonical projection. Denote by $$G/N\overset{\bar\alpha_{G/N}}{\curvearrowright} A$$ the \emph{quotient action} induced by $\alpha$, that is, for $x\in A$ and $h\in G/N$ let $\bar\alpha_{G/N}(h)(x)=\alpha(g)(x)$ for any $g\in G$ with $\pi(g) =h$.

\medskip
Under these hypotheses, and in particular in view of the assumption $N\subseteq \ker\alpha$, one could expect that $h_{alg}(\alpha)=h_{alg}(\bar\alpha_{G/N})$, but this is not the case as item (b) of the following example shows.

\begin{example}\label{gamma=alpha}
\begin{enumerate}[(a)]
\item Let $A$ be an abelian group, $L$ an amenable group, and $L\overset{\gamma}{\curvearrowright}A$ a left action. Let $N$ be an amenable group, $G=N\times L$ and let $G\overset{\alpha}{\curvearrowright}A$ be defined by 
$$\alpha(k,h)(a) = \gamma(h)(a)\quad \text{for every}\ (k,h)\in G.$$
Then $N\subseteq \ker\alpha$ and $\alpha\restriction_L=\gamma$, moreover $\bar\alpha_{G/N}$ is weakly conjugated to $\gamma$. So, by Proposition~\ref{conju}
\begin{equation}\label{New:Eqq}
\halg(\bar\alpha_{G/N}) = \halg(\gamma)\  \mbox{ and }\ \ent(\bar\alpha_{G/N}) = \ent(\gamma).  
\end{equation}
We verify that
\begin{equation}\label{New:Eqq1}
h_{alg}(\alpha)=\begin{cases}0 & \text{if $N$ is infinite}\\\frac{h_{alg}(\gamma)}{|N|} & \text{if $N$ is finite}\end{cases}\quad\text{and}\quad \ent(\alpha)=\begin{cases}0 & \text{if $N$ is infinite}\\\frac{\ent(\gamma)}{|N|} & \text{if $N$ is finite}\end{cases}.
\end{equation}
The case when $N$ is finite is settled by Proposition~\ref{finind}. So, assume that $N$ is infinite.
Let $(N_i)_{i \in I}$ be a right F\o lner net of $N$ and $(L_j)_{j\in J}$ be a right F\o lner net of $L$. Then $(N_i\times L_j)_{(i,j)\in I\times J}$ is a right F\o lner net of $G$ by Lemma~\ref{Fnprod}.
Let $X\in\Pf^0(A)$. For every $(i,j)\in I\times J$, we have that $$T_{N_i \times L_j}(\alpha, X) = T_{L_j}(\gamma, X).$$
Since the limit $H_{alg}(\gamma,X) = \lim_{j\in J} \frac{\ell(T_{L_j}(\gamma, X))}{\card{L_j}}$ is finite, there exists $j_0\in J$ such that, for $C=  H_{alg}(\gamma,X)+1$, $$\frac{\ell(T_{L_j}(\alpha, X))}{\card{L_j}} \leq C$$ for all $j \geq j_0$.
Therefore, $$\frac{\ell(T_{N_i \times L_j}(\alpha, X))}{\card{N_i\times L_j}} = \frac{1}{\card{N_i}} \frac{\ell(T_{L_j}(\gamma, X))}{\card{L_j}} \leq  \frac{C}{\card{N_i}}$$
for all $j \geq j_0$. Since $\lim_{i\in I} \frac{1}{\card{N_i}} = 0$, we deduce that 
$$H_{alg}(\alpha, X) = \lim_{(i,j)\in I\times J} \frac{\ell(T_{N_i \times L_j}(\alpha,X))}{\card{N_i \times L_j}}\leq\lim_{i\in I}\frac{C}{\card{N_i}} = 0.$$
Since the above is true for every $X\in\Pf^0(A)$, we conclude that $h_{alg}(\alpha) =0$. The second equality follows from the first one and Proposition~\ref{h=ent}.

\item We discuss the equality $h_{alg}(\alpha)=h_{alg}(\bar\alpha_{G/N})$ in the example in item (a); in view of Equation~\eqref{New:Eqq}, we need to see when $h_{alg}(\alpha)=h_{alg}(\gamma)$. 

We use Equation~\eqref{New:Eqq1}.
If $N$ is infinite, then $h_{alg}(\alpha)=0$. If $N$ is finite, and $h_{alg}(\gamma)=0$ or $h_{alg}(\gamma)= \infty$, then $h_{alg}(\alpha)=h_{alg}(\gamma)$. 
If $0 < h_{alg}(\gamma) <\infty$, then $h_{alg}(\alpha) = h_{alg}(\gamma)$ precisely when $N$ is trivial.

For specific examples consider $L = \Z$, the right Bernoulli shift $$\beta=\beta_{K}:K^{(\Z)}\to K^{(\Z)},\quad (x_n)_{n\in\Z}\mapsto(x_{n-1})_{n\in\N},$$ and the action $L\overset{\alpha_\beta}{\curvearrowright}A = K^{(\Z)}$ induced by $\beta$.
In case $K=\Z(p)$ for a prime $p$, one has $\ent(\alpha_\beta)=h_{alg}(\alpha_\beta)=\halg(\beta) = \ent(\beta) = \log p$.
To get an action with $h_{alg}(\gamma) =\infty$ it is enough to take $K = \Z$.  
\end{enumerate}
\end{example}

Item (a) of Example~\ref{gamma=alpha} is a particular case of the following general result on quotient actions. While the computation in the above example is elementary, the proof of the next theorem uses Theorem~\ref{Fubini}.

\begin{theorem}\label{quotient}
Let $G$ be an amenable group, $A$ an abelian group, and $G\overset{\alpha}{\curvearrowright}A$ a left action. Let $N$ be a normal subgroup of $G$ such that $N\subseteq\ker\alpha$ and consider $G/N\overset{\bar\alpha_{G/N}}{\curvearrowright}A$. Then 
$$h_{alg}(\alpha)=\begin{cases}0&\text{if $N$ is infinite},\\ \frac{h_{alg}(\overline\alpha_{G/N})}{\card{N}}&\text{if $N$ is finite}.\end{cases}
\quad\text{and}\quad
\ent(\alpha)=\begin{cases}0&\text{if $N$ is infinite},\\ \frac{\ent(\overline\alpha_{G/N})}{\card{N}}&\text{if $N$ is finite}.\end{cases}$$
\end{theorem}
\begin{proof} 
Let $C=G/N$, $\gamma=\bar\alpha_{G/N}$ and fix a section $\sigma:C\to G$. For $B\in\Pf(A)$, by definition $H_{alg}(\alpha,B)=\mathcal H_G(f_B)$.
By Theorem~\ref{Fubini}, $$H_{alg}(\alpha,B)=\mathcal H_C(\theta),$$ where 
$$\theta=\Theta_\sigma(f_B):\Pf(C)\to\R_+,\ Y\mapsto \mathcal H_N((f_B)^{\sigma(Y)}).$$
Let $Y\in\Pf(C)$ and $X\in\Pf(N)$ with $m= |X|$. Since $N\subseteq \ker\alpha$,
$$T_{X\sigma(Y)}(\alpha,B)=T_X(\alpha,T_{\sigma(Y)}(\alpha,B))= T_{\sigma(Y)}(\alpha,B)_m=T_{\sigma(Y)}(\alpha,B_m)=T_Y(\gamma,B_m),$$
where in the next-to-the-last equality Lemma~\ref{next-to-the-last} applies.
Therefore,
\begin{equation}\label{Fub1}
(f_B)^{\sigma(Y)}(X)=f_B(X\sigma(Y))=\ell(T_{X\sigma(Y)}(\alpha,B))=\ell(T_{\sigma(Y)}(\alpha,B)_m)= \ell(T_Y(\gamma,B_m)).
\end{equation}
Here we consider two cases. 

\smallskip 
\noindent {\bf Case 1.} Assume that $N$ is infinite. Then $(f_B)^{\sigma(Y)}\restriction_{\Pf(N)}$ needs not be a constant function if $B$ is not a subgroup of $A$ (i.e., $B_m \ne B$). Nevertheless, in this case we can find a convenient estimate of this function using the folklore fact that  for every finite subset $W$ of $A$ one has $|W_m| \leq (m+1)^{|W|}$, so $\ell(W_m) \leq |W| \log(m+1)$.  Hence,  Equation~\eqref{Fub1} gives
\begin{equation}\label{Fub2}
(f_B)^{\sigma(Y)}(X)=\ell(T_Y(\gamma,B)_m \leq |T_Y(\gamma,B)| \log(m+1).
\end{equation}
 where $|T_Y(\gamma,B)|$ is constant with respect to $X$. 

Now we can compute $\theta$ by taking $X$ to be a generic member of a right F\o lner net $(N_j)_{j\in J}$ of $N$. From Equation~\eqref{Fub2} we deduce that, for a right F\o lner net $(C_i)_{i\in I}$ of $C$, for every $i\in I$
$$\theta(C_i)=\mathcal H_N((f_B)^{\sigma(C_i)})\leq  \lim_{j\in J}\frac{|T_{C_i}(\gamma,B)| \log(|N_j|+1)}{\card{N_j}} = 0,$$
as $|N_j|\to \infty$. Therefore, $H_{alg}(\alpha,B)=\mathcal H_C(\theta)= 0$ by Theorem~\ref{Fubini}. Consequently, since $B\in\Pf(A)$ is arbitrary, we can conclude that $h_{alg}(\alpha)=0$, if $N$ is infinite.

\smallskip 
\noindent {\bf Case 2.} Assume that $N$ is finite and set $X=N$, so $m=|N|=|X|$ in Equation~\eqref{Fub1}. Pick a right F\o lner net $(C_i)_{i\in _I}$ of $C$ and using Theorem~\ref{Fubini} along with Equation~\eqref{Fub1} (applied to $Y=C_i$) conclude that
\begin{equation}\label{LL1}
H_{alg}(\alpha,B)=\lim_{i\in I}\frac{\mathcal H_N((f_B)^{\sigma(C_i)})}{\card{C_i}}=\lim_{i\in I}\frac{\frac{\ell(T_{C_i}(\gamma,B_m))}{m}}{\card{C_i}}=\frac{1}{m}\lim_{i\in I}\frac{\ell(T_{C_i}(\gamma,B_m))}{\card{C_i}}=\frac{1}{m}H_{alg}(\gamma,B_m).
\end{equation}
Therefore, $$H_{alg}(\alpha,B)=\frac{1}{m}H_{alg}(\gamma,B_m)\leq \frac{1}{m}h_{alg}(\gamma),$$ and since $B\in\Pf(A)$ is arbitrary, we conclude that 
\begin{equation}\label{LL}
h_{alg}(\alpha)\leq\frac{1}{m}h_{alg}(\gamma).
\end{equation}

 To prove the opposite inequality consider first the case when $h_{alg}(\gamma)<\infty$ and fix an $\eps>0$. There exists $B\in\Pf(A)$ such that $H_{alg}(\gamma,B)\geq h_{alg}(\gamma) - \eps$.  Then Equation~\eqref{LL1} gives
$$h_{alg}(\alpha)\geq H_{alg}(\alpha,B) = \frac{1}{m}H_{alg}(\gamma,B_m)\geq  \frac{1}{m} H_{alg}(\gamma,B)\geq \frac{1}{m}h_{alg}(\gamma) - \eps.$$
Along with Equation~\eqref{LL}, this proves that $h_{alg}(\alpha) = \frac{1}{m}h_{alg}(\gamma)$. 

A similar argument works in the case $h_{alg}(\gamma)=\infty$. This completes the proof of the first assertion. 
The second assertion follows from the first one and Proposition~\ref{h=ent}. 
\end{proof}

In the following direct consequence of the above theorem we see that the algebraic entropy of $G\overset{\alpha}{\curvearrowright}A$ is always smaller that the algebraic entropy of $G/N\overset{\bar\alpha_{G/N}}{\curvearrowright}A$ for $N$ a normal subgroup of $G$ contained in $\ker\alpha$. Moreover, we see that the algebraic entropy of these two actions is the same only in special cases.

\begin{corollary}\label{quot:act}
Let $G$ be an amenable group, $A$ an abelian group, $G\overset{\alpha}{\curvearrowright}A$ a left action, and $N$ a normal subgroup of $G$ such that $N\subseteq\ker\alpha$. Then 
$$h_{alg}(\alpha)\leq h_{alg}(\bar\alpha_{G/N})\quad \text{and}\quad \ent(\alpha)\leq\ent(\bar\alpha_{G/N}).$$
Furthermore, $h_{alg}(\alpha)=h_{alg}(\bar\alpha_{G/N})$ if and only if either $h_{alg}(\bar\alpha_{G/N})=0$, or $N=\{1\}$, or $h_{alg}(\bar\alpha_{G/N})= \infty$
and $N$ is finite. 
The same assertions hold for $\ent$.
\end{corollary}

The inequalities in the above corollary can be strict also when $N$ is infinite, as shown by Example~\ref{gamma=alpha}(b).

\begin{corollary}\label{tfcor}
Let $G$ be a torsion-free amenable group, $A$ an abelian group, and $G\overset{\alpha}{\curvearrowright}A$ a left action. 
If $\halg(\alpha) > 0$, then the action is faithful (i.e., $\ker\alpha=\{1\})$.
\end{corollary}
\begin{proof}
Assume that $N= \ker \alpha$ is non-trivial. Then $N$ is infinite, as $N$ is torsion-free. By Theorem~\ref{quotient} applied to $G$ and $N$ we deduce that  $\halg(\alpha)=0$, a contradiction. 
\end{proof}

\begin{corollary}
Let $G$ be a non-abelian torsion-free amenable group, $A$ an abelian group, and $G\overset{\alpha}{\curvearrowright} A$ a left action. If $\mathrm{Aut}(A)$ is  abelian, then $h_{alg}(\alpha) = 0$. 
In particular, all actions $G\overset{\alpha}{\curvearrowright} \Z(p^\infty)$, $G\overset{\alpha}{\curvearrowright} \Q/\Z$ and $G\overset{\alpha}{\curvearrowright} \mathbb J_p$ have $\halg(\alpha) = 0$.
\end{corollary}
\begin{proof} 
Since $G$ is non-abelian and $\mathrm{Aut}(A)$ is abelian, $\ker\alpha$ is non-trivial. Hence, $\halg(\alpha)=0$ by Corollary~\ref{tfcor}.

To prove the second assertion, note that $\mathrm{Aut}(\Z(p^\infty)) \cong \mathrm{Aut}(\mathbb J_p)
\cong U(\mathbb J_p)$ is abelian for every prime $p$. So, in view of the isomorphism $\Q/\Z \cong \bigoplus_p \Z(p^\infty)$, the group 
$\mathrm{Aut}(\Q/\Z)\cong\prod_{p\in\mathbb P}\mathrm{Aut}(\Z(p^\infty))\cong\prod_{p\in\mathbb P}U(\mathbb J_p)$ is abelian as well. 
\end{proof}

When $A$ is one of the groups $\Z(p^\infty)$ and $\Q/\Z$, the conclusion of the above theorem remains true also in case $G$ is abelian (see Corollary~\ref{h=0all}). 
However, one can produce automorphisms $\phi$ of $A=\mathbb J_p$ with $\halg(\phi)=\infty$; indeed, taking for example as $\phi$ the multiplication $x\mapsto \xi x$ 
by an invertible $p$-adic number $\xi$ that is transcendental over $\Z$, one has that the subgroup $B$ of $A$ generated by $\xi$ is $\phi$-invariant and  isomorphic to $\bigoplus_\Z\Z$, and  so $\halg(\phi)=\halg(\phi\restriction_B) = \infty$ since $\phi$ acts on $B\cong \bigoplus_\Z\Z$ as the two-sided Bernulli shift (see \cite{DGB0}).

\section{Addition Theorem}\label{AT:sec}

This section is dedicated to the proof of the Addition Theorem for the algebraic entropy stated in the introduction.


\subsection{The function $\ell(-,-)$}

We start defining the useful auxiliary function $\ell(-,-)$. Let $A$ be an abelian group, $B\leq A$ and $\pi:A\to A/B$ the canonical projection. For $Y\in\P(A)$ let
\begin{equation}\label{Lodz}
\ell(Y,B) = \ell(\pi(Y)). 
\end{equation}
Even if $Y$ is not necessarily a subgroup, we shall write sometimes $(Y + B) / B$ for the subset $\pi(Y)=\{y+B\colon y\in Y\}$ of $A/B$ in order to avoid the explicit use of $\pi$.
Clearly, it may happen that $B$ or $Y$ are infinite, but $\ell(Y,B)$ is  finite. From the known properties of $\ell(-)$ we obtain that, for every $Y,Y'\in\P(A)$,
\begin{equation}\label{Lodz1}
\ell(Y + Y',B) \leq \ell(Y,B) + \ell(Y',B).
\end{equation}

\begin{remark}\label{2019}
The the utility of the auxiliary function $\ell(-,-)$ is best illustrated in the calculation of the length of  orbits. Indeed, it allows for a transferring of the computation from the quotient $A/B$ of an abelian group $A$ to the group $A$ itself.  More specifically, if $S$ is a semigroup, $S\overset{\alpha}\curvearrowright A$ a left action, $B$ an $\alpha$-invariant subgroup of $A$, and $\pi:A\to A/B$ the canonical projection, then for $F\in\Pf(S)$ and $Y\in\Pf^0(A)$ one has 
\begin{equation}\label{1May2019}
\ell(T_F(\alpha_{A/B},\pi(Y)))=\ell(T_F(\alpha,Y),B),
\end{equation} 
due to the obvious equality $T_F(\alpha_{A/B},\pi(Y))=\pi(T_F(\alpha,Y))$.
\end{remark}
 
We start giving a basic property of the function $\ell(-,-)$.

\begin{lemma}\label{carina} Let $A$ be an abelian group and let $B,C,D$ be subgroups of $A$. Then:
\begin{enumerate}[(a)]
  \item $\ell(C,B)=\ell(C,B\cap C)$;
  \item $\ell(C,B)=\ell(C,D)$ if $B\cap C=D\cap C$.
\end{enumerate}
\end{lemma}
\begin{proof}
(a) is clear and (b) follows from (a).
\end{proof}

In the following lemma we collect other useful properties of the function $\ell(-,-)$. 

\begin{lemma}\label{lem:ell}
Let $A$ an abelian group, $X,X'\in\Pf^0(A)$, $C\in\mathcal F(A)$, and $B, B'$ be subgroups of $A$. Then:
\begin{enumerate}[(a)]
\item the function $\ell(Y, B)$ is increasing in $Y$ and decreasing in $B$;
\item $\ell(X + C) = \ell(X , C) + \ell(C)$ and $\ell(X+ C , B) = \ell(X , C + B) + \ell(C , B)$;
\item  $\ell(X + X' , B + B') \leq \ell(X ,B) + \ell(X', B')$.
\end{enumerate}
\end{lemma}
\begin{proof}
(a) and (b) are obvious.

(c) By \eqref{Lodz1} and item (a), $$\ell(X+X',B+B')\leq\ell(X,B+B')+\ell(X',B + B')\leq \ell(X,B)+\ell(X',B').$$
This concludes the proof.
\end{proof}

Further properties of the function $\ell(-,-)$ follow, obtained from properties of trajectories from Lemma~\ref{lemma?}.

\begin{lemma}\label{lem:ell2}
Let $S$ be a semigroup, $A$ an abelian group, and $S\overset{\alpha}\curvearrowright A$ a left action. Let $F,F'\in\Pf(S)$, $X\in\Pf^0(A)$ and $B\leq A$. Then:
\begin{enumerate}[(a)]
\item $\ell(\alpha(g)(X),\alpha(g)(B)) \leq \ell(X , B)$ for every $g \in S$;
\item if $F' \subseteq F$, then $$\ell(T_{F'}(X),B) \leq \ell(T_F(X),B) \quad \text{and}\quad \ell(X ,T_F(B))\leq \ell(X , T_{F'}(B));$$
\item $\ell(T_F(X),T_F(B)) \leq \card F\, \ell(X , B)$;
\item if $B$ is $\alpha$-invariant, then $\ell(T_F(X),B) \leq \card F\, \ell(X , B)$.
\end{enumerate}
\end{lemma}
\begin{proof}
(a) The map $\alpha(g): A \to A$ induces a surjective homomorphism $(X + B)/B\to (\alpha(g)(X+B))/\alpha(g)(B)$.

(b) Since $0 \in X$, we have $T_{F'}(X) \subseteq T_F(X)$, and similarly $T_{F'}(B) \subseteq T_F (B)$. Hence, (a) applies, along with Lemma ~\ref{lem:ell}(a).

(c) By Lemma~\ref{lem:ell}(c) and item (a), $$\ell(T_F(X),T_F(B)) \leq \sum_{g \in F} \ell(\alpha(g)(X),T_F(B)) \leq \sum_{g \in F} \ell(\alpha(g)(X), \alpha(g)(B)) \leq \sum_{g \in F} \ell(X , B) = \card F \ell(X , B).$$

(d) Since $T_F(B) \leq B$, by Lemma~\ref{lem:ell}(a) and item (c), $$\ell(T_F(X),B) \leq \ell(T_F(X),T_F(B)) \leq \card F\ell(X , B).$$
This concludes the proof.
\end{proof}

We conclude with another lemma on the function $\ell(-,-)$.

\begin{lemma}\label{lem:ell3}
Let $S$ be a semigroup, $A$ an abelian group, and $S\overset{\alpha}\curvearrowright A$ a left action.
If $F_1, \dotsc, F_n \in\Pf(S)$, $C_1, \dotsc, C_n \in\mathcal F(A)$, and $B_1, \dotsc, B_n \leq A$, then
$$\ell(T_{F_1}(C_1)+ \ldots + T_{F_n}(C_n), T_{F_1}(B_1)+\ldots+T_{F_n}(B_n)) \leq \sum_{i=1}^n \card{F_i}\, \ell(C_i , B_i).$$
\end{lemma}
\begin{proof} 
Applying Lemma~\ref{lem:ell}(c), and taking into account that $\ell(T_{F_i}(C_i), T_{F_i}(B_i)) \leq |F_i|\ell(C_i , B_i)$ by item (c) of Lemma
\ref{lem:ell2}, we conclude that
$$\ell(T_{F_1}(C_1)+ \ldots + T_{F_n}(C_n), T_{F_1}(B_1)+\ldots+T_{F_n}(B_n)) \leq \sum_{i=1}^n \ell(T_{F_i}(C_i), T_{F_i}(B_i)) \leq \sum_{i=1}^n\card{F_i} \ell(C_i , B_i),$$
hence we have the thesis.
\end{proof}

\subsection{The Filling Theorem}

In the sequel we expose a result from \cite{CCK}, and some of its consequences, that we apply in the next subsection to prove the Addition Theorem.

Let $S$ be a cancellative semigroup. For every $D, E \subseteq S$, define $$\partial_E(D)= \set{s \in D: (s E) \setminus D \neq \emptyset}.$$
In view of \cite[Proposition 2.4]{CCK}, $(F_i)_{i \in \N}$ is a right F\o lner net of $S$ if and only if, for every $E \in\Pf(S)$,
$$\lim_{i\in I}\frac{\card{\partial_{E}(F_i)}}{\card{F_i}} = 0.$$

The following theorem is exactly \cite[Theorem~3.8]{CCK} for $S^{op}$. Actually, for a set $X$ and $\eps>0$, a family $(Y_j)_{j\in J}$ in $\Pf(X)$ is \emph{$\eps$-disjoint} if there exists a family $(Z_j)_{j\in J}$ of pairwise disjoint subsets of $X$ such that $ Z_j\subseteq Y_j$ and $(1-\eps)\card{Y_j}\leq\card{Z_j}$ for every $j\in J$.
 
\begin{theorem}[Filling Theorem]\label{thm:filling}
Let $S$ be a cancellative semigroup.
For every $0 < \eps \leq 1/2$, there exists an integer $N = N(\eps)\geq 1$ such that, if $(F_j)_{j\in\{1\ldots,N\}}$ is a finite sequence of non-empty finite subsets of $S$ such that
\begin{equation}\label{eq:filling-1}
\frac{ \card{\partial_{F_j}(F_k)}}{\card{F_k}} \leq \frac{\eps^{2N}}{\card{F_j}} \quad \forall 1 \leq j < k \leq N
\end{equation}
and $D \subseteq S$ is a non-empty finite subset of $S$ such that
\begin{equation}\label{eq:filling-2}
\frac{\card{\partial_{F_j}(D)}}{\card{D}} \leq \eps^{2N} \quad \forall 1 \leq j \leq N,
\end{equation}
then there exists a finite sequence $(P_j)_{j \in\{1,\ldots,N\}}$ of finite subsets of $S$, such that:
\begin{enumerate}[(1)]
\item for every $j\in\{1,\ldots,N\}$, the family $(s F_j)_{s \in P_j}$ is $\eps$-disjoint\footnote{and $P_j \subseteq D\setminus \partial_{F_j}(D)$ -- we
separated this property, since we do not use it in the sequel.};
\item the subsets $P_j F_j$, $j\in\{1,\ldots,N\}$, are contained in $D$ and pairwise disjoint;
\item $U := \bigcup_{j=1}^N P_j F_j\subseteq D$ satisfies $\card{D \setminus U} \leq \eps \card D$,
\end{enumerate}
\end{theorem}


In the notations of Theorem~\ref{thm:filling}, let $$b= \sum_{j=1}^N \card{P_j}\card{F_j},\quad u= \card U,\quad d= \card D.$$
We need the following inequality.

\begin{lemma}\label{lemma_ignoto}
In the above notations, $u \leq b$ and 
$$b-u \leq \eps b$$
\end{lemma}
\begin{proof}
The sets $P_j F_j$, with $j\in\{1,\ldots,N\}$, are pairwise disjoint, and so $u = \sum_{j=1}^N \card{P_j F_j}$.
Fix $j\in\{1,\ldots,N\}$. Since the family $(s F_j)_{s \in P_j}$ is $\eps$-disjoint, by \cite[Lemma 3.2]{CCK} we have
$$(1 - \eps) \card{P_j} \card{F_j}\leq \card{P_j F_j},$$
hence
$$\card{P_j}\card{F_j} - \card{P_j F_j} \leq \eps \card{P_j}\card{F_j}.$$
A summation with respect to $j$ concludes the proof.  
\end{proof}

\begin{definition}\label{Def:ti}\label{def:tiling}
Let $S$ be a semigroup.
Fix $\eps > 0$, let $D \in \Pf(S)$ and $n \in \N_+$. The $n$-tuple $(F_1, \dotsc, F_n)$ in $\Pf(S)$ is an \emph{$\eps$-tiling} of $D$, witnessed by the $n$-tuple $(P_1, \dotsc , P_n)$ in $\Pf(S)$, if the subsets $P_i F_i$, $i\in\{1,\ldots,n\}$, are pairwise disjoint, and denoting
$$U= \bigcup_{i=1}^n P_i F_i,\quad d= \card D,\quad u = \card U,\quad b = \sum_{i=1}^n \card{P_i}\card{F_i},$$
\begin{enumerate}[(1)]
\item $U  \subseteq D$;
\item $\card{D \setminus U}={d - u} < \eps d$;
\item $0 \leq b - u  < \eps b$.
\end{enumerate}
\end{definition}

We need the following property of $\eps$-tilings.

\begin{lemma}\label{remtil}
Let $S$ be a semigroup. Fix $\eps > 0$, let $D \in \Pf(S)$ and $n \in \N_+$. If $(F_1, \dotsc, F_n)$ is an $\eps$-tiling of $D$, witnessed by $(P_1, \dotsc , P_n)$, then $u\leq b$ and 
\begin{equation}\label{reee}
\abs{\frac 1 d - \frac 1 b} < \frac{2\eps}{b}, 
\end{equation}
\end{lemma}
\begin{proof}
From the inequalities $0\leq d- u \leq \eps d$ and $0\leq  b-u \leq \eps b$, given respectively by (2) and (3) in the above definition,
we deduce that both $d$ and $u$ belong to the interval $[u,\frac{u}{1-\eps}]$. Hence, $|d-b|\leq \frac{\eps u}{1-\eps}\leq 2\eps u\leq 2\eps d$. Dividing  both sides of the inequality $|d-b|\leq 2\eps d$ by $bd$ we obtain~\eqref{reee}. 
\end{proof}


The following is an important consequence of the Filling Theorem, that we apply in the proof of the Addition Theorem in the particular case when $\mathfrak F=\mathfrak D$.

\begin{corollary}\label{cor:tiling}
Let $S$ be a cancellative right amenable semigroup, and let $\FolSeq=(F_i)_{i \in I}$ and $\mathfrak D=(D_j)_{j \in J}$ be two right F\o lner nets of $S$.
Fix $0 < \eps < 1/2$, and let $N = N(\eps)$ be as in Theorem~\ref{thm:filling}.
Then, there exist $F_1, \dotsc, F_N \in \FolSeq$ and $J' \subseteq J$ cofinal, such that, for every $j \in J'$, $(F_1, \dotsc, F_N)$ is an $\eps$-tiling of $D_j$.
\end{corollary}
\begin{proof}
We can find $F_1, \dotsc, F_N$ in $\FolSeq$ satisfying~\eqref{eq:filling-1} in Theorem~\ref{thm:filling}.
Let $J'$ be the set of all $j \in J$ such that $D_j$ satisfies~\eqref{eq:filling-2} in Theorem~\ref{thm:filling}.
It is easy to see that $J'$ is cofinal in $J$, so  $(F_j)_{j\in J'}$ is a subnet of $(F_j)_{j\in J}$.

Fix $j \in J'$. We can apply Theorem~\ref{thm:filling}, and get $P_1, \dotsc, P_N$ finite subsets of $S$ as in the theorem.
We claim that $(F_1, \dotsc, F_N)$ is an $\eps$-tiling of $D_j$, witnessed by $(P_1, \dotsc, P_N)$. Let $U$, $d$, $u$, $b$ as in Definition~\ref{def:tiling}.
It is clear that $U \subseteq D_j$, since $P_i F_i \subseteq D_j$ for all $i\in\{1,\ldots,N\}$ by Theorem~\ref{thm:filling}(2). This yields that $\card{D \setminus U} = d - u$, and furthermore $d-u<\eps d$ Theorem~\ref{thm:filling}(3). Finally, Lemma~\ref{lemma_ignoto} shows that $b-u< \eps b$.
\end{proof}

\begin{remark}[See also \cite{OW}]
In \cite[Definition~3.6]{CCK}, for a semigroup $S$, given $K,D\in \Pf(S)$ and $\eps>0$, they define an $(\eps, K)$-filling pattern for a set $D$.
Our notion of $\eps$-tiling is related to it, but different. In fact, an $(\eps, K)$-filling pattern uses only one ``tile'' (the set $K$), while an $\eps$-tiling uses $n$ different tiles (the sets $P_1, \dotsc, P_n$). Moreover, the resulting set $U$ for $\eps$-tilings is ``large'' in $D$, in the sense that $\card{D  \setminus U} < \eps \card D$.
\end{remark}

\subsection{Proof of the Addition Theorem}

The following result covers one inequality of the Addition Theorem. 

\begin{proposition}\label{AThalf}
Let $S$ be a cancellative right amenable semigroup, $A$ be a torsion abelian group, $S\overset{\alpha}{\curvearrowright}A$ a left action, and $B$ an $\alpha$-invariant subgroup of $A$. Then
\begin{equation*}
\ent(\alpha) \geq \ent(\alpha_B) + \ent(\alpha_{A/B}).
\end{equation*}
\end{proposition}
\begin{proof} 
Let $\pi: A \to A/B$ be the canonical projection and  let $(F_i)_{i\in I}$ be a right F\o lner net of $S$. Let $X\in\mathcal F(B)$ and $Z \in\mathcal F(A/B)$.  Since $A$ is torsion, we can find $Y\in\mathcal F(A)$ such that $\pi(Y) = Z$. Let $Y' = Y + X\in\mathcal F(A)$. Then $\pi(Y') = Z$, and let $X'= Y' \cap B\in\mathcal F(B)$, so that $X\subseteq X'$. Therefore, by Lemma~\ref{cofinal}(a), 
$$H_{alg}(\alpha_B, X) \leq H_{alg}(\alpha_B, X') = H_{alg}(\alpha, X').$$

The exact sequence $0 \to X' \to Y' \to Z \to 0$ gives rise, for every $i\in I$, to the sequence 
$$0 \to T_{F_i}(\alpha_B, X') \overset{f}\rightarrow T_{F_i}(\alpha, Y') \overset{g}{\rightarrow}T_{F_i}(\alpha_{A/B}, Z) \to 0,$$
where $g= \pi\restriction_{T_{F_i}(\alpha, Y')}$, while $f$ is the inclusion map. 
This sequence need not be exact any more (as the kernel of the map $g$ may properly contain the image of $f$), nevertheless we have that
$$\ell(T_{F_i}(\alpha_{A/B}, Z)) + \ell(T_{F_i}(\alpha_{B}, X')) \leq \ell(T_{F_i}(\alpha_{A/B}, Z)) + \ell(\ker g) = \ell(T_{F_i}(\alpha, Y')).$$
Dividing by $|F_i|$ and taking the limit, we conclude that
$$H_{alg}(\alpha_{A/B}, Z) + H_{alg}(\alpha_B, X) \leq H_{alg}(\alpha_{A/B}, Z) + H_{alg}(\alpha_B, X') \leq H_{alg}(\alpha, Y') \leq \ent(\alpha).$$
To end the proof it suffices to take the supremum over all $Z\in\mathcal F(A/B)$ and $X\in\mathcal F(B)$.
\end{proof}

Now we prove the ``second half'' of the Addition Theorem.

\begin{proposition}
Let $S$ be a right amenable monoid, $A$ a torsion abelian group, $\alpha$ a left action of $S$ on $A$, and $B$ an $\alpha$-invariant subgroup of $A$. Then
$$\ent(\alpha) \leq \ent(\alpha_B)  +\ent (\alpha_{A/B}).$$
\end{proposition}
\begin{proof}
Let $Y\in\mathcal F(A)$ and fix $\eps > 0$.
Let $\pi : A \to A/B$ be the canonical projection, and $Z = \pi(Y)$.
To prove the thesis, it suffices to show that
\begin{equation}\label{aim1}
H_{alg}(\alpha, Y) \leq_{5\eps} \ent(\alpha_B)  + H_{alg}(\alpha_{A/B}, Z).
\end{equation}

Let $\FolSeq= (F_i)_{i\in I}$ be a right  F\o lner net of $S$, such that $1 \in F_i$ for every $i \in I$. By definition and by Remark~\ref{2019},
$$H_{alg}(\alpha, Y) = \lim_{i\in I} \frac{\ell(T_{F_i}(\alpha,Y))}{\card{F_i}}\quad\text{and}\quad H_{alg} (\alpha_{A/B}, Z) = \lim_{i\in I} \frac{\ell(T_{F_i}(\alpha_{A/B},Y), B)}{\card{F_i}}.$$
Therefore, after taking a subnet of $\FolSeq$ if necessary, we have that for every $i \in I$
\begin{equation}\label{laaast:eq}
H_{alg}(\alpha, Y) \eqeps \frac{\ell(T_{F_i}(\alpha,Y))}{\card{F_i}}\quad\text{and}\quad H_{alg}(\alpha_{A/B}, Z) \eqeps \frac{\ell(T_{F_i}(\alpha_{A/B},Y), B)}{\card{F_i}}.
\end{equation}

Let 
\begin{equation}\label{beps}
\beps=\frac{\eps}{2\ell(Y)}.
\end{equation} 
By Corollary~\ref{cor:tiling}, there exist $N\in\N_+$, $F_1, \dotsc, F_N \in \FolSeq$ and a subnet $(F_j)_{j\in J}$ of $\FolSeq$, such that $(F_1, \dotsc, F_N)$ is an $\beps$-tiling of $F_j$ for every $j \in J$. 
This means that, for every fixed $j\in J$, letting $(P_1, \dotsc , P_N)$ in $\Pf(S)$ be the $N$-uple witnessing that $(F_1, \dotsc, F_N)$ is an $\beps$-tiling of $F_j$, and moreover
\begin{equation}\label{notj}
d= \card{F_j},\quad U= \bigcup_{i=1}^N P_i F_i,\quad  u= \card U,\quad b = \sum_{i=1}^N \card{P_i}\card{F_i},
\end{equation}
we have that
\begin{equation}\label{epseq}
\frac{\card{F_j \setminus U}}{d} < \beps, \quad \abs{d-b} \leq 2\beps d,\quad   u \leq b.
\end{equation}

Let $X = \sum_{i=1}^N (T_{F_i}(\alpha,Y))\cap B\in\mathcal F(B)$. 

\begin{claim}\label{claim} $ $ 
\begin{enumerate}[(i)]
\item $\ell(T_{F_j}(\alpha,Y), T_{F_j}(\alpha,X)) -\ell(T_U(\alpha,Y),T_U(\alpha,X)) \leq\eps\frac{d}{2}$;
\item $\ell(T_U(\alpha,Y),T_U(\alpha,X))\left\vert\frac{1}d - \frac{1}b \right\vert \leq \eps$. 
\end{enumerate}
\end{claim}
\begin{proof}
(i) Since $U\subseteq F_j$,
\begin{equation}
T_{F_j}(\alpha,Y) = T_{F_j\setminus U}(\alpha,Y)+T_U(\alpha,Y)\quad\text{and}\quad T_{F_j}(\alpha,X) = T_{F_j\setminus U}(\alpha,X)+T_U(\alpha,X).
\end{equation}
Hence, Lemma~\ref{lem:ell}(c) yields 
\begin{align*}
\ell(T_{F_j}(\alpha,Y), T_{F_j}(\alpha,X))&= \ell(T_{F_j\setminus U}(\alpha,Y)+T_U(\alpha,Y), T_{F_j\setminus U}(\alpha,X)+T_U(\alpha,X))\\
& \leq \ell(T_{F_j\setminus U}(\alpha,Y), T_{F_j\setminus U}(\alpha,X)) +  \ell(T_U(\alpha,Y), T_U(\alpha,X)).
\end{align*}
Consequently, by Lemma~\ref{lem:ell2}(c) and in view of~\eqref{epseq} and \eqref{beps}, 
$$\ell(T_{F_j}(\alpha,Y), T_{F_j}(\alpha,X) -   \ell(T_U(\alpha,Y), T_U(\alpha,X)) \leq \ell(T_{F_j\setminus U}(\alpha,Y), T_{F_j\setminus U}(\alpha,X))\leq|F_j\setminus U|\ell(Y,X)\leq \beps d\ell(Y)=\eps\frac{d}{2}.$$ 


(ii) By Lemma~\ref{remtil}, $$\abs{\frac{1}{d} - \frac{1} b} \leq 2\beps  \frac{1}b.$$
Thus, by Lemma~\ref{lem:ell2}(c) and~\eqref{epseq}, 
$$\ell(T_U(\alpha,Y),T_U(\alpha,X)) \abs{\frac{1}d-\frac{1}b} \leq u \ell(Y , X) 2\beps  \frac{1}b \leq 2\beps  \ell(Y) \frac{u}{b} \leq 2\beps \ell(Y)=\eps,$$
and this concludes the proof.
\end{proof}

After further taking a subnet of $(F_j)_{j\in J}$, we may assume that, for every $j \in J$,
$$H_{alg}(\alpha_B, X)\eqeps \frac{\ell(T_{F_j}(\alpha,X))}{\card{F_j}}.$$ 

Fix $j\in J$. By~\eqref{laaast:eq} and Lemma~\ref{lem:ell}(b), we have that
\begin{equation*}\begin{split}
H_{alg}(\alpha, Y) \eqeps \frac{\ell(T_{F_j}(\alpha,Y))}{\card{F_j}} \leq \frac{\ell(T_{F_j}(\alpha,Y)+T_{F_j}(\alpha,X))}{\card{F_j}}\leq \frac{\ell(T_{F_j}(\alpha,Y),T_{F_j}(\alpha,X))}{\card{F_j}} + \frac{\ell(T_{F_j}(\alpha,X))}{\card{F_j}}
\\ \leqeps  \frac{\ell(T_{F_j}(\alpha,Y), T_{F_j}(\alpha,X))}{\card{F_j}}+H_{alg}(\alpha,X)\leq\frac{\ell(T_{F_j}(\alpha,Y),T_{F_j}(\alpha,X))}{\card{F_j}} + \ent(\alpha_B),
\end{split}\end{equation*}
and so 
$$H_{alg}(\alpha, Y)  \leq_{2\eps}\frac{\ell(T_{F_j}(\alpha,Y),T_{F_j}(\alpha,X))}{\card{F_j}}+\ent(\alpha_B).$$
In order to prove~\eqref{aim1} (and so the thesis), it remains to show that
\begin{equation}\label{aim3}
\frac{\ell(T_{F_j}(\alpha,Y),T_{F_j}(\alpha,X))}{\card{F_j}} \leq_{3 \eps} H_{alg}(\alpha_{A/B}, Z).
\end{equation}
To this end, let 
\begin{equation}\label{rdef}
r= \max_{i\in\{1,\ldots ,N\}} \frac{\ell(T_{F_i}(\alpha,Y),T_{F_i}(\alpha,X))}{\card{F_i}}.
\end{equation}
Since $1\in F_i$, for every $i\in\{1,\ldots,N\}$, one has $T_{F_i}(\alpha,Y)\subseteq \sum_{i=1}^N T_{F_i}(\alpha,Y)$. Hence, 
$$T_{F_i}(\alpha,Y)\cap B\subseteq \sum_{i=1}^N T_{F_i}(\alpha,Y) \cap B=X,$$ moreover, $X\subseteq T_{F_i}(\alpha,X)$ since $1\in F_i$; therefore,
$$T_{F_i}(\alpha,Y)\cap B\subseteq X\subseteq T_{F_i}(\alpha,X)\subseteq B,$$ and so $T_{F_i}(\alpha,Y)\cap T_{F_i}(\alpha,X)=T_{F_i}(\alpha,Y)\cap B$. 
By Lemma~\ref{carina}, this gives 
$$\ell(T_{F_i}(\alpha,Y),T_{F_i}(\alpha,X))=\ell(T_{F_i}(\alpha,Y),B).$$ 
Hence, by Equation~\eqref{laaast:eq}
$$\frac{\ell(T_{F_i}(\alpha,Y),T_{F_i}(\alpha,X))}{\card{F_i}} =\frac{\ell(T_{F_i}(\alpha,Y),B)}{\card{F_i}} \eqeps H_{alg}(\alpha_{A/B}, Z),$$
therefore, 
\begin{equation}\label{rrr}
r \eqeps H_{alg}(\alpha_{A/B}, Z).
\end{equation}
For every $i\in\{1,\ldots, N\}$, let 
\begin{equation}\label{deltadef}
\delta_i= \frac{\card{P_i}\card{F_i}}b;
\end{equation}
clearly, $0 \leq \delta_i \leq 1$ and $\sum_{i=1}^N \delta_i = 1$. Hence, by the definition of $r$ in \eqref{rdef},
\begin{equation}\label{rrrr}
\sum_{i=1}^N \delta_i \frac{\ell(T_{F_i}(\alpha,Y),T_{F_i}(\alpha,X))}{\card{F_i}} \leq r.
\end{equation}
 In the notation~\eqref{notj}, by Claim~\ref{claim}(i,ii), we have that
\begin{equation}\label{E1}
\frac{\ell(T_{F_j}(\alpha,Y),T_{F_j}(\alpha,X))}{\card{F_j}}=\frac{\ell(T_{F_j}(\alpha,Y),T_{F_j}(\alpha,X))}d \leq_\eps \frac{\ell(T_U(\alpha,Y),T_U(\alpha,X))}d\leq_\eps\frac{\ell(T_U(\alpha,Y),T_U(\alpha,X))}b.
\end{equation}
 Since $T_U(\alpha,Y)=\sum_{i=1}^N T_{P_iF_i}(\alpha,Y)$ and $T_{P_iF_i}(\alpha,Y)=T_{P_i}(\alpha,T_{F_i}(\alpha,X))$ for every $i\in\{1,\ldots,N\}$, and analogously for $T_U(\alpha,X)$, by Lemma~\ref{lem:ell3} we have that 
\begin{equation}\label{E1'}
\ell(T_U(\alpha,Y),T_U(\alpha,X))\leq \sum_{i=1}^N \card{P_i}\ell(T_{F_i}(\alpha,Y),T_{F_i}(\alpha,X))
\end{equation}
Now, since $|P_i|/b=\delta_i/|F_i$ for every $i\in\{1,\ldots, N\}$ by the definition of $\delta_i$ in \eqref{deltadef},
\begin{equation}\label{E2}
\frac{\sum_{i=1}^N \card{P_i}\ell(T_{F_i}(\alpha,Y),T_{F_i}(\alpha,X))}b=\sum_{i=1}^N\delta_i\frac{\ell(T_{F_i}(\alpha,Y),T_{F_i}(\alpha,X))}{\card{F_i}}.
\end{equation}
 By \eqref{E1}, \eqref{E1'}, \eqref{E2}, \eqref{rrrr} and \eqref{rrr} applied in this order, we conclude that
$$\frac{\ell(T_{F_j}(\alpha,Y),T_{F_j}(\alpha,X))}{\card{F_j}}\leq_\eps \sum_{i=1}^N\delta_i\frac{\ell(T_{F_i}(\alpha,Y),T_{F_i}(\alpha,X))}{\card{F_i}} \leq r \eqeps H_{alg}(\alpha_{A/B}, Z).$$
We have obtained~\eqref{aim3}, as required to conclude the proof.
\end{proof}

\section{Bridge Theorem}\label{BT-sec}

\subsection{Topological entropy for amenable semigroup actions}\label{htop-sec}

Following \cite{CCK}, let $C$ be a compact topological space, let $S$ be a cancellative left amenable semigroup. and consider the left action $S\overset{\gamma}{\curvearrowright}C$ 
by continuous maps, that is, $\gamma(s):C\to C$ is a continuous selfmap for every $s\in S$. 

\medskip
Let $\mathcal U=\{U_j\}_{j\in J}$ and $\mathcal V=\{V_k\}_{k\in K}$ be two open covers of $C$. One says that $\mathcal V$ refines $\mathcal U$, denoted by $\mathcal V\succ\mathcal U$, if for every $k\in K$ there exists $j\in J$ such that $V_k\subseteq U_j$. Moreover, 
 $$\mathcal U\vee\mathcal V=\{U_j\cap V_k: {(j,k)\in J\times K}\}.$$
 
Let also $$N(\mathcal U)=\min\{n\in\N_+\colon\mathcal U\ \text{admits a subcover of size $n$}\}.$$
We use in the sequel that
\begin{equation}\label{succ}
\text{if}\ \mathcal V\succ\mathcal U\ \text{then}\ N(\mathcal V)\geq N(\mathcal U).
\end{equation}

If $f:C\to C$ is a continuous selfmap, let $$f^{-1}(\mathcal U)=\{f^{-1}(U_j)\}_{j\in J}.$$
For an open cover $\mathcal U$ of $C$ and for every $F\in\Pf(S)$, let $$\mathcal U_{\gamma,F}=\bigvee_{s\in F}\gamma(s)^{-1}(\mathcal U).$$
Consider the function $$f_\mathcal U:\Pf(S)\to\R,\quad F\mapsto \log N(\mathcal U_{\gamma,F}).$$
For every $\mathcal U$, the function $f_\mathcal U$ is non-decreasing, subadditive, right subinvariant and uniformly bounded on singletons (see \cite{CCK}).
So by applying Theorem~\ref{CCKLemma}, we have the following definition.

\begin{definition}[See \cite{CCK}]
Let $S$ be a cancellative left amenable semigroup, $C$ a compact space, and $S\overset{\gamma}{\curvearrowright}C$ a left action. For an open cover $\mathcal U$ of $C$, the \emph{topological entropy of $\gamma$ with respect to $\mathcal U$} is $$H_{top}(\gamma,\mathcal U)=\lim_{i\in I}\frac{f_\mathcal U(F_i)}{|F_i|},$$ where $(F_i)_{i\in I}$ is a left F\o lner net of $S$. The \emph{topological entropy of $\gamma$} is $$h_{top}(\gamma)=\sup\{H_{top}(\gamma,\mathcal U)\colon \mathcal U\ \text{open cover of}\ C\}.$$
\end{definition}

\medskip
We are interested in the case when $C = K$ is a totally disconnected compact abelian group. So, we consider the topological entropy for left actions $S\overset{\gamma}{\curvearrowright}K$ by continuous endomorphisms, that is, $\gamma(s):K\to K$ is a continuous endomorphism for every $s\in S$.
In this setting we can compute the topological entropy using open subgroups instead of open covers. Indeed, for a totally disconnected compact group $K$, let $\mathcal B(K)$ be the family of all open subgroups of $K$. In particular, each $U\in\mathcal B(K)$ has finite index in $K$.

For every $U\in\mathcal B(K)$, let $$\zeta(U)=\{k+U\colon k\in K\}.$$ Clearly,
\begin{equation}\label{NzU}
N(\zeta(U))=[K:U].
\end{equation}

Since $K$ is a totally disconnected compact group, $\mathcal B(K)$ is a local base of $K$ by van Dantzig's theorem, so every open cover of $K$ is refined by some $\zeta(U)$ with $U\in\mathcal B(K)$. Hence, by~\eqref{succ},
$$h_{top}(\gamma)=\sup\{H_{top}(\gamma,\zeta(U))\colon U\in\mathcal B(K)\}.$$

Define, for every $F\in\Pf(S)$, the \emph{$\gamma$-cotrajectory of $U$ with respect to $F$} by 
$$C_F(\gamma,U)=\bigcap_{s\in S}\gamma(s)^{-1}(U).$$
In particular, $C_F(\gamma,U)\in\mathcal B(K)$, so each $C_F(\gamma,U)$ has finite index in $K$.

\begin{lemma}\label{V=K/C}
Let $S$ be a cancellative left amenable semigroup, $K$ a totally disconnected compact abelian group, and $S\overset{\gamma}{\curvearrowright}K$ a left action. For every $U\in\mathcal B(K)$ and every $F\in\Pf(S)$,
$$\zeta(U)_{\gamma,F}=\zeta(C_F(\gamma,U)).$$
\end{lemma} 
\begin{proof} 
Recall that $\zeta(U)_{\gamma,F}=\bigvee_{s\in F}\gamma(s)^{-1}(\zeta(U))$.

Let $s\in F$ and $k\in K$. Then
\begin{equation}\label{1}
\gamma(s)^{-1}(k+U)=k'+\gamma(s)^{-1}(U)\quad \text{for every}\ k'\in \gamma(s)^{-1}(k+U).
\end{equation}
Hence,
$$\zeta(U)_{\gamma,F}=\left\{\bigcap_{s\in F}\gamma(s)^{-1}(k_s+U)\colon k_s\in K\right\}=\left\{\bigcap_{s\in F}(k_s'+\gamma(s)^{-1}(U))\colon k_s'\in K\right\}.$$

Now, for $k_s\in K$, with $s\in F$,
\begin{equation}\label{2}
\bigcap_{s\in F}(k_s+\gamma(s)^{-1}(U))=z+\bigcap_{s\in F}\gamma(s)^{-1}(U)\quad\text{for every}\ z\in \bigcap_{s\in F}(k_s+\gamma(s)^{-1}(U)).
\end{equation}
Therefore, 
$$\zeta(U)_{\gamma,F}=\left\{\bigcap_{s\in F}(k_s'+\gamma(s)^{-1}(U))\colon k_s'\in K\right\}=\left\{z+\bigcap_{s\in F}\gamma(s)^{-1}(U)\colon z\in K\right\}=\zeta(C_F(\gamma,U)).$$
This concludes the proof.
\end{proof}

\begin{proposition}\label{Htop}
Let $S$ be a cancellative left amenable semigroup, $K$ a totally disconnected compact group, $S\overset{\gamma}{\curvearrowright}K$ a left action, and $(F_i)_{i\in I}$ a left F\o lner net. If $U\in\mathcal B(K)$, then
$$H_{top}(\gamma,\zeta(U))=\lim_{i\in I}\frac{\log[K:C_{F_i}(\gamma,U)]}{|F_i|}.$$
\end{proposition}
\begin{proof}
By definition, by Lemma~\ref{V=K/C} and Equation~\eqref{NzU},
$$H_{top}(\gamma,\zeta(U))=\lim_{i\in I}\frac{\log N(\zeta(U)_{\gamma,F_i})}{|F_i|}=\lim_{i\in I}\frac{\log N(\zeta(C_{F_i}(\gamma,U))}{|F_i|}=\lim_{i\in I}\frac{[K:C_{F_i}(\gamma,U)]}{|F_i|},$$
hence the thesis holds.
\end{proof}

From now on we write simply $H_{top}(\gamma,U)$ in place of $H_{top}(\gamma,\zeta(U))$.

\subsection{The entropy of the dual action}\label{sec:dual}

Let $A$ be a locally compact abelian group and denote by $\widehat A$ its Pontryagin dual. For a continuous homomorphism $\phi:A\to B$, where $B$ is another locally compact abelian group, let $\widehat \phi:\widehat B\to \widehat A$ be the dual of $\phi$, defined by $\widehat\phi(\chi)=\chi\circ\phi$.

\medskip
Let $S$ be a cancellative left amenable semigroup and $K$ a compact abelian group, and consider the left action $S\overset{\gamma}{\curvearrowright} K$.
Then $\gamma$ induces the right action $\widehat K\overset{\widehat\gamma}{\curvearrowleft}S$, defined by
$$\widehat\gamma(s)=\widehat{\gamma(s)}:\widehat K\to \widehat K\quad \text{for every}\ s\in S.$$
In fact, fixed $s,t\in S$, since $\gamma(st)=\gamma(s)\gamma(t)$, we have that
$$\widehat\gamma(st)=\widehat{\gamma(st)}=\widehat{\gamma(s)\gamma(t)}=\widehat{\gamma(t)}\widehat{\gamma(s)}=\widehat\gamma(t)\widehat\gamma(s).$$

Analogously, let $S$ be a cancellative left amenable semigroup and $A$ an abelian group, and consider the right action $A\overset{\alpha}{\curvearrowleft} S$.
Then $\alpha$ induces the left action $S\overset{\widehat\alpha}{\curvearrowright}\widehat A$, defined by
$$\widehat\alpha(s)=\widehat{\alpha(s)}:\widehat A\to \widehat A\quad \text{for every}\ s\in S.$$
In fact, fixed $s,t\in S$, since $\alpha(st)=\alpha(t)\alpha(s)$, we have 
$$\widehat\alpha(st)=\widehat{\alpha(st)}=\widehat{\alpha(t)\alpha(s)}=\widehat{\alpha(s)}\widehat{\alpha(t)}=\widehat\alpha(s)\widehat\alpha(t).$$

\smallskip
According to Pontryagin--van Kampen duality theorem $A\cong_{top} \widehat{\widehat A}$, so in the sequel we shall simply 
identify $\widehat{\widehat A}$ with $A$. As a direct consequence one obtains: 

\begin{proposition}\label{ww}
Let $S$ be a cancellative left amenable semigroup and $K$ a compact abelian group, and consider the left action $S\overset{\gamma}{\curvearrowright} K$. Then $$\widehat{\widehat\gamma}=\gamma.$$
Let $A\overset{\alpha}{\curvearrowleft} S$ be a 
right action of $S$ on an abelian group $A$. Then $$\widehat{\widehat\alpha}=\alpha.$$
\end{proposition}

This shows in particular that every left action $S\overset{\gamma}{\curvearrowright} K$ of a cancellative left amenable semigroup $S$ on a compact abelian group $K$ is induced by a right action $A\overset{\alpha}{\curvearrowleft} S$ of $S$ on an abelian group $A$, and vice versa.

\medskip
We collect here some known facts concerning Pontryagin duality. 
Recall that, if $A$ is a locally compact abelian group, and $B$ is a subgroup of $A$, then the \emph{annihilator} of $B$ in $\widehat A$ is $B^\perp=\{\chi\in\widehat A\colon \chi(B)=0\}$. 
Under the identification of $\widehat{\widehat A}$ with $A$ we have that, for every closed subgroup $B$ of $A$,
\begin{equation}\label{perpperp}
(B^\perp)^\perp=B,
\end{equation}
and moreover, 
\begin{equation}\label{pontriso}
\widehat B\cong_{top} \widehat A/B^\perp\quad \text{and}\quad \widehat{A/B}\cong_{top}B^\perp.
\end{equation}

\begin{lemma}\label{pontr}
Let $A$ be an abelian group.
\begin{enumerate}[(a)]
\item If $A$ is discrete (respectively, compact) then $\widehat A$ is compact (respectively, discrete).
\item If $A$ is discrete, then $\widehat A$ is totally disconnected precisely when $A$  is torsion. 
\item If $A$ is finite, then $A\cong \widehat A$.
\item If $B_1, B_2$ are subgroups of $A$, then $(B_1+B_2)^\perp=B_1^\perp\cap B_2^\perp$.
\item If $\phi:A\to A$ is an endomorphism, then $\phi (B)^\perp=(\widehat\phi)^{-1}(B^\perp)$.
\end{enumerate}
\end{lemma}

The following technical lemma is a key step in the proof of Theorem~\ref{BTalg}.

\begin{lemma}\label{CT}
Let $S$ be a cancellative left amenable semigroup, $A$ a torsion abelian group, and $A\overset{\alpha}{\curvearrowleft} S$ a right action.
For $B\in\mathcal F(A)$ and $F\in\Pf(S)$, 
$$|T_F(\alpha,B)|=[\widehat A:C_F(\widehat\alpha,B^\perp)].$$
\end{lemma}
\begin{proof}
Recall that $T_F(\alpha,B)=\sum_{s\in F}\alpha(s)(B)$ is a finite subgroup of $A$, so by Lemma~\ref{pontr}(c) and by~\eqref{pontriso} it is isomorphic to its dual $\widehat{T_F(\alpha,B)}\cong K/T_F(\alpha,B)^\perp$. In view of Lemma~\ref{pontr}(d,e), 
$$T_F(\alpha,B)^\perp\cong \bigcap_{s\in F}(\alpha(s)(B)^\perp)=\bigcap_{s\in F}\widehat{\alpha(s)}^{-1}(B^\perp)=\bigcap_{s\in F}\widehat\alpha(s)^{-1}(B^\perp)=C_F(\widehat\alpha,B^\perp).$$
Therefore, $|T_F(\alpha,B)|=|\widehat{T_F(\alpha,B)}|=[\widehat A:T_F(\alpha,B)^\perp]=[\widehat A:C_F(\widehat\alpha,B^\perp)]$.
\end{proof}

We are now in position to prove the so-called Bridge Theorem.

\begin{theorem}\label{BTalg}
Let $S$ be a cancellative left amenable semigroup, $A$ a torsion abelian group, and $A\overset{\alpha}{\curvearrowleft} S$ a right action. Then 
$$h_{alg}^r(\alpha)=h_{top}(\widehat\alpha).$$
\end{theorem}
\begin{proof}
Let $B\in\mathcal F(A)$; by Lemma~\ref{pontr}(c,e), $B^\perp$ is a closed subgroup of $\widehat A$ with finite index, so it is open and $B^\perp\in\mathcal B(\widehat K)$.
Then, for $(F_i)_{i\in I}$ a left F\o lner net of $S$, since $B^\perp\in\mathcal B(\widehat A)$, Proposition~\ref{Htop} and Lemma~\ref{CT} give 
$$H^r_{alg}(\alpha,B)=\lim_{i\in I}\frac{\log|T_{F_i}(\alpha,B)|}{|F_i|}=\lim_{i\in I}\frac{\log[\widehat A:C_{F_i}(\widehat\alpha,B^\perp)]}{|F_i|}=H_{top}(\widehat\alpha,B^\perp).$$

By~\eqref{perpperp} there is a bijection $\mathcal F(A)\to \mathcal B(\widehat A)$, given by $B\mapsto B^\perp$, so we can conclude that $h^r_{alg}(\alpha)=h_{top}(\widehat\alpha)$ in view of Proposition~\ref{Htop} and~\eqref{h=ent}.
\end{proof}

The following is a consequence of Theorem~\ref{BTalg} and Proposition~\ref{ww}.

\begin{corollary}\label{BTcor}
Let $S$ be a cancellative left amenable semigroup, $K$ a totally disconnected compact abelian group and $S\overset{\gamma}{\curvearrowright}K$. Then $$h_{top}(\gamma)=h^r_{alg}(\widehat\gamma).$$
\end{corollary}
\begin{proof}
Let $\alpha=\widehat\gamma$. By Proposition~\ref{ww} we have that $\gamma=\widehat \alpha$. So, Theorem~\ref{BTalg} implies $h_{top}(\gamma)=h_{top}(\widehat\alpha)=h_{alg}^r(\alpha)$.
\end{proof}

As a consequence of the Addition Theorem and the Bridge Theorem proved in this paper for the algebraic entropy, we obtain the following Addition Theorem for the topological entropy.

\begin{theorem}[Addition Theorem]
Let $S$ be a cancellative left amenable semigroup, $K$ a totally disconnected compact abelian group, $S\overset{\gamma}\curvearrowright K$ and $L$ a closed $\gamma$-invariant subgroup of $K$. Then $$h_{top}(\gamma)=h_{top}(\gamma_L)+h_{top}(\gamma_{K/L}).$$
\end{theorem}
\begin{proof}
Let $A=\widehat K$, $\alpha=\widehat \gamma$ and $B=L^\perp\leq A$.
By Corollary~\ref{BTcor}, 
$$h_{top}(\gamma)=h_{alg}^r(\alpha),\quad h_{top}(\gamma_L)=h_{alg}^r(\widehat{\gamma_L}), \quad h_{top}(\gamma_{K/L})=h_{alg}^r(\widehat{\gamma_{K/L}}).$$
Since $\widehat{\gamma_L}$ is conjugated to $\alpha_{A/B}$ and $\widehat{\gamma_{K/L}}$ is conjugated to $\alpha_B$, Proposition~\ref{conju} gives 
$$h_{alg}^r(\widehat{\gamma_L})=h_{alg}^r(\alpha_{A/B}),\quad h_{alg}^r(\widehat{\gamma_{K/L}})=h_{alg}^r(\alpha_B).$$
So it suffices to apply the Addition Theorem and the previous equalities to get
$$h_{top}(\gamma)=h_{alg}^r(\alpha)=h_{alg}^r(\alpha_B)+h_{alg}^r(\alpha_{A/B})=h_{top}(\gamma_{K/L})+h_{top}(\gamma_L),$$
which concludes the proof.
\end{proof}

\section{Final comments and open questions}

 In this final section we collect several open questions related to the results obtained in the paper.

\medskip
We start from the following question related to Example~\ref{ex:non-cancellative}.

\begin{question}
Let $S$ be a cancellative monoid, $C$ a monoid, and $\pi: S \to C$ a surjective homomorphism admitting a good section $\sigma$.
Is $C$ necessarily cancellative?
\end{question}

It is known that if $f : S \to Q$ is a surjective semigroup homomorphism and $S$ is left amenable, then $Q$ is left amenable as well
(see \cite{Day3} or \cite[Lemma 3]{Don}).  On the other hand, the following question is open.

\begin{question}\label{quest:new0}
If $f:S\to S_1$ is a surjective homomorphism of semigroups (groups) and if $(F_i)_{i\in I}$ is a right F\o lner net of $S$, is then $(f(F_i))_{i\in I}$ a right F\o lner net of $S_1$?
\end{question}

In view of Proposition~\ref{restr:act} we propose the following conjecture. 

\begin{conjecture}\label{Conj1}
Let $G$ be an amenable group, $A$ an abelian group, and $G\overset{\alpha}{\curvearrowright}A$ a left action. If $H$ is a subgroup of $G$, then 
$h_{alg}(\alpha)\leq h_{alg}(\alpha\restriction_H)$ and $\ent(\alpha)\leq \ent(\alpha\restriction_H)$.
\end{conjecture}

We also conjecture that one can remove the condition in Theorem~\ref{teo:submonoid} that the subgroup of $G$ of infinite index is normal in $G$:

\begin{conjecture}\label{Conj2}
Let $G$ be an amenable group, $A$ an abelian group, $G\overset{\alpha}{\curvearrowright}A$ a left action, and $H$ a non-trivial subgroup of $G$ of infinite index. 
\begin{itemize}
\item[(a)] If $\ent(\alpha \rest_H)<\infty$, then $\ent(\alpha)=0$.
\item[(b)] If $\halg(\alpha \rest_H)<\infty$, then $\halg(\alpha)=0.$
\end{itemize}
\end{conjecture}

If Conjecture~\ref{Conj2} holds true, then it implies that also Conjecture~\ref{Conj1} holds true in view of Proposition~\ref{finind}.

\medskip
We end with two general conjectures related to the Addition Theorem and the Bridge Theorem (see Theorem~\ref{ATintro}  and Theorem~\ref{BTintro} respectively). Indeed, we think that they hold without the hypothesis on the abelian group $A$ to be torsion.

\begin{conjecture}[Addition Theorem]
Let $S \overset{\alpha}\curvearrowright A$ be a left action of a cancellative right amenable monoid $S$ on an abelian group $A$, and let $B$ be an $\alpha$-invariant subgroup of $A$. Then
\[h_{alg}(\alpha)=h_{alg}(\alpha_B)+h_{alg}(\alpha_{A/B}).\]
\end{conjecture}

\begin{conjecture}[Bridge Theorem]
Let $S \overset{\gamma}\curvearrowright K$ be a left action of a cancellative left amenable monoid on a compact abelian group $K$.
Then $$h_{top}(\gamma)=h_{alg}^r(\widehat\gamma)$$
\end{conjecture}

In case $S$ is a group, a positive answer was announced by Virili \cite{V2} in the more general case of actions on locally compact abelian groups.

\end{document}